\documentclass[a4paper]{amsart}
\usepackage{amsmath,amsthm,amssymb}
\usepackage[T1]{fontenc}
\usepackage{mathrsfs}
\usepackage{url}
\usepackage{hyperref}
\usepackage{cleveref}
\usepackage{graphicx}
\newtheorem{tw}{Theorem}[section]
\newtheorem{lm}[tw]{Lemma}
\newtheorem{wn}[tw]{Corollary}
\newtheorem{pr}[tw]{Proposition}
\theoremstyle{definition}

\newtheorem{uw}[tw]{Remark}
\newtheorem*{problem}{Problem}

\numberwithin{equation}{section}

\newcommand{\R}{\mathbb{R}}
\newcommand{\Z}{\mathbb{Z}}
\newcommand{\N}{\mathbb{N}}
\newcommand{\T}{\mathbb{T}}

\newcommand{\cS}{\mathcal{S}}

\newcommand{\Q}{\mathbb{Q}}
\newcommand{\cB}{\mathcal{B}}
\newcommand{\cC}{\mathcal{C}}

\newcommand{\cT}{\mathcal{T}}

\newcommand{\1}{\mathbbm{1}}
\newcommand{\xbm}{(X,\mathcal{ B},\mu)}

\newcommand{\vep}{\varepsilon}
\newcommand{\ct}{\mathcal{T}}

\begin{document}
\date{\today}
\title[Non-reversibility  for ergodic flows]{Non-reversibility and self-joinings of higher orders for ergodic flows}
\author[K. Fr\k{a}czek \and J. Ku\l aga \and M. Lema\'nczyk]{Krzysztof Fr\k{a}czek \and
Joanna Ku\l aga \and Mariusz Lema\'nczyk}
\address{Faculty of Mathematics and Computer Science, Nicolaus
Copernicus University, ul. Chopina 12/18, 87-100 Toru\'n, Poland}
\email{fraczek@mat.umk.pl, joanna.kulaga@gmail.com,
mlem@mat.umk.pl} 

\maketitle

\begin{abstract}
By studying the weak closure of multidimensional off-diagonal
self-joinings we provide a criterion for non-isomorphism of a flow
with its inverse, hence the non-reversibility of a flow. This is
applied to special flows over rigid automorphisms. In particular,
we apply the criterion to special flows over irrational rotations,
providing a large class of non-reversible flows, including some
analytic reparametrizations of linear flows on $\T^2$, so called
von Neumann's flows and some special flows with piecewise
polynomial roof functions. A topological counterpart is also
developed with the full solution of the problem of the topological
self-similarity of continuous special flows over irrational
rotations. This yields examples of continuous special flows over
irrational rotations without topological self-similarities and
having all non-zero real numbers as scales of measure-theoretic
self-similarities.
\end{abstract}

\tableofcontents
\section{Introduction}
Given a (measurable)  measure-preserving flow $\cT=(T_t)_{t\in\R}$
on a probability standard Borel space $\xbm$ one says that it is
\emph{reversible} if $\cT$ is isomorphic to its inverse with a
conjugating automorphism $S:\xbm\to\xbm$ being an
involution\footnote{It should be noticed that, in general, even if
(\ref{inwol}) and~(\ref{inwol1}) are satisfied for some $S$ then
we can find $S'$ which is not an involution but
satisfies~(\ref{inwol}) \cite{Go-Ju-Le-Ru}. For example, take
$T(x,y)=(x+\alpha,x+y)$ on $\T^2$. Then
$T^{-1}(x,y)=(x-\alpha,-(x-\alpha)+y)$ and $S(x,y)=(-x,x+y)$
settles an isomorphism of $T$ and its  inverse. Of course
$S^2=Id$. On the other hand if we set
$\sigma_{\gamma}(x,y)=(x,y+\gamma)$ then $\sigma_\gamma
T=T\sigma_{\gamma}$  and  $\sigma_\gamma S=S\sigma_\gamma$.  Hence
$(S\sigma_\gamma)T=T^{-1}(S\sigma_\gamma)$. But
$(S\sigma_\gamma)^n=S^{n\; ({\rm mod}\; 2)}\sigma_{n\gamma}$, so
we obtain a conjugation which is of infinite order (if $\gamma$ is
irrational).

Another example can be given by taking first a weakly mixing flow
$(S_t)$  and then considering $T_t=S_t\times S_{-t}$ in which
$(x,y)\mapsto (y,x)$ yields reversibility of $\cT$. On the other
hand, $W(x,y)=(S_1y,x)$ also settles an isomorphism of $\cT$ and
its inverse and since $W^2=S_1\times S_1$, $W$ is even weakly
mixing.}, i.e.:
\begin{equation}\label{inwol}
T_t\circ S=S\circ T_{-t}\;\;\text{for each $t\in\R$}
\end{equation}
 and
\begin{equation}\label{inwol1}S^2=Id.\end{equation}
As far as we know, in ergodic theory, this problem was not
systematically studied for flows. In case of automorphisms first
steps were taken up in \cite{Go-Ju-Le-Ru}. In that paper it has
been shown that for an arbitrary automorphism $T$ with simple
spectrum all isomorphisms (if there is any) between $T$ and
$T^{-1}$ must be involutions. The same result holds for flows: a
simple  spectrum flow isomorphic to its inverse is reversible, in
fact,~(\ref{inwol})
implies~(\ref{inwol1})\footnote{\label{trivcent} The proof from
\cite{Go-Ju-Le-Ru} goes through for flows.

One more natural case when isomorphism of $\cT$ and its inverse
implies reversibility arises if we  assume that the centralizer
$C((T_t))$ is trivial, i.e.\ equal to $\{T_t:\:t\in\R\}$ and the
$\R$-action $t\mapsto T_t$ is free. Indeed, as in
\cite{Go-Ju-Le-Ru}, we notice that whenever $S$
satisfies~(\ref{inwol}) then $S^2$ belongs to $C(\cT)$, so
$S^2=T_{t_0}$. Now, clearly $T_{t_0}S=ST_{t_0}$ and since
$T_{t_0}S=ST_{-t_0}$, we  have $T_{-t_0}=T_{t_0}$ and hence
$t_0=0$ by the freeness assumption.}. Another class of flows in
which~(\ref{inwol}) puts some restrictions on the order of $S$ is
the class of flows having so called  weak closure property: each
element $R$ of the centralizer $C(\cT)$ is a weak limit of
time-$t$ automorphisms, i.e.\ $R=\lim_{k\to\infty} T_{t_k}$  for
some $t_k\to\infty$, namely we must have $S^4=Id$~\footnote{We
borrow the argument from \cite{Go-Ju-Le-Ru}: $C(\cT)\ni
S^2=\lim_{k\to\infty} T_{t_k}$ and since $T_{t_k}S=ST_{-t_k}$, by
passing to the limit, $S^3=S^{-1}$.}. Moreover, if $S^2\neq Id$
then $\cT$ is not reversible\footnote{Again, borrowing the
argument from \cite{Go-Ju-Le-Ru}, suppose that $S'$
satisfies~(\ref{inwol}). Then $SS'\in C(\cT)$, so
$SS'=\lim_{k\to\infty} T_{t_k}$.  Since $T_{t_k}S=ST_{-t_k}$,
$(SS')S=S(SS')^{-1}$, whence $(S')^{2}=S^{-2}$, but $S^4=Id$, so
$S^2=(S')^2$.

Note that it follows that if $\cT$ satisfies the weak  closure
theorem, is isomorphic to its inverse and is not reversible then
it has a 2-point fiber factor, namely $\{B\in{\mathcal B}:\:
S^2B=B\}$, which is reversible.}.

It is easy to observe that isomorphisms between $T$ and $T^{-1}$
lift to isomorphisms of the corresponding suspension flow  (see
Section~\ref{specialflow} for a definition) and its inverse,
moreover, as observed e.g.\ in \cite{Da-Le}, each isomorphism
between the suspension flow and its inverse must come from an
isomorphism of $T$ and $T^{-1}$. In \cite{Go-Ju-Le-Ru} there is a
construction of an automorphism $T$ satisfying the weak closure
property, isomorphic to its inverse and such that all conjugations
$S$ between $T$ and $T^{-1}$ have order four. By taking the
suspension flow over this example  we obtain an ergodic flow
having the weak closure property, being isomorphic to its inverse
and such that all conjugations satisfying~(\ref{inwol}) are of
order four, so this flow is  not reversible.

The problem of reversibility is closely related to  the
self-similarity problem (see \cite{Da-Ry}, \cite{Fr-Le.selfs}).
Recall  that $s\in \R^*$ is a scale of self-similarity for a
measure-preserving flow $\cT=(T_t)_{t\in\R}$ if $\cT$ is
isomorphic to the flow $\cT\circ s:=(T_{st})_{t\in\R}$. The
multiplicative subgroup of all scales of self-similarity we will
denote by $I(\ct)\subset \R^*$. The flow $\ct$ is called
self-similar if $I(\ct)\nsubseteq\{-1,1\}$. Of course, if $\ct$ is
reversible then $-1\in I(\ct)$.

One of possibilities to show the absence of  self-similarities for
a non-rigid flow is to show that in the weak closure of its
$2$-off-diagonal self-joinings  there is an integral of
off-diagonal joinings \cite{Fr-Le.selfs}. However, it is rather
easy to see that on the level of 2-self-joinings we cannot
distinguish between an action and its inverse. This is Ryzhikov's
paper \cite{Ry1} which was historically the first to show that a
certain asymmetry between an automorphism and its inverse can be
detected on the level of $3$-self-joinings by studying the weak
closure of $3$-off-diagonal self-joinings (see also \cite{Da-Ry}).
By taking the suspension of Ryzhikov's automorphism we obtain a
flow non-isomorphic to its inverse. One of the purposes of the
paper is to generalize this approach and present potential
asymmetries in the weak closure of higher dimensional off-diagonal
self-joinings when we change time in the suspension over a rigid
automorphism, see Proposition~\ref{main1}.

In Section~\ref{sec:joinings} we extend techniques introduced in \cite{Fr-Le} for $2$-joinings
to the class of higher order joinings, see Proposition~\ref{main1}.
Recall that $2$-joining approach was used fruitfully in proving the absence
of self-similarity for some classes of special flows over irrational rotations
on the circle, including so called von Neumann flows, see  \cite{Fr-Le.selfs}.
However, for proving non-isomorphism of the flow and its inverse this method breaks down.
In this case, as in \cite{Ry1}, we will apply $3$-joinings to distinguish between
the flow and its inverse, see Proposition~\ref{kryt}.
In Section~\ref{sec:vonn}, using Proposition~\ref{kryt}, we prove that
any von Neumann flow is not isomorphic to its inverse for almost every
rotation in the base.

In Section~\ref{polyroof}, the approach developed in
Section~\ref{sec:joinings} is applied to special flows $T^f$ built
over irrational rotations $Tx=x+\alpha$ on the circle and under
$C^{r-1}$-roof functions ($r$ is an odd natural number) which are
polynomials of degree $r$ on two complementary intervals
$[0,\beta)$ and $[\beta,1)$ ($0<\beta<1$). Using $r+1$--joinings
we prove that for a.e.\ $\beta$ the flow $T^f$ is not isomorphic
to its inverse, whenever $\alpha$ satisfies a Diophantine type
condition (along a subsequence, see \eqref{eq:diof}).

In Section~\ref{sec:analytic} the 3-joining approach turns out to
be sufficient to construct an analytic area-preserving flow on the
two torus which is not isomorphic to its inverse.  In other words,
we show that we can change time in an ergodic linear  flow (which
is always reversible) in an analytic way so that the resulting
flow is weakly mixing and not reversible.   We use  the AACCP
method introduced in \cite{Kw-Le-Ru}.  Additionally, slightly
modifying the construction, we prove that the resulting flow has
no rational self-similarities. In fact, we obtain disjointness (in
the Furstenberg sense) of any two different rational time
automorphisms. This kind of investigations is partly motivated by
Sarnak's conjecture on orthogonality of deterministic sequences
from M\"obius function through disjointness: see \cite{Bo-Sa-Zi}.

In Section~\ref{sec:chacon}, we come back  to automorphisms, and
as in \cite{Ry1}, we show that the $3$-joining method can be
applied  to a class of rank one automorphisms having a subsequence
of towers of Chacon's type.  We show that they are not reversible.

In Section~\ref{sec:topol}  we  deal with topological self-similarities of
continuous time changes of minimal linear flows on the two torus.
Each such flow is topologically conjugate to the special flow $T^f$ build over
an irrational rotation $Tx=x+\alpha$ on the circle and under a continuous
roof function $f:\T\to\R_+$.
We show that if $T^f$ is topologically self-similar then $\alpha$ is
a quadratic irrational and $f$ is topologically cohomological to a constant
function. It follows that if a continuous time change of a minimal linear
flow on the two torus is topologically self-similar then it is topologically
conjugate to a minimal linear flows as well. As a byproduct we obtain
an example of a continuous flow on the torus which has no topological self-similarities
and the group scales of self-similarity (as a measure-preserving system) is equal to $\R^*$.

First historical examples of automorphisms non-isomorphic to their
inverses were provided by Anzai \cite{An}, Malkin \cite{Ma} and
Oseledets \cite{Os}. Moreover, the property of being isomorphic to
its inverse (the more, reversibility) is not a typical property.
As shown by del Junco \cite{Ju}  (for automorphisms) and by
Danilenko and Ryzhikov \cite{Da-Ry} (for flows) typical flow is
disjoint from its inverse. But there are quite a few natural
examples of flows which are reversible. Let us go through a
selection of known examples.

\vspace{1ex}

\noindent\textbf{A) All ergodic  flows with discrete spectrum are
reversible.} This follows easily from the Halmos-von Neumann
theorem,  see e.g.\ \cite{Co-Fo-Si} (the fact that each
isomorphism must be an involution is a consequence of the
simplicity of the spectrum of such flows).

\vspace{1ex}

\noindent\textbf{B) All Gaussian flows are reversible.} Indeed,
each Gaussian flow is determined by a one-parameter unitary group
${\mathcal U}=(U_t)_{t\in\R}$ acting on a separable Hilbert space
$H$ such that there is a spectral decomposition
\begin{equation}\label{de}H=\bigoplus_{n=1}^\infty\R(x_n)\;\;\text{with}\;\;\sigma_{x_1}
\gg\sigma_{x_2}\gg\ldots\;\;\text{and}\;\;\sigma_{x_n}(A)=\sigma_{x_n}(-A)\end{equation}
for each Borel subset $A\subset\R$ and $n\geq1$ (and
$\sigma_{x_1}$  is assumed to be continuous), see \cite{Ka-Th},
\cite{Le}, \cite{Le-Pa-Th}. Now, the action  ${\mathcal U}$ on
$\R(x_n)$ is isomorphic to the action ${\mathcal V}^{(n)}$:
$$
V^{(n)}_t(f)(x)=e^{2\pi itx}f(x)\;\;\text{for}\;\;f\in
L^2(\R,\sigma_{x_n}),$$ so $I_nf(x)=f(-x)$ is an involution which
settles an isomorphism of ${\mathcal V}^{(n)}$ and its inverse.
Then, up to isomorphism, $I=\bigoplus_{n=1}^\infty I^{(n)}$ is an
involution which settles an isomorphism of ${\mathcal U}$ and its
inverse\footnote{Notice that the same argument works for an
arbitrary Koopman representation $U_t=U_{T_t}$. In other words, an
arbitrary Koopman representation is unitarily reversible.} and it
extends to a measure-preserving isomorphism of the corresponding
Gaussian flow $(T_t)$ and its inverse, see e.g.\  \cite{Le-Pa-Th}.

\vspace{1ex}

\noindent\textbf{C) Some horocycle flows are reversible.}
Let $\Gamma\subset PSL_2(\R)$ be a discrete subgroup  with finite
covolume. Then the homogeneous space $X=\Gamma\backslash
PSL_2(\R)$ is the unit tangent bundle of a surface $M$ of constant
negative curvature. Let us consider the corresponding  \emph{horocycle flow}
$(h_t)_{t\in\R}$ and \emph{geodesic flow} $(g_s)_{s\in\R}$
on $X$. Since
\begin{equation}\label{horkom}
g_sh_tg_s^{-1}=h_{e^{-2s}t}\text{ for all }s,t\in\R,
\end{equation}
the  flows $(h_t)_{t\in\R}$ and $(h_{e^{-2s}t})_{t\in\R}$ are measure-theoretic
isomorphic for each $s\in\R$, so
all positive numbers are self-similarity scales for a horocycle flow.

We will now show that some horocycle flows are reversible.
Let now $J$ denote the matrix
\[
J=\begin{pmatrix}1&0\\0&-1\end{pmatrix}.
\]
Clearly, $J\notin SL_2(\R)$. However, if $\Gamma$ satisfies $\Gamma=
J^{-1}\Gamma J$ then $J$ will also act on $\Gamma\backslash
PSL_2(\R)$:
\[
J(\Gamma x):=J^{-1} \Gamma x J=(J^{-1}\Gamma J)J^{-1}xJ= \Gamma
J^{-1}x J.
\] Moreover,
\[
J^{-1} h_t J=h_{-t}
\]
and since $J$ yields an order two map, we obtain that in this case
the horocycle flow is reversible. It follows that
$I((h_t)_{t\in\R})=\R^*$.
\begin{wn}\label{cor:horrev}
In the modular case $\Gamma:=PSL_2(\Z)\subset PSL_2(\R)$, the
horocycle flow $(h_t)_{t\in\R}$ is reversible.
\end{wn}
There are even cocompact lattices $\Gamma$ which are not
``compatible'' with the matrix $J$. In this case a deep theory of
Ratner \cite{Ra} implies that in particular $(h_t)_{t\in\R}$ is not
measure-theoretically isomorphic to its inverse.

Let us come back to the horocycle flow $(h_t)_{t\in\R}$  on the
modular space $\Gamma\backslash PSL_2(\R)$, $\Gamma=PSL_2(\Z)$. By
Corollary~\ref{cor:horrev}, this flow is reversible. Moreover,
$C((h_t)_{t\in\R})=\{h_t:\:t\in\R\}$. Indeed, first note that
\[\{\alpha\in PSL_2(\R):\alpha\Gamma\alpha^{-1}=\Gamma\}=\Gamma.\]
In view of the celebrated Ratner's Rigidity Theorem (see
Corollary~2 in \cite{Ra-rig}), it follows that $C((h_t)_{t\in\R})$
is indeed  trivial\footnote{The general result of Ratner states
that elements of the centralizer are the composition of $h_{t_0}$
with the automorphism given by $\alpha$ as above.}. Hence, we
obtain the following  more precise version of
Corollary~\ref{cor:horrev} (cf.\ footnote~\ref{trivcent}).

\begin{wn}\label{cor:horrev1}
In the modular case  $\Gamma:=PSL_2(\Z)\subset PSL_2(\R)$ we have
$C((h_t)_{t\in\R})$ $=\{h_t:\:t\in\R\}$. Then, each $S$
establishing isomorphism of $(h_t)_{t\in\R}$ with its inverse is
an involution. Moreover, $S=h_{t_0}\circ J$ for some $t_0\in\R$.
\end{wn}

\vspace{1ex}

\noindent\textbf{D) All Bernoulli flows are reversible.} This is
done in two steps. If the entropy is infinite then  (via
Ornstein's isomorphisms theorem \cite{Or}) we have a Gaussian
realization of such a flow and we use \textbf{B)}. If the entropy
is finite then (again via \cite{Or}) we can consider the geodesic
flow on $\Gamma\backslash PSL_2(\R)$. Then
\[
K^{-1}g_tK=g_{-t}\text{ for all }t\in\R,\quad \text{ where }\quad K=\begin{pmatrix}0&1\\-1&0\end{pmatrix}.
\]
This establishes an isomorphism between $(g_t)_{t\in\R}$ and
$(g_{-t})_{t\in\R}$ via an involution ($K^2=Id$ as an element of
$PSL_2(\R)$)~\footnote{An alternative  proof of reversibility of
Bernoulli  was pointed to us by J.-P.\ Thouvenot.  Indeed,
consider the shift $T\colon \{0,1\}^{\Z}\to \{0,1\}^{\Z}$ given by
$T((x_n)_{n\in\Z})=(x_{n+1})_{n\in\Z}$, where $\{0,1\}^{\Z}$ is
equipped with the product measure $\mu=P^{\otimes \Z}$ with
$P(\{0\})=\mu(\{1\})=1/2$. Then the map  $I:(x_n)_{n\in\Z}\mapsto
(x_{-n})_{n\in\Z}$ is an involution conjugating $T$ with $T^{-1}$.
Moreover, there is a roof function $f$ constant on each of the
cylinder sets $\{(x_n)_{n\in\Z}\colon x_0=0\}$,
$\{(x_n)_{n\in\Z}\colon x_0=1\}$ such that the special flow $T^f$
is Bernoulli~\cite{MR0272985}. Now, it suffices to apply
Remark~\ref{uwagarev2} below to conclude that $T^f$ (as well as
$T^f\circ s$ for each $s\in\R\setminus\{0\}$) is reversible. For
the infinite entropy case it suffices to consider the infinite
product $T^f\times T^f\times\ldots$} and hence the isomorphism of
$(g_{st})_{t\in\R}$ and $(g_{-st})_{t\in\R}$ for each
$s\in\R\setminus\{0\}$.

\vspace{1ex}

\noindent\textbf{E) Geodesic flow revisited, Hamiltonian
dynamics.}\footnote{This was  pointed out to us by E. Gutkin.} In
this case we obtain always reversibility, because each such flow
acts on a tangent space following geodesics: the configuration
space consists of pairs $(x,v)$ ($x$ -- placement, $v$ -- speed)
and the involution is simply given by
\[(x,v)\mapsto (x,-v).\]

\section{Special flows}\label{specialflow}
Assume that $T$ is an ergodic automorphism of  a standard
probability Borel space $(X,\mathcal{B},\mu)$. We let
$\mathcal{ B}(\R)$ and $\lambda_{\R}$ stand for the Borel $\sigma$-algebra and
Lebesgue measure on $\R$ respectively.

Assume $f:X\to\R$ is an $L^1$ strictly positive function. Denote
by $\mathcal{T}^f=(T^f_t)_{t\in\R}$  the corresponding special
flow under $f$ (see e.g.\ \cite{Co-Fo-Si}, Chapter 11). Recall
that such a flow acts on $(X^f,\mathcal{ B}^f,\mu^f)$, where
$X^f=\{(x,s)\in X\times \R:\:0\leq s<f(x)\}$ and $\mathcal{ B}^f$
$(\mu^f)$ is the restriction of $\mathcal{ B}\otimes\mathcal{
B}(\R)$ $(\mu\otimes\lambda_\R)$ to $X^f$. Under the action of
$\mathcal{T}^f$ a point in $X^f$ moves vertically at unit speed,
and we identify the point $(x,f(x))$ with $(Tx,0)$. Clearly, $T^f$ is ergodic as $T$ is ergodic. To describe
this $\R$-action formally set
\[f^{(k)}(x)=\left\{\begin{array}{ccc}
f(x)+f(Tx)+\ldots+f(T^{k-1}x) & \text{if} & k>0\\ 0 & \text{if} &
k=0\\ -\left(f(T^kx)+\ldots+f(T^{-1}x)\right)  & \text{if} & k<0.
\end{array}\right.\]
Let us consider the skew product $T_{-f}:X\times\R\to X\times\R$,
\[T_{-f}(x,r)=(Tx,r-f(x))\]
and the flow $(\sigma_t)_{t\in\R}$ on $(X\times\R,\cB\otimes\cB(\R),\mu\otimes\lambda_{\R})$
\[
\sigma_t(x,r)=(x,r+t).
\]
Then for every
$(x,r)\in X^f$ we have
\begin{equation}\label{co0}
T^f_t(x,r)=T^n_{-f}\circ\sigma_t(x,r)=(T^nx,r+t-f^{(n)}(x)),
\end{equation}
where $n\in\Z$ is  unique for which $f^{(n)}(x)\leq
r+t<f^{(n+1)}(x)$.

\begin{uw}\label{uw:aperiodic}
Recall that if $T$ is an ergodic automorphism of a standard
probability Borel space $(X,\mathcal{B},\mu)$ is aperiodic.
Moreover, any special flow $T^f$ is also aperiodic, i.e.\ for
every $t\neq 0$ we have $\mu^f(\{(x,s)\in
X^f:T^f_t(x,s)=(x,s)\})=0$.
\end{uw}

\begin{uw}\label{uwagarev1}
The special flow $T^f$ can also be seen as the quotient
$\R$-action  $(\sigma_t)_{t\in\R}$, $\sigma_t(x,r)=(x,r+t)$ on the
space $X\times\R/\equiv$, where $\equiv$ is the $T_{-f}$-orbital
equivalence relation, $T_{-f}(x,r)=(Tx,-f(x)+r)$. Indeed,
$\sigma_t\circ T_{-f}=T_{-f}\circ \sigma_t$, so $\sigma_t$ acts on
the quotient space. Moreover, the quotient space
$X\times\R/\equiv$ is naturally isomorphic with $X^f$ by choosing
the unique point from the $T_{-f}$-orbit of $(x,r)$ belonging to
$X^f$. Finally,
$$\left(T^f\right)_t(x,r)=\left(T_{-f}\right)^k\sigma_t(x,r)$$
for a unique $k\in\Z$.
\end{uw}

Using Remark~\ref{uwagarev1} we will now provide  a criterion for
a special flow to be isomorphic to its inverse.

\begin{uw}\label{uwagarev2}
Assume that $T$ is   isomorphic to its inverse: $ST=T^{-1}S$.
Assume moreover that
\begin{equation}\label{fullver}
f(Sx)-f(x)=h(x)-h(Tx)\end{equation} for a measurable $h:X\to\R$.
We claim that the special  flow $T^f$ is isomorphic to its inverse
and is reversible if $S^2=Id$ and $h(TSx)=h(x)$. Indeed, first
notice that
\begin{equation}\label{fv1}
\left(T_{-f}\right)^{-1}=T^{-1}_{f\circ T^{-1}}.\end{equation}
Set
$$ g(x)=f(x)-h(Tx)$$
and consider $S_{g,-1}:X\times\R\to X\times\R$,
$$
S_{g,-1}(x,r)=(Sx,g(x)-r).$$
Note that $S_{g,-1}$ is measurable and preserves the measure $\mu\otimes\lambda_{\R}$.
It follows immediately that
\begin{equation}\label{rv2}
S_{g,-1}\circ\sigma_t=\sigma_{-t}\circ S_{g,-1}\;\;\text{for
each}\;\;t\in\R.\end{equation} All we need to show is that
$S_{g,-1}$ acts on the  space of orbits, that is, it sends a
$T_{-f}$-orbit into another $T_{-f}$-orbit. For that, it is enough
to show that
\begin{equation}\label{rv3}
S_{g,-1}\circ T_{-f}\circ\left(S_{g,-1}\right)^{-1}= \left(T_{-f}\right)^{-1}.\end{equation}
Now, in view of~(\ref{fv1}), the equation~(\ref{rv3}) is equivalent to showing that
$$
f(T^{-1}Sx)+g(x)=g(Tx)+f(x)$$
which indeed holds as by~(\ref{fullver}) (replacing $x$ by $Tx$) we have
$$
f(STx)-f(Tx)=h(Tx)-h(T^2x),$$
so $f(T^{-1}Sx)-f(Tx)=h(Tx)-h(T^2x)$, whence
$$
f(T^{-1}Sx)-f(x)=f(Tx)-f(x)+h(Tx)-h(T^2x)=g(Tx)-g(x),$$
so indeed $S_{g,-1}$ settles an isomorphisms of $T^f$ and its inverse.

For the second part, we simply check that under  the assumption
$S^2=Id$, we have $g(Sx)=g(x)$ if and only if $h(x)=h(TSx)$.

Finally, notice that in the original functional
equation~(\ref{fullver}) we can consider $RS$ instead of $S$ with
$R\in C(T)$ (note however that even if $S^2=Id$ we may now have
$(RS)^2\neq Id$).
\end{uw}

To illustrate Remark~\ref{uwagarev2} consider  the special flow
over irrational rotation $Tx=x+\alpha$ on $\T=[0,1)$ with the roof
function $f$ of the form
$$
f(x)=\left\{\begin{array}{lll}b&\text{for} & x\in[0,a)\\
          c&\text{for}& x\in[a,1),\end{array}\right.$$
where $0<a<1$ and $b,c>0$. Then take $Rx=x+a$ and $Sx=-x$.  Note
that $RS$ is involution and check that $f\circ R\circ S=f$, which
by Remark~\ref{uwagarev2} means that $T^f$ is reversible.

If we take $f=1$ then the resulting special flow is called  the
{\em suspension flow} of $T$. Note also that special flows are
obtained by so called (measurable) change of time of the
suspension flow (see \cite{Co-Fo-Si}). It is easy to see that
$$
T^1_t(x,r)=(T^{[t+r]}x,\{t+r\}).\footnote{$[\cdot]$ stands for the integer part of a real number.}
$$
Recall that a sequence $(q_n)$  of integers, $q_n\to\infty$,  is
called a {\em rigidity sequence} for $T$ if $T^{q_n}\to Id$
(similarly we define a real-valued rigidity sequence for flows).
Note that whenever $(q_n)$ is
\begin{equation}\label{sztywnosczawieszenia}
\text{a rigidity  sequence for $T$, it is a rigidity sequence for
the suspension.}
\end{equation}
Directly from Remark~\ref{uwagarev2} it follows that the
suspension of the reversible automorphism yields a reversible
flow.

\begin{uw}\label{uwcent}
Similarly as the functional equation~(\ref{fullver}) defines an
isomorphism of $T^f$ with its inverse, if $S\in C(T)$
in~(\ref{fullver}) then
\begin{equation}
\label{fullver05}
f(Sx)-f(x)=g(x)-g(Tx)\end{equation}
with $g:X\to\R$ measurable, yields an element of $C(T^f)$. Indeed,  consider the skew product
\[S_{g}:X\times\R\to X\times\R,\qquad S_g(x,r)=(Sx,r+g(x)).\]
Then $S_g$ commutes with the flow $(\sigma_t)_{t\in\R}$ and, by~\eqref{eq:coheq},
with the skew product $T_{-f}$. It follows that $S_g$ can be considered
as an automorphism on $X^f=(X\times\R)/\equiv$ with commutes with the special flow $T^f$.
\end{uw}

The following lemma tells us that whenever the  centralizer of
$T^f$ is trivial, we can solve the functional
equation~(\ref{fullver05}) only in a trivial way.

\begin{lm}\label{lm:trivcent}
Assume that $T$ is ergodic and $C(T^f)=\{T^f_t:\:t\in\R\}$. Suppose that
$S\in C(T)$ and $g:X\to\R$ is a measurable function such that
\begin{equation}\label{eq:coheq}
f\circ S-f=g-g\circ T.
\end{equation}
Then there exist $k\in\Z$ and $t_0\in\R$ such that
\[S=T^k\quad{ and }\quad g=t_0-f^{(k)}.\]
\end{lm}

\begin{proof}
By Remark~\ref{uwcent}, the automorphism
\[S_{g}:X\times\R\to X\times\R,\qquad S_g(x,r)=(Sx,r+g(x))\]
can be considered as an element of $C(T^f)$.
By assumption, there exists $t_0\in\R$ such that $S_g=T^f_{t_0}$ on $X^f$. Therefore,
there exists a measurable function $k:X\times\R\to\Z$ with
\[(Sx,r+g(x))=S_g(x,r)=T_{-f}^{k(x,r)}(x,r+t_0)=(T^{k(x,r)}x,r+t_0-f^{(k(x,r))}(x)),\]
so
\[Sx=T^{k(x,r)}\quad\text{ and }\quad g(x)=t_0-f^{(k(x,r))}(x)).\]
It follows that $k$ does not depend on the second coordinate,
i.e.\  $k(x,r)=k(x)$ (indeed, $f^{(k_1)}(x)\neq f^{(k_2)}(x)$
whenever $k_1\neq k_2$) and $Sx=T^{k(x)}x$. Thus
\[T^{1+k(x)}x=TSx=STx=T^{k(Tx)}(Tx)=T^{k(Tx)+1}x.\]
By the ergodicity of $T$, $k\circ T=k$ and hence $k$ is constant,
which proves our claim.
\end{proof}

\section{Joinings and non-reversibility}\label{sec:joinings}
In this section we will present a method of proving
non-reversibility by  studying the weak closure of off diagonal
self-joinings (of order at least 3).

\subsection{Self-joinings for ergodic flows}
Assume that $\mathcal{ T}=(T_t)_{t\in\R}$ is an ergodic flow on
$\xbm$. For any $k\geq 2$ by a \emph{ $k$-self-joining} of $\ct$
we mean any probability $(T_t\times\ldots\times
T_t)_{t\in\R}$-invariant measure $\lambda$ on $(X^k,
\mathcal{B}^{\otimes k})$ whose projections on all coordinates are
equal to $\mu$, i.e.\
\[\lambda(X\times\ldots\times X\times A_i\times X\times\ldots\times X) =\mu(A_i)\quad\text{ for any }1\leq i\leq k\text{ and }A_k\in\mathcal{B}. \]
We will denote by $J_k(\ct)$ the set of all $k$-self-joinings for
$\ct$. If the flow $(T_t\times\ldots\times T_t)_{t\in\R}$ on
$(X^k,\lambda)$  is ergodic then $\lambda$ is called an ergodic
$k$-joining.

Let $\{B_n:n\in\N\}$ be a countable family in $\mathcal{B}$
which is dense in $\mathcal{B}$ for the (pseudo-) metric
$d_{\mu}(A,B)=\mu(A\triangle B)$. Let us consider the metric $d$
on $J_k(\ct)$ defined by
\[d(\lambda,\lambda')=\sum_{n_1,\ldots,n_k\in\N}\frac{1}{2^{n_1+\ldots+n_k}}|\lambda(B_{n_1}\times \ldots\times B_{n_k})-\lambda'(B_{n_1}\times \ldots\times B_{n_k})|.\]
Endowed with corresponding to $d$ topology the set $J_k(\ct)$ is compact.

For any family $S_1,\ldots,S_{k-1}$ of elements of the centralizer
$C(\ct)$ we will denote by $\mu_{S_1,\ldots,S_{k-1}}$  the
$k$-joining determined by
\[\mu_{S_1,\ldots,S_{k-1}}(A_1\times\ldots\times A_{k-1}\times A_k)=
\mu(S_1^{-1}A_1\cap\ldots\cap S^{-1}_{k-1}A_{k-1}\cap A_k)\] for
all $A_1,\ldots, A_k\in\mathcal{B}$. Since
$\mu_{S_1,\ldots,S_{k-1}}$ is the image of the measure $\mu$ via
the map $x\mapsto(S_1x,\ldots,S_{k-1}x,x)$, this joining is
ergodic. When all $S_i$ are time $t_i$-automorphisms  of the flow,
then  $\mu_{S_1,\ldots,S_{k-1}}$ is called an {\em off-diagonal}
self-joining.

For any probability Borel measure $P\in\mathcal{P}(\R^{k-1})$ we
will deal with the measure
$\int_{\R^{k-1}}\mu_{T_{t_1},\ldots,T_{t_{k-1}}}dP(t_1,\ldots,t_{k-1})$
defined by
\[\int_{\R^{k-1}}\mu_{T_{t_1},\ldots,T_{t_{k-1}}}dP(t_1,\ldots,t_{k-1})(A):=
\int_{\R^{k-1}}\mu_{T_{t_1},\ldots,T_{t_{k-1}}}(A)dP(t_1,\ldots,t_{k-1})\]
for any $A\in\mathcal{B}^{\otimes k}$. Then
$\int_{\R^{k-1}}\mu_{T_{t_1},\ldots,T_{t_{k-1}}}\,dP(t_1,\ldots,t_{k-1})\in
J_k(\ct)$. In the following section we will provide a criterion of
having  such an integral self-joining in the weak closure of
off-diagonal joinings for some special flows.

Similarly, we also consider joinings between different ( ergodic)
flows, say $\cT=(T_t)_{t\in\R}$ and $\cS=(S_t)_{t\in\R}$.
Following \cite{Fu}, we say that $\cT$ and $\cS$ are {\em
disjoint} if $J(\cT,\cS)=\{\mu\otimes\nu)$. We write $\cT\perp
\cS$.

\subsection{Basic criterion of existence of integral joinings in the weak closure}
Let $G$ be a locally compact Abelian Polish group. Assume that
$\|\cdot\|$ is an F-norm inducing a translation invariant metric
$d$ on $G$. Denote by $\overline{G}$ the one-point
compactification of $G$. Assume moreover that $T:\xbm\to\xbm$ is
an ergodic automorphism and $F_n:X\to G$, $n\geq1$, is a sequence
of measurable functions such that
\begin{equation}\label{ll0}
(F_n)_\ast\mu\to P\in \mathcal{ P}(G)\end{equation}
$\ast$-weakly; $ \mathcal{ P}(G)$ stands for the set of probability Borel measures on $G$.
The following result is a natural generalization of Lemma~4.1 from
\cite{Fr-Le}.
\begin{pr}\label{lm:kocyklowy}
Under the above assumptions, suppose moreover that
\begin{equation}\label{eq:inmeasure}
F_n\circ T-F_n\to 0\;\;\text{in measure}.
\end{equation} Then for
each $\phi\in C(\overline{G})$, $h:X\to G$ measurable and $j\in
L^1\xbm$ we have
\[
\lim_{n\to\infty}\int_X\phi(F_n(x)+h(x))j(x)\,d\mu(x)=
\int_X\int_G \phi(g+h(x))j(x)\,dP(g)d\mu(x).
\]
\end{pr}
\begin{proof}
We will first assume that $h=0$. In order to prove the  above weak
convergence we need to show that
\begin{equation}\label{ll1}
\lim_{n\to\infty}\int_X\phi(F_n(x))j(x)\,d\mu(x)=0
\end{equation}
for  each $j$ whose mean is zero. Now, since the functions of the
form $k\circ T-k$ with $k\in L^1\xbm$ are dense in the latter
subspace we need to show that
$\lim_{n\to\infty}\int_X\phi(F_n(x))(k(Tx)-k(x))\,d\mu(x)=0$. We
have
\begin{align*}
\int_X\phi(F_n(x))j(x)\,d\mu(x)&=\int_X\phi(F_n(x))k(Tx)\,d\mu(x)-\int_X\phi(F_n(Tx))k(Tx)\,d\mu(x)
\\
&=\int_X\left(\phi(F_n(x))-\phi(F_n(Tx))\right)k(Tx)\,d\mu(x).
\end{align*}
Now, since $\phi$ is uniformly continuous and bounded
and~(\ref{eq:inmeasure}) holds,~(\ref{ll1}) follows.

Suppose now that $h=\sum_{i=1}^mh_i\cdot{\mathbf 1}_{A_i}$ is a
simple function ($h_i\in G$ and the sets $A_i$, $1\leq i\leq m$
form a measurable partition of $X$). We have
\begin{multline*}
\int_X\phi(F_n(x)+h(x))j(x)\,d\mu(x)=\sum_{i=1}^m \int_X\phi(F_n(x)+h_i)j(x){\mathbf 1}_{A_i}(x)\,d\mu(x) \\
\to \sum_{i=1}^m\int_G\phi(g+h_i)\,dP(g)\int_X(j\cdot{\mathbf 1}_{A_i})(x)\,d\mu(x)\\
=\int_X\int_G\phi(g+h(x))j(x)\, dP(g)d\mu(x).
\end{multline*}
All we need to show now is that for each $\vep>0$
we can find a measurable $h_{\vep}:X\to G$ taking only finitely
many values so that
\begin{equation}\label{ll3}
\left|\int_X\phi(F_n(x)+h(x))j(x)\,d\mu(x)-
\int_X\phi(F_n(x)+h_{\vep}(x))j(x)\,d\mu(x)\right|<\vep
\end{equation}
and
\begin{equation}\label{ll4}
\left|\int_X\phi(g+h(x))j(x)\,d\mu(x)-
\int_X\phi(g+h_{\vep}(x))j(x)\,d\mu(x)\right|<\vep.
\end{equation}
Given $\vep>0$ we select $\delta>0$ so that
\[
\|g_1-g_2\|<\delta\;\;\Rightarrow\;\;|\phi(g_1)-\phi(g_2)|<\vep/
(2\|j\|_{L^1}).
\]
Then select $\eta>0$ so that whenever
$\mu(A)<\eta$
\[
\int_A|j(x)|\,d\mu(x)<\vep/(4\|\phi\|_{\infty}).
\]
Finally choose $h_{\vep}:X\to G$ measurable so that $h_{\vep}$
takes only finitely many values and
\[
\mu\left(\{x\in X:\:|h_{\vep}(x)-h(x)|\geq\delta\}\right)<\eta.
\]
We have
\begin{multline*}
\left|\int_X\phi(F_n(x)+h(x))j(x)\,d\mu(x)- \int_X\phi(F_n(x)+h_{\vep}(x))j(x)\,d\mu(x)\right| \\
\leq 2\int_{\{x\in X:\:\|h_{\vep}(x)-h(x)\|\geq\delta\}}\|\phi\|_\infty|j(x)|\,d\mu(x) \\
+\int_{\{x\in X:\:\|h_{\vep}(x)-h(x)\|<\delta\}}\frac{\vep}
{2\|j\|_{L^1}}|j(x)|\,d\mu(x)<\vep.
\end{multline*}
We established~(\ref{ll3})
and~(\ref{ll4}) follows in the same manner.
\end{proof}

\begin{lm}\label{lm:4.2}
For all $t_0,\dots, t_{d-1}\in\R$ and all  measurable sets
$A_0,\dots, A_d \subset X^f$ we have
\begin{equation*}
\mu^f\left(\bigcap_{i=0}^{d-1}(T^f)_{t_i}A_i\cap A_d \right)=
\sum_{k_0,k_1,\dots, k_{d-1}\in\Z}\mu \otimes \lambda_{\R}
\left( \bigcap_{i=0}^{d-1}\left((T_{-f})^{k_i}\sigma_{t_i}A_i \right)\cap A_d  \right).
\end{equation*}
Moreover, the sets
$\bigcap_{i=0}^{d-1}\left((T_{-f})^{k_i}\sigma_{t_i}A_i
\right)\cap A_d$, $(k_0,\dots,k_{d-1})\in\Z^d$ are pairwise
disjoint.
\end{lm}

\begin{proof}
Given $(t_0,\dots,t_{d-1})\in \R^d$ and $(x,r)\in X^f$,
\[
(T^f)_{t_i}(x,r)=(T_{-f})^{k_i}\sigma_{t_i}(x,r)\text{ for a unique }k_i\in \Z \text{ for }0\leq i\leq d-1.
\]
Hence if we fix $i\in\{0,\ldots,d-1\}$ then
$$
T^f_{t_i}(A_i)=\bigcup_{k\in\Z} T^k_{-f}\sigma_{t_i}(A_i)\cap
X^f.$$  The sets on the RHS of the above equality are pairwise
disjoint and they correspond to the images via $T^f_{t_i}$ of the
partition of $X^f$ into pairwise disjoint sets on which the action
of $T^f_{t_i}$ corresponds to $T^k_{-f}\sigma_{t_i}$, $k\in\Z$.
Therefore (remembering that $A_d\subset X^f$)
$$\bigcap_{i=0}^{d-1}T^f_{t_i}(A_i)\cap A_d=\bigcap_{i=0}^{d-1}\bigcup_{k_i\in\Z} T^k_{-f}\sigma_{t_i}(A_i)\cap A_d$$
$$=\bigcup_{k_0,k_1,\dots, k_{d-1}\in\Z}
\left( \bigcap_{i=0}^{d-1}\left((T_{-f})^{k_i}\sigma_{t_i}A_i \right)\cap A_d  \right).$$
It follows that the above representation corresponds
to the partition of the space $X^f$  into countably many subsets $X_{k_0,\dots,k_{d-1}}^f$,
$(k_0,\dots,k_{d-1})\in\Z^d$, on which, for each $i=0,\ldots,d-1$,  $(T^f)_{t_i}$ acts as
$(T_{-f})^{k_i}\sigma_{t_i}$. Moreover,
since $(T^f)_{t_i}$ is an automorphism, the images
$(T_{-f})^{k_i}\sigma_{t_i}\left(X^f_{k_0,\dots,k_{d-1}} \right)$
are pairwise disjoint for $(k_0,\dots,k_{d-1})\in\Z^d$ and the result follows.
\end{proof}

\begin{lm}\label{lm:4.3}
Suppose that $A_0,\dots,A_d\subset X\times \R$ are  measurable
rectangles of the form $A_i=B_i\times C_i$ for $0\leq i\leq d$.
Then
\begin{align*}
\mu\otimes \lambda_{\R} &\left(\bigcap_{i=0}^{d-1}\left( (T_{-f})^{k_i}A_i\right) \cap A_d\right) \\
&=\int_{\bigcap_{i=0}^{d-1}T^{k_i}B_i\cap B_d} \lambda_{\R} \left( \bigcap_{i=0}^{d-1} \left(C_i+f^{(-k_i)}(x) \right)\cap C_d \right) \ d\mu(x).
\end{align*}
\end{lm}

\begin{proof}
We have $(x,r) \in \bigcap_{i=0}^{d-1}(T_{-f})^{k_i}(B_i\times C_i)\cap B_d\times C_d$
if and only if
\[
(x,r)=(T^{k_i}y_i,r_i-f^{(k_i)}(y_i))\quad\text{ with }\quad (y_i,r_i)\in B_i\times C_i
\]
for $0\leq i \leq d-1$ and $(x,r)\in B_d\times C_d$. Thus
\[
(x,r) \in \bigcap_{i=0}^{d-1}(T_{-f})^{k_i}(B_i\times C_i)\cap B_d\times C_d
\]
if and only if
\[
x\in \bigcap_{i=0}^{d-1}T^{k_i}B_i \cap B_d\quad\text{ and  }\quad
r\in \bigcap_{i=0}^{d-1}\left( C_i -
f^{(k_i)}(T^{-k_i}x)\right)\cap C_d.
\]
Since $f^{(m)}(T^{-m}x)=-f^{(-m)}(x)$ for any $m\in\Z$, the result
follows.
\end{proof}

As an immediate consequence of (the second part of) Lemma~\ref{lm:4.2} and
Lemma~\ref{lm:4.3} we obtain the following result.
\begin{uw}\label{wn:wn}
For fixed $t_0,\dots,t_{d-1}\in\R$ and $k_{i_0}\in\Z$ with $0\leq
i_0\leq d-1$ and for all measurable sets $A_i\subset X^f$ of the
form $A_i=B_i\times C_i$ where $0\leq i\leq d$ we
have\footnote{Here and in what follows $ \sum_{k_j\in\Z,j\neq
i_0}$ means
$\sum_{k_0,\ldots,k_{i_0-1},k_{i_0+1},\ldots,k_{d-1}}$.}
\begin{multline*}
    \sum_{k_j\in\Z,j\neq i_0}  \mu\otimes \lambda_{\R} \left(\bigcap_{i=0}^{d-1}(T_{-f})^{k_i}\sigma_{t_i}A_i \cap A_d \right)
    \leq \mu\otimes \lambda_{\R}\left((T_{-f})^{k_{i_0}}\sigma_{t_{i_0}}A_{i_0}\cap A_d \right)\\
    =\int_{T^{k_{i_0}}B_{i_0}\cap B_d}\lambda_{\R} \left( (C_{i_0}+t_{i_0}+f^{(-k_{i_0})}(x))\cap C_d\right)\ d\mu(x)\\
    \leq \int_X \lambda_{\R} \left(\left(C_{i_0}+t_{i_0}+f^{(-k_{i_0})}(x)\right)\cap C_d \right) \ d\mu(x).
\end{multline*}
\end{uw}

Suppose that $f\in L^2(X,\mu)$ and $(q_n)_{n\in\N}$ is a
sequence of integer numbers such that the sequence
$(f_0^{(q_n)})_{n\in\N}$ is bounded in $L^2(X,\mu)$, where
$f_0:=f-\int f\,d\mu$.

\begin{lm}[Lemma~4.4 in \cite{Fr-Le}]\label{lm:fl04}
For every pair of bounded sets $D,E\subset \R$ there exists a
sequence $(a_k)_{k\in\Z}$ of positive numbers such that
\begin{itemize}
\item
$\sum_{k\in\Z}a_k<+\infty$,
\item
$\int_X \lambda_{\R}\left((D-f_0^{(q_n)}(x)+f^{(k)}(x)) \cap E
\right)\ d\mu\leq a_k$ for each $n\in\N$ and $k\in\Z$.
\end{itemize}
\end{lm}

\begin{uw}\label{uw:uw}
For any $l_1,l_2\in\Z$ we have
\begin{align*}
f^{(l_1+l_2)}(x)-l_1=f^{(l_1)}(x)-l_1+f^{(l_2)}(T^{l_1}x)
=f_0^{(l_1)}(x)+f^{(l_2)}(T^{l_1}x).
\end{align*}
\end{uw}

\begin{pr}\label{main1}
Suppose that $f\in L^2(X,\mu)$ is a positive function with $\int_X
f\, d\mu=1$ and there exists $c>0$ such that $f^{(k)}\geq ck$ for
a.a. $x\in X$ and for all $k\in \N$ large enough. Let
$(q_n^i)_{n\geq 1}$ be rigidity sequences for~$T$ for $0\leq i\leq
d-1$. Moreover, suppose that the sequences $\left(f_0^{(q_n^i)}
\right)_{n\geq 1}$ are bounded in $L^2(X,\mu)$  for $0\leq i\leq
d-1$ and
\begin{equation}\label{eq:weakconvmeas}
\left(f_0^{(q_n^0)},\dots,f_0^{(q_n^{d-1})}\right)_\ast (\mu) \to P \text{ weakly in
}\mathcal{P}(\mathbb{R}^d).
\end{equation}
Then
\begin{equation}\label{eq:teza}
\left(\mu^f\right)_{T^f_{q^0_n},\ldots,T^f_{q^{d-1}_n}}
\to \int_{\R} \left(\mu^f\right)_{T^f_{t_0},\ldots,T^f_{t_{d-1}}}\,dP(t_0,\dots,t_{d-1}).
\end{equation}
\end{pr}

\begin{uw}
Before we pass to the proof let us see the assertion of the
proposition in case of the suspension flow, i.e.\ $f=1$, that is,
$f_0=0$. In this case $P$ is the Dirac measure at zero of $\R^d$,
so in~\eqref{eq:teza} we have a convergence to the diagonal
$(d+1)$-self-joining $\Delta_{d+1}$. This can be see directly in
view of~\eqref{sztywnosczawieszenia}; indeed, all sequences
$(q_n^i)$ are rigidity sequences for the suspension flow and hence
yield convergence of the LHS  in~\eqref{eq:teza} to
$\Delta_{d+1}$. It follows that Proposition~\ref{main1} provides a
class of (measurable) change of times of the suspension flow, so
that the LHS in~(\ref{eq:teza}) weakly converges to the integral
of off-diagonal $(d+1)$-self-joinings given by the limit
distribution in~(\ref{eq:weakconvmeas}).

If $T$ is rigid and reversible,  then so is its suspension. We
will see later the the changes of time  described in
Proposition~\ref{main1} may lead to non-reversible flows.
\end{uw}

\begin{proof}
First notice that all we  need to show is that~\eqref{eq:teza}
holds for all measurable rectangles $A_i\subset X^f$ of the form
$A_i=B_i\times C_i$ ($0\leq i \leq d$) such that $C_i$ are bounded
for $0\leq i \leq d$.

Setting
\begin{equation*}
a_{k_0,\dots,k_{d-1}}^n:=\mu\otimes
\lambda_{\R}\left(\bigcap_{i=0}^{d-1}\left((T_{-f})^{-k_i}(T_{-f})^{-q_n^i}
\sigma_{-q_n^i}A_i\right)\cap A_d \right)
\end{equation*}
for $n\in \N$, $k_0,\dots,k_{d-1}\in \Z$, by Lemma~\ref{lm:4.2},
we have
\begin{equation}\label{eq:sumaas}
    \mu^f\left(\bigcap_{i=0}^{d-1}(T^f)_{-q_n^i}A_i\cap A_d \right)
    =\sum_{k_0,\dots,k_{d-1}\in \Z}a_{k_0,\dots,k_{d-1}}^n.
\end{equation}
Since $\sigma_{-q_n^i}(A_i)=B_i\times(C_i-q_n^i)$, using
Lemma~\ref{lm:4.3} and Remark~\ref{uw:uw}, we obtain
\begin{align}\label{eq:akl}
\begin{split}
    &a_{k_0,\dots,k_{d-1}}^n
    \!=\!\int_{\bigcap_{i=0}^{d-1}T^{-k_i-q_n^i}B_i\cap B_d}\!\lambda_{\R}
    \Big(\!\bigcap_{i=0}^{d-1}\!\big(\!C_i-q_n^i+f^{(k_i+q_n^i)}(x)\!\big)\!\cap C_d \!\Big) d\mu(x)\\
    &=\!\int_{\bigcap_{i=0}^{d-1}T^{-k_i-q_n^i}B_i\cap B_d}
    \!\lambda_{\R} \Big(\!\bigcap_{i=0}^{d-1}\!\big(\!C_i+f_0^{(q_n^i)}(x)+f^{(k_i)}(T^{q_n^i}x)\! \big)\! \cap C_d\!\Big) d\mu(x).
\end{split}
\end{align}
Using again  Remark~\ref{uw:uw}, for all
$n\in \N$, $k_0,\dots,k_{d-1}\in \Z$  we have
\begin{align}\label{eq:4b}
\begin{split}
&b_{k_0,\dots,k_{d-1}}^n:=\!\int_{\bigcap_{i=0}^{d-1}T^{-k_i}B_i\cap
B_d}\!\lambda_{\R} \!
\left(\!\bigcap_{i=0}^{d-1}\!\left(\!C_i-q_n^i+f^{(k+q_n^i)}(x)\!\right)\!\cap C_d \!\right)\! d\mu(x)\\
&=\!\int_{\bigcap_{i=0}^{d-1}T^{-k_i}B_i\cap B_d}\! \lambda_{\R}\!
\left(\!\bigcap_{i=0}^{d-1}\!\left(\!C_i\!+\!f_0^{(q_n^i)}(x)\!+\!f^{(k_i)}(T^{q_n^i}x)
\!\right)\!\cap C_d \!\right)\! d\mu(x).
\end{split}
\end{align}
We claim that
\begin{equation}\label{eq:poje}
\lim_{n\to\infty}\left(a_{k_0,\dots,k_{d-1}}^n-b_{k_0,\dots,k_{d-1}}^n
\right)=0\quad\text{ for all }\quad  k_0,\dots,k_{d-1}\in \Z.
\end{equation}
Notice that in formulas~\eqref{eq:akl} and~\eqref{eq:4b}
describing $a_{k_0,\dots,k_{d-1}}^n$ and $b_{k_0,\dots,k_{d-1}}^n$
respectively we have
\[
\psi_n(x):=\lambda_{\R}\left(\bigcap_{i=0}^{d-1}\left(C_i-q_n^i+f^{(k_i+q_n^i)}(x)\right)\cap
C_d \right) \leq \lambda_{\R}(C_d).
\]
Therefore,
\begin{align*}
\big|a_{k_0,\dots,k_{d-1}}^n-b_{k_0,\dots,k_{d-1}}^n&\big|=
\Big|\int_{\bigcap_{i=0}^{d-1}T^{-k_i-q_n^i}B_i\cap
B_d}\psi_n\,d\mu-\int_{\bigcap_{i=0}^{d-1}T^{-k_i}B_i\cap
B_d}\psi_n\,d\mu\Big|\\
&\leq\lambda_{\R}(C_d)\
\mu\Big(\Big(\bigcap_{i=0}^{d-1}T^{-k_i-q_n^i}B_i\cap
B_d\Big)\triangle \Big(\bigcap_{i=0}^{d-1}T^{-k_i}B_i\cap
B_d\Big)\Big)\\&\leq\lambda_{\R}(C_d)\sum_{i=0}^{d-1}\mu(T^{q_n^i}B_i\triangle
B_i).
\end{align*}
and $(q_n^i)_{n\geq 1}$ for $0\leq i\leq d-1$ are rigidity
sequences for $T$, this gives \eqref{eq:poje}.

Let $\varepsilon>0$ and fix $0\leq i_0\leq d-1$ and
$k_{i_0}\in\Z$.
By Remark~\ref{wn:wn}  and Remark~\ref{uw:uw}, for any $n\in\N$ we
have
\begin{align*}
    \sum_{k_j\in\Z,j\neq i_0}\!a_{k_0,\dots,k_{d-1}}^n &\leq \int_X
    \lambda_{\R}\!\left( \left(C_{i_0}-q_n^{i_0}+f^{(k_{i_0}+q_n^{i_0})}(x)
    \right)\cap C_d\right) d\mu(x)
    \\
    &=\int_X \!\lambda_{\R}\left(\left( C_{i_0}+f_0^{(q_n^{i_0})}(x)+f^{(k_{i_0})}(T^{q_n^{i_0}}x)\right)\cap C_d\right)  d\mu(x)\\
    &=\int_X \!\lambda_{\R}\left(\left( C_{i_0}-f_0^{(-q_n^{i_0})}(x)+f^{(k_{i_0})}(x)\right)\cap C_d\right)  d\mu(x).
\end{align*}
Therefore, by Lemma~\ref{lm:fl04}, there exists $M>0$ such that
for  any $0\leq i_0\leq d-1$ and $n\in \N$
\begin{equation*}
    \sum_{|k_{i_0}|>M}\sum_{k_j\in\Z, j\neq i_0} a_{k_0,\dots, k_{d-1}}^n<\frac{\varepsilon}{4d}.
\end{equation*}
It follows that
\begin{equation}\label{eq:sumaazmaxem}
\sum_{\max(|k_0|,\dots,|k_{d-1}|)>M} a_{k_0,\dots,k_{d-1}}^n\leq
\sum_{0\leq i_0\leq d-1} \sum_{|k_{i_0}|>M}\sum_{k_j \in\Z,j\neq
i_0}a_{k_0,\dots,k_{d-1}}^n\leq \vep/4.
\end{equation}

Let us consider $F_n:X\to\R^d$, $F_n(x)= (F^0_n(x),\ldots,
F^{d-1}_n(x))$ with
\begin{align*}
    F^i_n(x)=f_0^{(q_n^i)}(x)+f^{(k_i)}(T^{q_n^i}x)-f^{(k_i)}(x)\quad\text{ for
    }\quad i=0,\ldots,d-1
\end{align*}
and $(k_0,\ldots,k_{d-1})$ fixed.
Since $(q_n^i)_{n\geq 1}$ is a rigidity sequence for $T$,
$f^{(k_i)}\circ T^{q_n^i}-f^{(k_i)}$ tends to zero in measure when $n\to\infty$
for every $i=0,\ldots,d-1$. Therefore, \eqref{eq:weakconvmeas}
implies  $(F_n)_*\mu\to P$ weakly in $\mathcal{P}(\mathbb{R}^d)$.
Moreover,
\[F^i_n\circ T-F^i_n=(f\circ T^{k_i})\circ T^{q_n^i}-(f\circ T^{k_i}) \quad \text{ for  }\quad i=0,\ldots,d-1,\]
so $F_n\circ T-F_n\to 0$ in measure. Now using
Proposition~\ref{lm:kocyklowy} with $G=\R^d$ and \eqref{eq:4b} we
obtain
\begin{align}\label{eq:przedost}
\begin{split}
&b^n_{k_0,\dots,k_{d-1}}\\
&=\int_{  \bigcap_{i=0}^{d-1}T^{-k_i}B_i
\cap B_d}\!\lambda_{\R} \Big(\bigcap_{i=0}^{d-1}\big(C_i +
f_0^{(q_n^i)}(x) + f^{(k_i)}(T^{q_n^i}x)\big)\cap C_d
\Big) d\mu(x)
    \\
& =\int_{  \bigcap_{i=0}^{d-1}T^{-k_i}B_i \cap B_d}\!\lambda_{\R}
\Big(\bigcap_{i=0}^{d-1}\big(C_i + F^i_n(x) + f^{(k_i)}(x)\big)\cap
C_d \Big) d\mu(x)
\\& \to\! \int_{\bigcap_{i=0}^{d-1}T^{-k_i}B_i \cap B_d} \int_{\R^d}\!\lambda_{\R}
\Big(\bigcap_{i=0}^{d-1}\big(\!C_i\!+\!t_i\!+\! f^{(k_i)}(x)\!
\big)\!\cap\!C_d\!\Big) dP(t_0,\dots,t_{d-1})d\mu(x)\\
&\quad =:c_{k_0,\dots,k_{d-1}}
\end{split}
\end{align}
for each $k_0,\dots, k_{d-1}\in\Z$. By Fubini's theorem and
Lemma~\ref{lm:4.3} we have
\begin{align}\label{eq:przedost1}
\begin{split}
    &c_{k_0,\dots,k_{d-1}}\\
    &=\int_{\R^d} \int_{\bigcap_{i=0}^{d-1}T^{k_i}B_i\cap
    B_d}\!\lambda_{\R}
    \Big(\!\bigcap_{i=0}^{d-1}\!\big( C_i\!+\!t_i\!+\!f^{(k_i)}(x) \big)\!\cap\! C_d \Big) d\mu(x)\,dP(t_0,\dots,t_{d-1})\\
    &=\int_{\R^d} \mu\otimes \lambda_{\R} \Big( \bigcap_{i=0}^{d-1}(T_{-f})^{-k_i}\sigma_{t_i}(B_i\times C_i) \cap(B_d\times C_d) \Big)
    dP(t_0,\dots,t_{d-1})\\
    &=\int_{\R^d} \mu\otimes \lambda_{\R} \big( \bigcap_{i=0}^{d-1}(T_{-f})^{-k_i}\sigma_{t_i}A_i \cap A_d \big)
    dP(t_0,\dots,t_{d-1}).
\end{split}
\end{align}
Moreover, by Lemma~\ref{lm:4.2},
\begin{align}\label{eq:przedost2}
\begin{split}
&\sum_{k_0,\dots,k_{d-1}\in\Z}c_{k_0,\dots,k_{d-1}}\\
 &   =\sum_{k_0,\dots, k_{d-1}\in\Z} \int_{\R^d }\!\mu\!\otimes\!\lambda_{\R}\big(
    \bigcap_{i=0}^{d-1}(T_{-f})^{-k_i}\sigma_{t_i}A_i\cap A_d\big) dP(t_0,\dots,t_{d-1})\\
   & =\int_{\R^d} \mu^f \Big(\bigcap_{i=0}^{d-1}(T^f)_{t_i}A_i\cap A_d \Big) dP(t_0,\dots,t_{d-1}).
\end{split}
\end{align}
Increasing $M$, if necessary, we can assume that
\begin{equation}\label{eq:sumaczmaxem}
\sum_{\max(|k_0|,\dots,|k_{d-1}|)>M} c_{k_0,\dots,k_{d-1}}\leq
\vep/4.
\end{equation}
Combining \eqref{eq:poje} with \eqref{eq:przedost} we get
\[a_{k_0,\dots,k_{d-1}}^n\to c_{k_0,\dots,k_{d-1}}\quad \text{ for all }\quad k_0,\dots,k_{d-1}\in\Z.\]
Therefore, there exists $N\in\N$ such that for all $n\geq N$ and
$k_0,\dots,k_{d-1}\in\Z$ with $\max(|k_0|,\dots,|k_{d-1}|)\leq M$
\[|a_{k_0,\dots,k_{d-1}}^n-c_{k_0,\dots,k_{d-1}}|<\frac{\vep}{2(2M+1)^d}.\]
In view of \eqref{eq:sumaazmaxem} and \eqref{eq:sumaczmaxem}, it
follows that
\begin{multline*}
\Big|\sum_{k_0,\dots,k_{d-1}\in\Z}a_{k_0,\dots,k_{d-1}}^n-
\sum_{k_0,\dots,k_{d-1}\in\Z}c_{k_0,\dots,k_{d-1}}\Big|\\ \leq
\sum_{\max(|k_0|,\dots,|k_{d-1}|)>M} a^n_{k_0,\dots,k_{d-1}}
+\sum_{\max(|k_0|,\dots,|k_{d-1}|)>M} c_{k_0,\dots,k_{d-1}}\\+
\sum_{\max(|k_0|,\dots,|k_{d-1}|)\leq
M}|a_{k_0,\dots,k_{d-1}}^n-c_{k_0,\dots,k_{d-1}}|<\vep.
\end{multline*}
By \eqref{eq:sumaas} and \eqref{eq:przedost2}, this completes the
proof.
\end{proof}

\subsection{FS-type joinings and non-reversibility}
From now on we assume that all flows under consideration are
ergodic and aperiodic.

For any ${\bar{\vep}}=(\vep_0,\ldots,\vep_{d-1})\in\{0,1\}^{\prime
d}:=\{0,1\}^d\setminus\{(0,\ldots,0)\}$ and for any vector
$\bar{x}=(x_0,\ldots,x_{d-1})\in\R^d$ let
\[\bar{x}(\vep)=\vep_0x_0+\vep_1x_1+\ldots+\vep_{d-1}x_{d-1}.\]
If we look at the assumptions of Proposition~\ref{main1} we see
that for any choice of
${\bar{\vep}}=(\vep_0,\ldots,\vep_{d-1})\in\{0,1\}^{\prime d}$
setting $\bar{q}_n:=(q_n^0,q_n^1,\ldots,q_n^{d-1})$ we have
\begin{equation*}
%\label{no1}
\left(\bar{q}_n({\bar{\vep}})\right)_{n\geq1}\;\text{is a rigidity
sequence for $T$ and}\;
\left(f_0^{(\bar{q}_n({\bar{\vep}}))}\right)_{n\geq1}\;\text{is
bounded in $L^2$}.
\end{equation*}
we can assume that
\[\left(\left(f_0^{(\bar{q}_n({\bar{\vep}}))}\right)_{{\bar{\vep}}\in\{0,1\}^{\prime d}}\right)_\ast\to Q\in{\mathcal
P}(\R^{\{0,1\}^{\prime d}})\quad\text{when}\quad n\to\infty.\]

For any $\bar{t}\in\R^{\{0,1\}^{\prime d}}$ denote by
$\mu^f_{\bar{t}}\in J_{2^d}(T^f)$ the off-diagonal
$2^d$-self-joining defined the family of elements of the
centralizer
$\{T^f_{{\bar{t}}_{\bar{\vep}}}:{\bar{\vep}}\in\{0,1\}^{\prime
d}\}$, this is
\[\mu^f_{\bar{t}}\Big(\prod_{{\bar{\vep}}\in\{0,1\}^d}A_{\bar{\vep}}\Big)=\mu^f\Big(\bigcap_{{\bar{\vep}}\in\{0,1\}^d}T^f_{-{\bar{t}}_{\bar{\vep}}}A_{\bar{\vep}}\Big),\]
we make the convention that ${\bar{t}}_{(0,\ldots,0)}=0$ for any
${\bar{t}}\in\R^{\{0,1\}^{\prime d}}$. Hence, in view of
Proposition~\ref{main1}
\begin{equation}\label{no3}
\mu^f_{(\bar{q}_n({\bar{\vep}}))_{{\bar{\vep}}\in\{0,1\}^{\prime
d}}} \to \int_{\R^{\{0,1\}^{\prime d}}}
\mu^f_{-{\bar{t}}}\,dQ({\bar{t}}).
\end{equation}
Recall that given $\bar{a}=(a_0,\ldots,a_{d-1})\in \R^d$, by the
finite sum set $FS(\bar{a})$ of $\bar{a}$ we mean
\begin{align*}
FS(\bar{a})&=\{a_0,a_1,\ldots,a_{d-1},a_{0}+a_1,a_{0}+a_2,\ldots,a_0+a_1+\ldots+a_{d-1}\}\\
&=\big\{\bar{a}({\bar{\vep}}):{\bar{\vep}}\in\{0,1\}^{\prime
d}\big\}.
\end{align*}

The off-diagonal joinings on the LHS of~(\ref{no3}) have certain
symmetry property (explored below) which, when assuming
isomorphism of the flow with its inverse, should result in a
certain symmetry property of the limit measure~$Q$. Hence, if the
expected symmetry of $Q$ does not take place we obtain that the
flow is not isomorphic to its inverse. We now pass to a precise
description of the symmetry of $Q$ in a more general situation.

Assume that $\cT=(T_t)_{t\in\R}$ is an ergodic and aperiodic flow
on $\xbm$. Suppose that there exists a sequence
$(\bar{q}_n)_{n\geq 1}$ in $\R^d$, and a probability Borel measure
$Q\in{\mathcal P}(\R^{\{0,1\}^{\prime d}})$ such that
\begin{equation}\label{w1}
\mu_{(\bar{q}_n({\bar{\vep}}))_{{\bar{\vep}}\in\{0,1\}^{\prime
d}}} \to\int_{\R^{\{0,1\}^{\prime d}}}
\mu_{-{\bar{t}}}\,dQ({\bar{t}})\quad\text{in}\quad J_{2^d}(\cT).
\end{equation}
Note that, because of the aperiodicity of $\cT$, for distinct
$\bar{t},\bar{s}\in\R^{\{0,1\}^{\prime d}}$ the measures
$\mu_{{\bar{t}}}$, $\mu_{{\bar{s}}}$ are orthogonal. Therefore,
the integral in \eqref{w1} represents the ergodic decomposition of
the limit measure.

We also assume that $\mathcal{ T}$ and $\mathcal{ T}\circ(-1)$ are
isomorphic, i.e. for some invertible $S:\xbm\to\xbm$
\begin{equation}\label{w2}
S\circ T_t\circ S^{-1}=T_{-t}\quad\text{for each}\quad t\in\R.
\end{equation}
The map $S:X\to X$ induces a continuous (affine) invertible map $S_*:J_{2^d}(\cT)\to J_{2^d}(\cT)$
such that
\[
S_*(\rho)\Big(\prod_{{\bar{\vep}}\in\{0,1\}^d}A_{\bar{\vep}}\Big):=\rho\Big(\prod_{{\bar{\vep}}\in\{0,1\}^d}S^{-1}A_{\bar{\vep}}\Big)\quad\text{
for }\quad A_{{\bar{\vep}}}\in\mathcal{B},\;
{\bar{\vep}}\in\{0,1\}^d.
\]
Moreover, for any ${\bar{t}}\in\R^{\{0,1\}^{\prime d}}$
\begin{align*}
S_*&(\mu_{{\bar{t}}})\Big(\prod_{{\bar{\vep}}\in\{0,1\}^d}A_{\bar{\vep}}\Big)=
\mu_{{\bar{t}}}\Big(\prod_{{\bar{\vep}}\in\{0,1\}^d}S^{-1}A_{\bar{\vep}}\Big)=
\mu\Big(\bigcap_{{\bar{\vep}}\in\{0,1\}^d}T_{-{\bar{t}}_{\bar{\vep}}}S^{-1}A_{\bar{\vep}}\Big)\\
&=
\mu\Big(S^{-1}\bigcap_{{\bar{\vep}}\in\{0,1\}^d}T_{{\bar{t}}_{\bar{\vep}}}A_{\bar{\vep}}\Big)
=
\mu\Big(\bigcap_{{\bar{\vep}}\in\{0,1\}^d}T_{{\bar{t}}_{\bar{\vep}}}A_{\bar{\vep}}\Big)
=\mu_{-{\bar{t}}}\Big(\prod_{{\bar{\vep}}\in\{0,1\}^d}A_{\bar{\vep}}\Big)
\end{align*}
Thus
\begin{equation}\label{w3}
S_*\left(\mu_{{\bar{t}}}\right)=\mu_{-{\bar{t}}}.
\end{equation}
By the continuity of $S_*$
\begin{align*}
S_*\big(\mu_{(\bar{q}_n({\bar{\vep}}))_{{\bar{\vep}}\in\{0,1\}^{\prime
d}}}\big) \to S_*\Big(\int_{\R^{\{0,1\}^{\prime d}}}
\mu_{-{\bar{t}}}\,dQ({\bar{t}})\Big)=\int_{\R^{\{0,1\}^{\prime
d}}} S_*(\mu_{-{\bar{t}}})\,dQ({\bar{t}}).
\end{align*} In view of \eqref{w3}, it follows
that
\begin{equation}\label{eq:zbimin}
\mu_{(-\bar{q}_n({\bar{\vep}}))_{{\bar{\vep}}\in\{0,1\}^{\prime
d}}} \to \int_{\R^{\{0,1\}^{\prime
d}}}\mu_{{\bar{t}}}\,dQ({\bar{t}}).
\end{equation}
Let us consider the involution \[I:\{0,1\}^d\to\{0,1\}^d,\qquad
I(\vep_0,\ldots,\vep_{d-1})=(1-\vep_0,\ldots,1-\vep_{d-1}),\]
\[\bar{\theta}:\R^{\{0,1\}^{\prime d}}\to\R^{\{0,1\}^{\prime d}},\qquad
\bar{\theta}\Big(({\bar{t}}_{{\bar{\vep}}})_{{\bar{\vep}}\in\{0,1\}^{\prime
d}}\Big)=\Big(\big({\bar{t}}_{(1,\ldots,1)}-{\bar{t}}_{I({\bar{\vep}})}\big)_{{\bar{\vep}}\in\{0,1\}^{\prime
d}}\Big).\] Thus, by~(\ref{w1})
\begin{align*}
\mu&_{(-\bar{q}_n({\bar{\vep}}))_{{\bar{\vep}}\in\{0,1\}^{\prime
d}}}\Big(\prod_{{\bar{\vep}}\in\{0,1\}^d}A_{\bar{\vep}}\Big)=\mu\Big(\bigcap_{{\bar{\vep}}\in\{0,1\}^d}T_{\bar{q}_n({\bar{\vep}})}A_{\bar{\vep}}\Big)\\
&=\mu\Big(\bigcap_{{\bar{\vep}}\in\{0,1\}^d}T_{\bar{q}_n({\bar{\vep}})-\bar{q}_n(1,\ldots,1)}A_{\bar{\vep}}\Big)
=\mu\Big(\bigcap_{{\bar{\vep}}\in\{0,1\}^d}T_{-\bar{q}_n(I({\bar{\vep}}))}A_{\bar{\vep}}\Big)\\
&=\mu\Big(\bigcap_{{\bar{\vep}}\in\{0,1\}^d}T_{-\bar{q}_n({\bar{\vep}})}A_{I({\bar{\vep}})}\Big)=
\mu_{(\bar{q}_n({\bar{\vep}}))_{{\bar{\vep}}\in\{0,1\}^{\prime
d}}}\Big(\prod_{{\bar{\vep}}\in\{0,1\}^d}A_{I({\bar{\vep}})}\Big)\\
&\to \int_{\R^{\{0,1\}^{\prime
d}}}\mu_{-{\bar{t}}}\Big(\prod_{{\bar{\vep}}\in\{0,1\}^d}A_{I({\bar{\vep}})}\Big)\,dQ({\bar{t}})=\int_{\R^{\{0,1\}^{\prime
d}}}\mu\Big(\bigcap_{{\bar{\vep}}\in\{0,1\}^d}T_{{\bar{t}}_{{\bar{\vep}}}}A_{I({\bar{\vep}})}\Big)\,dQ({\bar{t}})\\
& =\int_{\R^{\{0,1\}^{\prime
d}}}\mu\Big(\bigcap_{{\bar{\vep}}\in\{0,1\}^d}T_{{\bar{t}}_{I({\bar{\vep}})}}A_{{\bar{\vep}}}\Big)\,dQ({\bar{t}})\\
&=\int_{\R^{\{0,1\}^{\prime
d}}}\mu\Big(\bigcap_{{\bar{\vep}}\in\{0,1\}^d}T_{-({\bar{t}}_{(1,\ldots,1)}-{\bar{t}}_{I({\bar{\vep}})})}A_{{\bar{\vep}}}\Big)\,dQ({\bar{t}})\\
&=\int_{\R^{\{0,1\}^{\prime
d}}}\mu_{\bar{\theta}({\bar{t}})}\Big(\prod_{{\bar{\vep}}\in\{0,1\}^d}A_{{\bar{\vep}}}\Big)\,dQ({\bar{t}});
\end{align*}
in the last line we use the fact that
${\bar{t}}_{(1,\ldots,1)}-{\bar{t}}_{I({\bar{\vep}})}=0$ for
${\bar{\vep}}=(0,\ldots,0)$. In view of \eqref{eq:zbimin}, it
follows that
\[
\int_{\R^{\{0,1\}^{\prime d}}}\mu_{T_{\bar{t}}}\,dQ({\bar{t}}) =
\int_{\R^{\{0,1\}^{\prime
d}}}\mu_{T_{{\bar{t}}}}\,d\bar{\theta}_*Q({\bar{t}}).
\]
By the uniqueness of ergodic decomposition, we get
$\bar{\theta}_*(Q)=Q$. In this way we have proved the following
result.

\begin{pr}\label{krytNOWE}
Assume that $\mathcal{ T}=(T_t)_{t\in\R}$ is an ergodic and
aperiodic flow on $\xbm$. Assume that $\cT$ satisfies~(\ref{w1}).
If the measure $Q$ is not invariant under the map
$\bar{\theta}:\R^{\{0,1\}^{\prime d}}\to\R^{\{0,1\}^{\prime d}}$
then $(T_t)_{t\in\R}$ is not isomorphic to its inverse. In
particular, $(T_t)_{t\in\R}$ is not reversible.
\end{pr}

\begin{uw}\label{bbbbbb}
Suppose additionally that the flow $\ct$ is weakly mixing. Then
each its non-trivial factor is also weakly mixing, so it is
ergodic and aperiodic. For each such factor (\ref{w1}) is
evidently valid. It follows that the absence of isomorphism to the
inverse is inherited by non-trivial factors of $\ct$.
\end{uw}

Two particular cases follows. First, consider the case $d=2$. Then
the space $\R^{\{0,1\}^{\prime 2}}$ is identified with $\R^3$ by
the map $\R^{\{0,1\}^{\prime 2}}\ni
t\mapsto(t_{(1,1)},t_{(1,0)},t_{(0,1)})\in\R^3$. The the map
$\bar{\theta}$ is identified with $\theta:\R^3\to\R^3$,
$\theta(t,u,v)=(t,t-v,t-u)$.

\begin{wn}\label{krytNOWEd2}
Assume that $\mathcal{ T}=(T_t)_{t\in\R}$ is an ergodic and
aperiodic flow on $\xbm$. Assume moreover that
\[
\mu_{T_{r_n+q_n},T_{r_n},T_{q_n}}\mapsto
\int_{\R^3}\mu_{T_{-t},T_{-u},T_{-v}}\,dQ(t,u,v)
\]
for some probability measure $Q\in{\mathcal P}(\R^3)$. If the
measure $Q$ is not invariant under the map
$(t,u,v)\mapsto(t,t-v,t-u)$ then $\cT$ is not isomorphic to its
inverse.
\end{wn}

Now suppose that $\bar{q}_n=(q_n,\ldots,q_n)$. Then
$\bar{q}_n(\bar{\vep})=|\bar{\vep}|q_n$, where
$|\bar{\vep}|=\vep_1+\cdots+\vep_d$. Let us consider the maps
\begin{align*}
\varrho:\R^d\to\R^{\{0,1\}^{\prime d}},&\quad
\varrho\big((x_j)_{j=0}^{d-1}\big)=
\big(x_{d-|\overline{\vep}|}\big)_{\bar{\vep}\in\{0,1\}^{\prime
d}},\\
\theta:\R^d\to\R^d,&\quad
\theta(t_0,t_1,\ldots,t_{d-1})=(t_0,t_0-t_{d-1},\ldots,t_0-t_{1}).
\end{align*}
Then $\varrho\circ\theta=\bar{\theta}\circ\varrho$. Moreover, if
\begin{equation}\label{w1konkret}
\mu_{T_{dq_n},T_{(d-1)q_n},\ldots,T_{q_n}}\mapsto
\int_{\R^d}\mu_{T_{-t_0},T_{-t_1},\ldots,T_{-t_{d-1}}}\,dP(t_0,\ldots,t_{d-1})
\end{equation}
for some $P\in\mathcal{P}(\R^d)$ then \eqref{w1} holds  for a
measure $Q=\varrho_*(P)\in \mathcal{P}(\R^{\{0,1\}^{\prime d}})$.
Moreover, $\bar{\theta}_*(Q)=Q$ implies
$\varrho_*\theta_*(P)=\varrho_*(P)$, and hence $\theta_*(P)=P$. As
a conclusion from Proposition~\ref{krytNOWE} we obtain the
following.

\begin{wn}\label{krytNOWEd2prime}
Assume that $\mathcal{ T}=(T_t)_{t\in\R}$ is an ergodic and
aperiodic flow on $\xbm$. Assume that \eqref{w1konkret} is valid
for a  measure $P\in{\mathcal P}(\R^d)$. If the measure $P$ is not
invariant under the map $\theta:\R^d\to\R^d$ then $\cT$ is not
isomorphic to its inverse.
\end{wn}

Finally consider $d=2$.

\begin{wn}\label{kryt}
Assume that $\mathcal{ T}=(T_t)_{t\in\R}$ is an ergodic and
aperiodic flow on $\xbm$. Assume also that $$
\mu_{T_{2q_n},T_{q_n}}\to\int_{\R^2}\mu_{T_{-t},T_{-u}}\,dQ(t,u)$$
for some probability measure $Q$ on $\R^2$. If the measure $Q$ is
not invariant under $\theta(t,u)=(t,t-u)$ then $(T_t)_{t\in\R}$ is
not isomorphic to its inverse. In particular, if $(0,x)\in\R^2$ is
an atom of $Q$ but  $(0,-x)$ is not then $\cT$ is not isomorphic
to its inverse.
\end{wn}

In next three sections we will deal with special flows built over
irrational rotations on the circle. Such flows are always ergodic
and aperiodic (see Remark~\ref{uw:aperiodic}), so we can apply the
results of this section for proving the absence of isomorphism
with their inverses.

%See Remark~\ref{cztery} below to see that in some cases
%Corollary~\ref{kryt} does not decide the absence of isomorphism,
%while Corollary~\ref{krytNOWEd2prime} does.

\section{Non-reversible special flows over irrational rotations}\label{sec:vonn}

In this section we will discuss non-reversibility property for special
flows built over irrational rotations on the circle and under
piecewise absolutely continuous roof functions. For a real number $t$ denote
by $\{t\}$ its fractional part
and by $\|t\|$ its distance to the nearest integer number.

We call a function $f:\T\to\R$ \emph{ piecewise absolutely
continuous} if there exist $\beta_1,\ldots,\beta_K\in\T$ such that
$f|_{(\beta_j,\beta_{j+1})}$ is an absolutely continuous function
for $j=1,\ldots,K$ ($\beta_{K+1}=\beta_1$).
 Let
$d_j:=f_-(\beta_j)-f_+(\beta_j)$, where
$f_{\pm}(\beta)=\lim_{y\to\beta^\pm}f(y)$. Then the number
\begin{equation}\label{sumskok}
S(f):=\sum_{j=1}^Kd_j=\int_{\T}f'(x)dx
\end{equation}
is the \emph{ sum of jumps} of $f$. Without loss of generality we
can restrict ourselves to functions continuous on the right. Each
such function can be represented as $f=f_{pl}+f_{ac}$, where
$f_{ac}:\T\to\R$ is an absolutely continuous function with zero
mean and
\[f_{pl}(x)=\sum_{i=1}^Kd_i\{x-\beta_i\}+d.\]
In this section we will prove non-reversibility for
special flows $T^f$ built over almost every irrational
rotation $Tx=x+\alpha$ and under roof functions $f$ with $S(f)\neq0$.
Such flows are called von Neumann flows.

We need some auxiliary simple lemmas.

\begin{lm}\label{kf1} Let $(X_n)$ be a sequence of random
variables (each one defined on a probability space $(\Omega, \mathcal{
F},\mu)$) with values on $\R^d$. Assume that for $n\geq1$ we have
a partition $\{A_k^n: k=1,\ldots,K\}$ of $\Omega$ such that
$\mu(A^n_k)\to\delta_k$ when $n\to \infty$ for each
$k=1,\ldots,K$. Assume moreover that for each $k=1,\ldots,K$
\[
(X_n)_\ast(\mu_{A^n_k})\to P_k\;\;\text{when}\;\;n\to\infty
\]
weakly in the space of probability measures on $\R^d$ ($\mu_C$
stands for the relevant conditional measure: $\mu_C(A):=\mu(A\cap
C)/\mu(C)$). Then
\[
(X_n)_\ast(\mu)\to\sum_{k=1}^K\delta_kP_k.
\]
\end{lm}
\begin{proof}
Assume that $\phi:\R^d\to\R$ is continuous and bounded. Then
\begin{align*}
\int_{\R^d}\phi(t)\,d\left(\left(X_n\right)_\ast(\mu)\right)(t)&=
\int_{\Omega}\phi(X_n)\,d\mu=\sum_{k=1}^K \mu(A^n_k)\int_\Omega\phi(X_n)\,d\mu_{A^n_k}\\&\to
\sum_{k=1}^K\delta_k\int_{\R^d}\phi(t)\,dP_k(t).
\end{align*}
\end{proof}

\begin{lm}\label{kf2}
Let $(X_n)$ and $(C_n)$ be  sequences of random
variables (each one defined on a probability space $(\Omega_n, \mathcal{
F}_n,\mu_n)$) with values on $\R^d$. Assume that $(X_n)_\ast(\mu_n)\to
P$ and $C_n$ tends uniformly to the constant function $c\in\R^d$. Then
\[
(X_n+C_n)_\ast(\mu_n)\to (T_c)_\ast(P),
\]
where $T_c(x)=x+c$ in $\R^d$.
\end{lm}
\begin{proof} Fix $s\in\R^d$. By assumption
\[
\int_{\Omega}e^{2\pi is\cdot X_n}\,d\mu_n\to \int_{\R^d}e^{2\pi i
s\cdot t}\,dP(t).
\]
Moreover,
\begin{align*}
\Big|&\int_{\Omega}e^{2\pi is\cdot
(X_n+C_n)}\,d\mu_n-\int_{\R^d}e^{2\pi is\cdot (t+c)}\,dP(t)\Big|\\
&=\Big|\int_{\Omega}e^{2\pi is\cdot
(X_n+C_n-c)}\,d\mu_n-\int_{\R^d}e^{2\pi is\cdot t}\,dP(t)\Big|\\
&\leq\Big|\int_{\Omega}e^{2\pi is\cdot
X_n}\,d\mu_n-\int_{\R^d}e^{2\pi is\cdot
t}\,dP(t)\Big|+2\pi|s|\int_{\Omega}|C_n-c|\,dP(t).
\end{align*} It
follows that
\begin{align*}
\int_{\Omega}e^{2\pi is\cdot (X_n+C_n)}\,d\mu_n\to
\int_{\R^d}e^{2\pi is\cdot (t+c)}\,dP(t)= \int_{\R^d}e^{2\pi
is\cdot t}\,d((T_c)_\ast(P))(t),
\end{align*}
which completes the proof.
\end{proof}

The following lemma holds.

\begin{lm}\label{kf3}
Let $(X_n)$ be a sequence of random
variables (each one defined on a probability space
$(\Omega_n,\mathcal{ F}_n,\mu_n)$) with values on $\R^d$ such that
$(X_n)_\ast(\mu_n)\to P$. Assume that $A:\R^d\to\R^{d'}$ is
continuous. Then
$$
\left(A(X_n)\right)_\ast(\mu_n)\to A_\ast(P).$$
\end{lm}

\begin{uw}\label{kf4}\rm
Directly from the definition it follows that $\{x+y\}=\{x+\{y\}\}$
for each $x,y\in\R$. Moreover, whenever $a,b\in\T=[0,1)$, we have
$$
\{x+a-b\}-\{x-b\}=a-{\mathbf 1}_{[\{b-a\},b)}(x)$$ for $x\in\T$,
where $[\{b-a\},b)$ is understood as an interval on the circle (if
$d>e$ then $[d,e)=[d,1)\cup[0,e)$). Indeed,
$\{x+a-b\}-\{x-b\}=\{a+\{x-b\}\}-\{x-b\}$ and for $0\leq t<1$ we
have  $\{a+t\}-t= a$ if $0\leq t<1-a$ and $a-1$ for $1-a\leq t<1$.
\end{uw}

For any irrational number $\alpha=[0;a_1,a_2,\ldots)\in\T$ denote
by $(p_n/q_n)_{n\geq 0}$ the sequence of convergents  in continued
fraction expansion of $\alpha$ (see e.g.\ \cite{Kh} for basic
properties of continued fraction expansion of  $\alpha$).

\begin{lm}\label{kf0}
The set $\Lambda\subset[0,1)$ of those $\alpha$ irrational for
which for each $\vep>0$ there exists $0<\delta<\vep$ such that
$$q_{n_k}\|q_{n_k}\alpha\|\to\delta$$
along a subsequence $n_k=n_k(\vep)$ is of full Lebesgue measure.
\end{lm}
\begin{proof}
We have
$$
\frac12\frac{1}{a_{n+1}+1}< q_n\|q_n\alpha\|<\frac1{a_{n+1}}.$$
The result follows directly from the ergodicity of the Gauss map
$G:[0,1)\to[0,1)$ (see e.g.\ \cite{Ei-Wa}).
\end{proof}

Assume that $f(x)=\sum_{i=1}^Kd_i\{x-\beta_i\}+d$. Let
$Tx=x+\alpha$ and suppose that $\{q_n\alpha\}=\|q_n\alpha\|$. The
case where $\{q_n\alpha\}=1-\|q_n\alpha\|$ can be treated in a
similar way. We have
$$
f_0^{(q_n)}(T^{q_n}x)-f_0^{(q_n)}(x)=
f^{(q_n)}(T^{q_n}x)-f^{(q_n)}(x)=\sum_{j=0}^{q_n-1}\left( f\circ
T^{q_n}-f\right)(T^jx).
$$
Moreover, in view of Remark~\ref{kf4},
\begin{align*}
f&(T^{q_n}y)-f(y)=\sum_{i=1}^Kd_i\left( \{y+q_n\alpha-\beta_i\}-\{y-\beta_i\}\right)
\\&=\sum_{i=1}^Kd_i\left(
\{y+\|q_n\alpha\|-\beta_i\}-\{y-\beta_i\}\right)=
\sum_{i=1}^Kd_i\left(\|q_n\alpha\|-{\mathbf
1}_{[\beta_i-\|q_n\alpha\|,\beta_i)}(y)\right).
\end{align*}
Thus
\begin{equation}\label{kf5}
f^{(q_n)}(T^{q_n}x)-f^{(q_n)}(x)=q_n\|q_n\alpha\|\Big(\sum_{i=1}^Kd_i\Big)
-\sum_{i=1}^Kd_i\sum_{j=0}^{q_n-1}{\mathbf
1}_{[\beta_i-\|q_n\alpha\|,\beta_i)}(x+j\alpha).\end{equation}
Moreover, since $[s,s+\|q_n\alpha\|)$ is the base of a Rokhlin
tower of height $q_{n+1}$, we have
\begin{equation}\label{kf6}
\sum_{j=0}^{q_n-1}{\mathbf
1}_{[\beta_i-\|q_n\alpha\|,\beta_i)}(x+j\alpha)=0\quad\text{ or}\quad 1.
\end{equation}

Given $n\geq1$ and $\epsilon\in\{0,1\}^K$, taking into
account~(\ref{kf6}), set
\begin{equation}\label{kf7}
A_\epsilon^n=\left\{x\in\T:\:\sum_{j=0}^{q_n-1}
{\mathbf1}_{[\beta_i-\|q_n\alpha\|,\beta_i)}(x+j\alpha)=
\epsilon_i\;\; \text{for}\;\;i=1,\ldots, K\right\}.\end{equation}
Then, in view of~(\ref{kf5}), for $x\in A^n_\epsilon$ we have
\begin{equation}\label{kf8}
f^{(q_n)}(T^{q_n}x)-f^{(q_n)}(x)=q_n\|q_n\alpha\|S(f)-C_\epsilon,\end{equation}
where
$C_\epsilon=\sum_{i=1}^Kd_i\epsilon_i$ (note that $C_{\underline{0}}=C_{(0,\ldots,0)}=0$).

Suppose that the roof function $f:\T\to\R$ is a piecewise absolutely
continuous function and let us decompose $f=f_{pl}+f_{ac}$.
Suppose that $q_n\|q_n\alpha\|\to\delta>0$ and
$\mu(A^n_\epsilon)\to p_\epsilon$ for $\epsilon\in\{0,1\}^K$ (sets $A^n_\epsilon$
are defined accordingly to the function $f_{pl}$).
By Koksma-Denjoy inequality (see e.g.\ \cite{Ku-Ni}), $\|(f_{pl})_0^{(q_m)}\|_{\sup}\leq\operatorname{Var}f$,
thus the sequence $\big(((f_{pl})_0^{(q_n)})_\ast(\mu_{A^n_\epsilon})\big)_{n\geq 0}$
of distributions is uniformly tight.
By
passing to a further subsequence, if necessary, we can also assume
that
\begin{equation}\label{kf9}
\left((f_{pl})_0^{(q_n)}\right)_\ast(\mu_{A^n_\epsilon})\to
P_\epsilon\;\;\text{when}\;n\to\infty.
\end{equation}
Recall that (see e.g.\ \cite{He})
\begin{equation}\label{eq:aczbie}
\|f_{ac}^{(q_n)}\|_{sup}\to 0.
\end{equation}
Set $X_n=\left(f_0^{(2q_n)},f_0^{(q_n)}\right):\T\to\R^2$. Then
$$
\left(f_0^{(2q_n)},f_0^{(q_n)}\right)=\left(
2(f_{pl})_0^{(q_n)}+f_{pl}^{(q_n)}\circ T^{q_n}-f_{pl}^{(q_n)},(f_{pl})_0^{(q_n)}\right)+
\left(f_{ac}^{(2q_n)},f_{ac}^{(q_n)}\right).$$ In view
of~(\ref{kf8}), for $x\in A_\epsilon^n$
\begin{equation}\label{kf10}
X_n=\left(2(f_{pl})_0^{(q_n)},(f_{pl})_0^{(q_n)}\right)+\left(q_n\|q_n\alpha\|S(f)
-C_\epsilon,0\right)+\left(f_{ac}^{(2q_n)},f_{ac}^{(q_n)}\right).
\end{equation}
Let $A:\R\to\R^2$, $Ax=(2x,x)$. Thus
$$
Y_n:=\left(2(f_{pl})_0^{(q_n)},(f_{pl})_0^{(q_n)}\right)=A\circ
f_0^{(q_n)},$$ so by Lemma~\ref{kf3} and~(\ref{kf9}),
\begin{equation}\label{kf11}
\left(Y_n\right)_\ast(\mu_{A_\epsilon^n})\to A_\ast(P_\epsilon).
\end{equation}
Since
$X_n=Y_n+(q_n\|q_n\alpha\|S(f)-C_\epsilon,0)+\left(f_{ac}^{(2q_n)},f_{ac}^{(q_n)}\right)$
on $A^n_\epsilon$, $(f_{ac}^{(2q_n)},f_{ac}^{(q_n)})$ uniformly
tends to zero (see \eqref{eq:aczbie}) and
$q_n\|q_n\alpha\|\to\delta$, in view of
Lemma~\ref{kf2}~\footnote{We apply the lemma for
$\mu_n=\mu_{A^n_{\epsilon}}$, $X_n=Y_n$ and
$C_n=(q_n\|q_n\alpha\|S(f)-C_\epsilon,0)+(f_{ac}^{(2q_n)},f_{ac}^{(q_n)})$.},
\begin{equation}\label{kf12}
\left(X_n\right)_\ast(\mu_{A^n_\epsilon})\to\left(T_{(\delta
S(f)-C_\epsilon,0)}\right)_\ast A_\ast(P_\epsilon).
\end{equation}
Therefore, by Lemma~\ref{kf1},
\begin{equation}\label{kf13}
\left(X_n\right)_\ast(\mu)\to\sum_{\epsilon\in\{0,1\}^K}
p_\epsilon\left(T_{(\delta S(f)-C_\epsilon,0)}\right)_\ast
A_\ast(P_\epsilon)=:P.
\end{equation}
On  the other hand (see Proposition~\ref{main1}),
$\lim_{n\to\infty}\left(X_n\right)_\ast(\mu)=P$, so
\begin{equation}\label{kf13bis}
\sum_{\epsilon\in\{0,1\}^K}
p_\epsilon\left(T_{(\delta S(f)-C_\epsilon,0)}\right)_\ast
A_\ast(P_\epsilon)=P.\end{equation}

\begin{tw}\label{kf15} If $\alpha\in \Lambda$ (see
Lemma~\ref{kf0}) and $S(f)\neq0$ then the special flow $T^f$ is
non-reversible, in fact $T^f$ is not isomorphic to its inverse.\end{tw}
\begin{proof} Take $\delta>0$ so that $K\delta<1$,
$q_n\|q_n\alpha\|\to \delta$ (by passing to a subsequence, if
necessary) and
$$
\delta<\min\{|C_\epsilon|>0,
\epsilon\in\{0,1\}^K\}/(2|S(f)|).$$

Suppose now that the special flow $T^f$ is reversible.  By
Proposition~\ref{main1} and \ref{kryt}, $\theta_\ast P=P$, where
$\theta(t,u)=(t,t-u)$. Using~(\ref{kf13bis}), since
\[
\theta\circ T_{(c,0)}\circ A=T_{(-c,0)}\circ A\circ T_c,
\] we have
\[
\theta_\ast P=\sum_{\epsilon\in\{0,1\}^K}p_\epsilon\left(
T_{(-\delta S(f)+C_\epsilon,0)}\right)_\ast A_\ast\left(
T_{\delta S(f)-C_\epsilon}\right)_\ast(P_\epsilon).
\]
Each
measure of the form $\left(T_{(c,0)}\right)_\ast A_\ast P'$ (with
$P'$ a probability on $\R$) is concentrated on the set
$$
R_c:=\{(2x+c,x):\:x\in\R\}.
$$
Clearly, $R_c\cap
R_{c'}=\emptyset$ for $c\neq c'$. If for some
$\epsilon\in\{0,1\}^K$, $p_\epsilon>0$ and
$\delta S(f)-C_\epsilon\neq0$, since $\theta_\ast P=P$, there must
exist $\epsilon'\in\{0,1\}^K$ such that
$$
p_{\epsilon'}>0\;\;\text{and}\;\;-\delta S(f)+C_{\epsilon}=\delta
S(f)-C_{\epsilon'},
$$
whence
\begin{equation}\label{kf14}
C_\epsilon+C_{\epsilon'}=2\delta S(f).\end{equation}
Then
\begin{align*}A_{\underline{0}}^n&=\{x\in\T:\:\sum_{j=0}^{q_n-1}
{\mathbf1}_{[\beta_i-\|q_n\alpha\|,\beta_i)}(x+j\alpha)
=0\;\;\text{for}\;i=1,\ldots,K\}\\
&=\{x\in\T:\:(\forall 0\leq
j<q_n)(\forall 1\leq i\leq K)\;\;x+j\alpha\notin
[\beta_i-\|q_n\alpha\|,\beta_i)\}\\
&
=\bigcap_{j=0}^{q_n-1}\bigcap_{i=1}^K\left(\T\setminus T^{-j}
[\beta_i-\|q_n\alpha\|,\beta_i)\right)= \T
\setminus\bigcup_{j=0}^{q_n-1}\bigcup_{i=1}^K
T^{-j}[\beta_i-\|q_n\alpha\|,\beta_i).
\end{align*}
It follows that
$$
1-\mu(A^n_{\underline{0}})=\mu\left(\bigcup_{j=0}^{q_n-1}\bigcup_{i=1}^K
T^{-j}[\beta_i-\|q_n\alpha\|,\beta_i)\right)\leq
Kq_n\|q_n\alpha\|,$$ so $\mu(A^n_{\underline{0}})\geq1-Kq_n\|q_n\alpha\|$ and
therefore
$$
\liminf\mu(A^n_{\underline{0}})\geq 1-K\delta>0.$$ Thus
$p_{\underline{0}}>0$ and $\delta S(f)-C_{\underline{0}}=\delta
S(f)\neq0$, it follows from~(\ref{kf14}) (applied to
$\epsilon=\underline{0}$)  that there exists
$\epsilon\in\{0,1\}^K$ such that
$$
C_\epsilon+C_{\underline{0}}=2\delta S(f),$$ whence
$\frac{|C_\epsilon|}{2|S(f)|}=\delta$ which  yields a
contradiction to the definition of $\delta$.
\end{proof}

\subsection{Non-reversibility in the affine case}
Given a special flow $T^f$  for which $\int_Xf\,d\mu=1$,
$T^{r_n},T^{q_n}\to Id$, assume that
$$
(f_0^{(r_n+q_n)},f_0^{(r_n)}, f_0^{(q_n)})\to P$$  with
$\|f_0^{(q_n)}\|_{L^2}, \|f_0^{(r_n)}\|_{L^2}\leq C$. In view of
Proposition~\ref{main1} and Corollary~\ref{krytNOWEd2},  we have:
$$
\mu^{f}_{T^f_{r_n+q_n},T^f_{r_n},T^f_{q_n}}\to \int_{\R^3}
\mu^f_{T^f_{-t},T^f_{-u},T^f_{-v}}\,dP(t,u,v).
$$
For each $(a,b,c)\in\R^3$ we have
$$
\widehat{P}(a,b,c)=\lim_{n\to\infty}\int_Xe^{2\pi i(af_0^{(r_n+q_n)}(x)
+bf_0^{(r_n)}(x)+cf_0^{(q_n)}(x))}\,d\mu(x).$$
Denote $\theta(t,u,v)=(t,t-v,t-u)$ and note that
\begin{equation}\label{uzu1}
\text{$\theta_\ast(P)=P$ if and only if
$\widehat{P}(a,b,c)=\widehat{P}(a+b+c,-c,-a)$.}\end{equation}
Moreover
\begin{align}\label{uzu2}\begin{aligned}
\widehat{P}&(a+b+c,-c,-b)=\lim_{n\to\infty}\int_{X}e^{2\pi
i\big((a+b+c)f_0^{(r_n+q_n)}
-cf_0^{(r_n)}-bf_0^{(q_n)}\big)}\,d\mu\\
&=\lim_{n\to\infty}\int_{X}e^{2\pi i\big(af_0^{(r_n+q_n)}+
bf_0^{(r_n)}+cf_0^{(q_n)}+(b+c)\left(f^{(r_n)}\circ
T^{q_n}-f^{(r_n)}\right)\big)}\,d\mu,
\end{aligned}
\end{align}
note that $f^{(r_n)}\circ T^{q_n}-f^{(r_n)}=f^{(q_n)}\circ T^{r_n}-f^{(q_n)}$.

%Similarly, supposing that $(f_0^{(2q_n)},f_0^{(q_n)})_\ast\to P$
%and setting $\theta(t,u)=(t,t-u)$, we obtain that,
%\begin{equation}\label{uzu1prime}
%\theta_\ast(P)=P\text{ if and only if }
%\widehat{P}(a,b)=\widehat{P}(a+b,-b)\text{ for each }a,b\in\R.
%\end{equation}
%Moreover
%\begin{align}\label{uzu2prime}\begin{aligned}
%\widehat{P}(a+b,-b)&=\lim_{n\to\infty}\int_{X} e^{2\pi
%i\big((a+b)f_0^{(2q_n)}
%-bf_0^{(q_n)}\big)}\,d\mu\\
%&=\lim_{n\to\infty}\int_{X}e^{2\pi i\big(af_0^{(2q_n)}+
%bf_0^{(q_n)}+b\left(f^{(q_n)}\circ
%T^{q_n}-f^{(q_n)}\right)\big)}\,d\mu.
%\end{aligned}\end{align}

%Although, (\ref{uzu2}) looks stronger than~(\ref{uzu2prime}),  we
%have not been able to find an example in which~(\ref{uzu2prime})
%decides about non-reversibility, while~(\ref{uzu2}) does
%not~\footnote{Note, for example, that if $f_0^{(q_n)}\circ
%T^{q_n}-f_{0}^{(q_n)}\to 0$ in measure that when $r_n=iq_n$ and we
%consider $jq_n$ (instead of $q_n$), where $i\neq j$ ($i,j\neq0$)
%we also have $ f^{(iq_n)}\circ T^{jq_n}-f^{(iq_n)}\to 0$ in
%measure; indeed, $f^{(iq_n)}\circ
%T^{jq_n}-f^{(iq_n)}=\sum_{k=0}^{i-1}\left(\sum_{l=0}^{j-1}\left(f_0^{(q_n)}\circ
%T^{q_n}-f_0^{(q_n)}\right)\circ T^{lq_n}\right)\circ T^{kq_n}$.},
%see also Remark~\ref{dlacz} below.

Consider now the affine case
$$
f(x)=x+c,\;Tx=x+\alpha$$ with $f_0(x)=x-\frac12$ and
$\alpha=[0;a_1,a_2,\ldots]$. Our aim is to get a  larger set of
$\alpha$'s than those resulting  from Theorem~\ref{kf15} for which
the special flow $T^f$ is not isomorphic to its inverse..

\begin{pr}\label{jom}
If there exists a subsequence of denominators $(q_{k_n})_{n\geq
1}$ of $\alpha$ such that
$q_{k_n+1}\|q_{k_n}\alpha\|\to\kappa\in(1/2,1)$ then $T^f$ is not
isomorphic to its inverse.
\end{pr}

\begin{proof}
To simplify notation we will write $n$ instead of $k_n$.

Suppose that $T^f$ is isomorphic to its inverse. In view of
Corollary~\ref{krytNOWEd2} and \eqref{uzu1}, if
$(f_0^{(q_{n+1}+q_n)},f_0^{(q_{n+1})}, f_0^{(q_n)})\to P$ then
$\widehat{P}(a,b,c)=\widehat{P}(a+b+c,-c,-a)$ for each
$a,b,c\in\R$. We have
\[f_0^{(q_{n+1})}(T^{q_n}x)-f_0^{(q_{n+1})}(x)=\sum_{j=0}^{q_{n+1}-1}(f(T^jx+q_n\alpha)-f(T^jx))\]
and, by Remark~\ref{kf4}, $f(y+q_n\alpha)-f(y)\in
\pm\|q_n\alpha\|+\Z$ for any $y\in\T$. Thus
\begin{align*}
f_0^{(q_{n+1})}(T^{q_n}x)-f_0^{(q_{n+1})}(x)=\pm
q_{n+1}\|q_n\alpha\|+M_n(x)\text{ with  }M_n(x)\in\Z.
\end{align*}
It follows that
\[ e^{2\pi il(f_0^{(q_{n+1})}(T^{q_n}x)-f_0^{(q_{n+1})}(x))}=e^{\pm 2\pi
il q_{n+1}\|q_n\alpha\|}\to e^{\pm 2\pi i l\kappa}\] for each
integer $l$. By our standing assumption, $e^{4\pi i \kappa}\neq1$.
Taking into account~(\ref{uzu2}) we obtain that
\[\widehat{P}(a,b,c)=\widehat{P}(a+b+c,-a,-c)=e^{\pm 2\pi i
(b+c)\kappa}\widehat{P}(a,b,c)\text{ whenever } b+c\in\Z,\] hence
\begin{equation}\label{sprz}
\widehat{P}(1,-1,-1)=0.
\end{equation}
On the other hand, the function
\begin{align*}
f_0^{(q_{n+1}+q_n)}(x)-f_0^{(q_{n+1})}(x)-f_0^{(q_n)}(x)
&=f_0^{(q_{n+1})}(T^{q_n}x)-f_0^{(q_{n+1})}(x)\\ &=\pm
q_{n+1}\|q_n\alpha\|+M_n(x),
\end{align*}
so
\begin{align*}
|\widehat{P}(1,-1,-1)|&=\lim_{n\to\infty}\Big|\int_{\T} e^{2\pi
i\big(f_0^{(q_{n+1}+q_n)}(x)-f_0^{(q_{n+1})}(x)-f_0^{(q_n)}(x)\big)}\,dx\Big|\\
& =\lim_{n\to\infty}\big| e^{\pm 2\pi
iq_{n+1}\|q_n\alpha\|}\big|=1.
\end{align*}
This implies $|\widehat{P}(1,-1,-1)|=1$ which gives rise to a
contradiction to~(\ref{sprz}).
\end{proof}

\begin{uw}
Since $\alpha=\frac{p_{n+1}+p_n
G^{n+1}(\alpha)}{q_{n+1}+q_nG^{n+1}(\alpha)}$ (see e.g.\
\cite{Ei-Wa}), we have
\[\frac{1}{1+\frac{1}{a_{n+1}}\frac{1}{a_{n+2}}}<
q_{n+1}\|q_n\alpha\|<\frac{1}{1+\frac{1}{a_{n+1}+1}\frac{1}{a_{n+2}+1}}\]
and
\[\frac{1}{1+\frac{1}{a_{n+1}+\frac{1}{a_n+1}}\frac{1}{a_{n+2}+\frac{1}{a_{n+3}+1}}}<
q_{n+1}\|q_n\alpha\|<\frac{1}{1+\frac{1}{a_{n+1}+\frac{1}{a_n}}\frac{1}{a_{n+2}+\frac{1}{a_{n+3}}}}.\]
Therefore
\[q_{k_n+1}\|q_{k_n}\alpha\|\to 1 \quad\Leftrightarrow\quad a_{k_n+1}+a_{k_n+2}\to+\infty\]
and
\[q_{k_n+1}\|q_{k_n}\alpha\|\to 1/2 \quad\Leftrightarrow\quad a_{k_n+1}=a_{k_n+2}=1\text{ and }
a_{k_n},a_{k_n+3}\to+\infty.\] The set of excluded irrational
rotations $E\subset\T$ in Theorem~\ref{jom} consists of all
irrational $\alpha$ for which the set of limit points of the
sequence $(q_{n+1}\|q_n\alpha\|)_{n\geq 1}$ is $\{1/2,1\}$.
Therefore $\alpha\in E$ if and only if the set of limit points of
the sequence $(a_n+a_{n+1})_{n\geq 1}$ is $\{2,+\infty\}$ and if
there exists a subsequence $(a_{k_n})_{n\geq 1}$ such that
$a_{k_n}=a_{k_n+1}=1$ then $a_{k_n-1},a_{k_n+2}\to+\infty$.
\end{uw}

\begin{uw}\label{dlacz}
A natural question arises  whether we could apply~(\ref{uzu2})
choosing a sequence of pairs of denominators, say we consider
$q_{l_n},q_{k_n}$, $n\geq 1$ when $\alpha$ is Liouville in the
sense that the sequence of partial quotients tends to infinity.
This approach seems to fail whenever $f$ is of bounded variation .
Indeed,
$$
\left|f_0^{(q_{l_n})}\circ T^{q_{k_n}}-f_0^{(q_{l_n})}\right|\leq
\|q_{k_n}\alpha\|{\rm Var} f_0^{(q_{l_n})}\leq
q_{k_n}\|q_{l_n}\alpha\|{\rm Var} f_0$$ and
$q_{k_n}\|q_{l_n}\alpha\|\to0$  whenever $\alpha$ is a Liouville
number.
\end{uw}

\section{Piecewise polynomial roof functions}\label{polyroof}
Let $r\geq 1$ be an odd number and let $0<\beta<1$. In this
section we will study the problem of isomorphism to the inverse
for special flows $T^f$ built over irrational rotations
$Tx=x+\alpha$ on the circle and under $C^{r-1}$-function which are
polynomials after restriction to intervals $[0,\beta)$ and
$[\beta,1)$.

Let us consider a $C^{r-1}$-function $f:\T\to\R_+$ such that
$D^{r-1}f$ is a function linear on both intervals $[0,\beta)$ and
$[\beta,1)$ with slopes $1-\beta$ and $-\beta$ respectively.
Therefore, $D^{r-1}f$ is an absolutely continuous function whose
derivative is equal to
\[D^{r}f=(1-\beta)1_{[0,\beta)}-\beta 1_{[\beta,1)}.\]
Thus $f$ restricted to each interval  $[0,\beta)$ and $[\beta,1)$
is a polynomial of degree $r$ with leading coefficients
$(1-\beta)/r!$ and $-\beta/r!$ respectively. Since $D^{r-1}f$ is
absolutely continuous and $D^rf$ is of bounded variation, the
Fourier coefficients satisfy
$\widehat{f}(n)=\operatorname{O}(1/|n|^{r+1})$.

If the irrational number $\alpha$ is slowly approximated by
rationals, more precisely
$\liminf_{n\to\infty}q_n^{r+1-\epsilon}\|q_n\alpha\|>0$ for some
$\epsilon>0$, then $f$ is cohomologous to a constant function, so
the special flow $T^f$ is isomorphic to its inverse. In this
section we deal with rotations satisfying
\[0<\limsup_{n\to\infty}q_n^{r+1}\|q_n\alpha\|.\]
Note that we can not expect non-reversibility of $T^f$ for any
$\beta$. Indeed, if $\beta\in\Z\alpha\cup(\Z\alpha+1/2)$ then
$D^{r-1}f$ is cohomologous to either zero function or a function
which is $x\mapsto 1-x$ invariant. Since $r-1$ is even, $f$ is
also cohomologous to either a constant function or a function
which is $x\mapsto 1-x$ invariant. In both cases $T^f$ is
isomorphic to its inverse.

The  main result of this section (Theorem~\ref{thm:poly})
establishes some technical conditions on $\alpha$ that gives
non-isomorphism of $T^f$ to its inverse for almost every choice of
$\beta\in\T$.

\begin{uw}\label{uw:pochdys}
In the proof of Theorem~\ref{thm:poly} we will use simple
properties of the following standard difference operator. For any
$h>0$ let us consider the difference operator
\[\Delta_h:\R^{[a,b]}\to\R^{[a,b-h]},\quad \Delta_hg(x)=g(x+h)-g(x).\]
For every natural $r$ denote by
$\Delta^r_h:\R^{[a,b]}\to\R^{[a,b-rh]}$ the $r$-th iteration of
the operator $\Delta_h$. By induction and using ${r \choose
k-1}+{r \choose k}={r+1 \choose k}$, we have the following
standard formula
\begin{equation}
\Delta^r_hg(x)=\sum_{k=0}^r(-1)^{r-k}{r \choose
k}g(x+kh)\quad\text{ for }\quad x\in[a,b-rh].
\end{equation}
Moreover, if $g$ is a polynomial function of degree $r$ with
leading coefficient $a_r$ then $\Delta^r_hg$ is a constant
function equal to $r!a_rh^r$.
\end{uw}

\begin{tw}\label{thm:poly}
Suppose that $\alpha$ is an irrational number for which there
exists a subsequence of denominators $(q_{k_n})_{n\geq 1}$ such
that
\begin{equation}\label{eq:diof}
q_{k_n}^{r+1}\|q_{k_n}\alpha\|\to\kappa\in\left(0,\frac{1}{2(r+1)}\right).
\end{equation}
Then for almost every $\beta\in\T$ the special flow $T^f$ is not
isomorphic to its inverse.
\end{tw}

\begin{proof}
By Weyl's theorem (see Theorem 4.1 in \cite{Ku-Ni}), for almost
every $\beta\in\T$ the sequence $(\{q_{k_n}\beta\})_{n\geq 1}$ is
uniformly distributed in $[0,1)$. It follows that there exists
$\gamma\in(0,1)\setminus\{1/2\}$ and a subsequence $(q_{k_{l_n}})$
such that $\{q_{k_{l_n}}\beta\}\to\gamma$. To simplify notation we
will write $n$ instead of $k_{l_n}$. Assume also that
$\{q_n\alpha\}=\|q_n\alpha\|$. The case where
$\{q_n\alpha\}=1-\|q_n\alpha\|$ can be treated in a similar way.

Suppose that $T^f$ is isomorphic to its inverse. Since
$\widehat{f}(n)=\operatorname{O}(1/|n|^{r+1})$, in view of
Corollary~3.1 in \cite{Aa-Le-Ma-Na}, the sequence
$\big(f_0^{(q^{r+1}_n)}\big)_{n\geq 1}$ is bounded in $L^2$.
Therefore, by passing to a further subsequence, if necessary, we
 can assume that
\[\Big(f_0^{((r+1)q^{r+1}_n)},f_0^{(rq^{r+1}_n)},\ldots,
f_0^{(q^{r+1}_n)}\Big)_*(\mu)\to P\quad\text{ in
}\quad\mathcal{P}(\R^{r+1}).\] Since
$D^rf=(1-\beta)1_{[0,\beta)}-\beta 1_{[\beta,1)}$, by the
Koksma-Denjoy inequality\footnote{$D^*_{q_n}$ is the discrepancy
of the sequence $\{0,\alpha,\ldots,(q_n-1)\alpha\}$.} (see
\cite{Ku-Ni}),
\[|(D^rf)^{(q_n)}(x)|\leq q_n D^*_{q_n}(\alpha)
\operatorname{Var}_{[0,1)}(D^rf)\leq 1+\frac{q_n}{q_{n+1}}.\]
The function $(D^rf)^{(q_n)}$ takes values only in the set
$\Z-q_n\beta$ and $\{q_{n}\beta\}\to\gamma\in(0,1)$. Since
$q_n/q_{n+1}\to 0$, it follows that for all $n$ large enough
$(D^rf)^{(q_n)}(x)$ is equal to $1-\{q_n\beta\}$ or
$-\{q_n\beta\}$ for every $x\in\T$.

Let
\[A_n:=\T\setminus\Big(\bigcup_{j=0}^{q_n-1}T^{-j}[1-(r+1)q^r_n\|q_n\alpha\|,1]\cup
\bigcup_{j=0}^{q_n-1}T^{-j}[\beta-(r+1)q^r_n\|q_n\alpha\|,\beta]\Big).\]
Thus, by \eqref{eq:diof},
\begin{equation}\label{eq:miaraan}
\mu(A_n)\geq 1-2(r+1)q_nq_n^{r}\|q_n\alpha\|\to
1-2(r+1)\kappa>1/2.
\end{equation}
Moreover, for every $x\in A_n$ the point $0$ and $\beta$ do not
belong to any interval $T^j[x,T^{(r+1)q_n^{r+1}}x]$ for all $0\leq
j<q_n$. It follows that $(D^rf)^{(q_n)}$ on
$[x,T^{(r+1)q_n^{r+1}}x]$ is constant and equal $s-\{q_n\beta\}$
for some $s\in\{0,1\}$. Therefore, for every $y\in
[x,T^{rq_n^{r+1}}x]$ and $0\leq j<q^r_n$ we have
\[T^{jq_n}y\in[T^{jq_n}x,T^{jq_n+rq_n^{r+1}}x]\subset[x,T^{(r+1)q_n^{r+1}}x],\]
so
\[(D^rf)^{(q^{r+1}_n)}(y)=\sum_{j=0}^{q^r_n-1}(D^rf)^{(q_n)}(T^{jq_n}y)=q^r_n(s-\{q_n\beta\}).\]
Therefore, for every $x\in A_n$ there exists $s=s(x)\in\{0,1\}$
such that $D^r(f_0^{(q^{r+1}_n)})=q^r_n(s-\{q_n\beta\})$ on
$[x,x+rq_n^r\|q_n\alpha\|]$, so $f_0^{(q^{r+1}_n)}$ restricted to
$[x,x+rh]$, with $h:=q_n^r\|q_n\alpha\|$, is a polynomial of
degree $r$ with leading coefficient $q^r_n(s-\{q_n\beta\})/r!$. In
view of Remark~\ref{uw:pochdys}, it follows that
\begin{align*}
\sum_{k=0}^r&(-1)^{r-k}{r \choose
k}f_0^{(q^{r+1}_n)}(T^{kq_n^{r+1}}x)= \sum_{k=0}^r(-1)^{r-k}{r
\choose
k}f_0^{(q^{r+1}_n)}(x+kh)\\
&=\Delta_h^rf_0^{(q^{r+1}_n)}(x)=q^r_n(s(x)-\{q_n\beta\})h^r=
(s(x)-\{q_n\beta\})\big(q^{r+1}_n\|q_n\alpha\|\big)^r.
\end{align*}
Moreover, {\allowdisplaybreaks
\begin{align*}
\sum_{k=1}^{r+1}&(-1)^{r+1-k}{r+1 \choose
k}f_0^{(kq^{r+1}_n)}(x)\\&=\sum_{k=1}^{r+1}(-1)^{r+1-k}{r+1
\choose
k}\sum_{l=0}^{k-1}f_0^{(q^{r+1}_n)}(T^{lq_n^{r+1}}x)\\
&\
=\sum_{l=0}^{r}f_0^{(q^{r+1}_n)}(T^{lq_n^{r+1}}x)\sum_{k=l+1}^{r+1}(-1)^{r+1-k}{r+1
\choose k}\\&
=\sum_{l=0}^{r}f_0^{(q^{r+1}_n)}(T^{lq_n^{r+1}}x)\sum_{k=l+1}^{r+1}(-1)^{r+1-k}\Big({r
\choose k}+{r \choose k-1}\Big)
\\&
=\sum_{l=0}^{r}(-1)^{r-l}{r \choose
l}f_0^{(q^{r+1}_n)}(T^{lq_n^{r+1}}x)=(s(x)-\{q_n\beta\})\big(q^{r+1}_n\|q_n\alpha\|\big)^r.
\end{align*}}
For $s=0,1$ set
$c^n_s:=(s-\{q_n\beta\})\big(q^{r+1}_n\|q_n\alpha\|\big)^r$ and
let
\[A_n^s=\Big\{x\in A_n:\sum_{k=1}^{r+1}(-1)^{r+1-k}{r+1 \choose
k}f_0^{(kq^{r+1}_n)}(x)=c^n_s\Big\}.\] By passing to a further
subsequence, if necessary, we can assume that
\begin{equation}\label{eq:mia1}
\mu(A_n^0)\to\nu_0,\quad \mu(A_n^1)\to\nu_1,
\end{equation}
\begin{equation}\label{eq:mia2}
\big(f_0^{(rq_n^{r+1})},\ldots,f_0^{(q_n^{r+1})}\big)_*(\mu_{A_n^0})\to
P_0,\quad \quad
\big(f_0^{(rq_n^{r+1})},\ldots,f_0^{(q_n^{r+1})}\big)_*(\mu_{A_n^1})\to
P_1
\end{equation} in $\mathcal{P}(\R^r)$ and
\begin{equation}\label{eq:mia3}
\big(f_0^{((r+1)q_n^{r+1})},\ldots,f_0^{(q_n^{r+1})}\big)_*(\mu_{A_n^c})\to
P_2\text{ in }\mathcal{P}(\R^{r+1}).
\end{equation} Since
$A_n=A_n^0\cup A_n^1$, by \eqref{eq:miaraan}, we have
\[\nu_0+\nu_1\geq 1-2(r+1)\kappa>1/2.\]
Set $\nu_2:=1-\nu_0-\nu_1$. Let us consider the following maps:
\begin{align*}
\theta:\R^{r+1}\to\R^{r+1},&\quad \theta(x_0,x_1,\ldots,x_r)=(x_0,x_0-x_r,\ldots,x_0-x_1) \\
A:\R^r\to\R^{r+1},&\quad
A(x_1,x_2,\ldots,x_r)=\left(\sum_{k=1}^{r}(-1)^{k+1}{r+1 \choose
k}x_k,x_1,\ldots,x_r\right),\\
R_c:\R^{r+1}\to\R^{r+1},&\quad
R_c(x_0,x_1,\ldots,x_r)=(x_0+c,x_1,\ldots,x_r),\\
B_c:\R^{r}\to\R^{r},&\quad
B_c(x_1,\ldots,x_r)=\left(\sum_{k=1}^{r}(-1)^{k+1}{r+1 \choose
k}x_k-x_{r+1-l}+c\right)_{l=1}^r.
\end{align*}
Then
\begin{equation}\label{eq:miksty}
\theta\circ R_c\circ A=R_{-c}\circ A\circ B_c.
\end{equation}
Indeed, equation \eqref{eq:miksty} is valid directly for last $r$
coordinates.  The zero coordinate of the LHS of \eqref{eq:miksty}
is
\[\text{LHS}_0:=\sum_{k=1}^{r}(-1)^{k+1}{r+1 \choose
k}x_k+c.\] The zero coordinate of the RHS of \eqref{eq:miksty} is
\[
\text{RHS}_0:=-c+\sum_{l=1}^{r}(-1)^{l+1}{r+1 \choose
l}\left(\sum_{k=1}^{r}(-1)^{k+1}{r+1 \choose
k}x_k-x_{r+1-l}+c\right).
\]
Since $r$ is odd,
\[\sum_{l=1}^{r}(-1)^{l+1}{r+1 \choose
l}=\sum_{l=0}^{r+1}(-1)^{l+1}{r+1 \choose l}+2=-(1-1)^{r+1}+2=2,\]
thus
\begin{align*}
\text{RHS}_0&=-c+2c+2\sum_{k=1}^{r}(-1)^{k+1}{r+1 \choose k}x_k
-\sum_{l=1}^{r}(-1)^{l+1}{r+1 \choose l}x_{r+1-l}\\
&=c+\sum_{k=1}^{r}(-1)^{k+1}{r+1 \choose k}x_k=\text{LHS}_0,
\end{align*}
which completes the proof of \eqref{eq:miksty}.

Since
\begin{multline*}\Big(f_0^{((r+1)q_n^{r+1})},f_0^{(rq_n^{r+1})},\ldots,f_0^{(q_n^{r+1})}\Big)=
\left(\sum_{k=1}^{r+1}(-1)^{k}{r+1 \choose
k}f_0^{(kq^{r+1}_n)},0,\ldots,0\right)\\+
\left(\sum_{k=1}^{r}(-1)^{k+1}{r+1 \choose
k}f_0^{(kq^{r+1}_n)},f_0^{(rq_n^{r+1})},\ldots,f_0^{(q_n^{r+1})}\right),
\end{multline*}
by the definitions of maps $A$, $R_c$ and the set $A_n^s$, it
follows that for any $x\in A_n^s$ and $s=0,1$ we have
\[\big(f_0^{((r+1)q_n^{r+1})}(x),\ldots,f_0^{(q_n^{r+1})}(x)\big)=
R_{c^n_s}\circ
A\big(f_0^{(rq_n^{r+1})}(x),\ldots,f_0^{(q_n^{r+1})}(x)\big).\]
Since additionally $c_s^n\to(s-\gamma)\kappa^r$ as $n\to\infty$
for $s=0,1$, by Lemmas~\ref{kf1},~\ref{kf2},~\ref{kf3} and
combined with \eqref{eq:mia1}, \eqref{eq:mia2}, \eqref{eq:mia3},
we have
\begin{multline*}
\big(f_0^{((r+1)q_n^{r+1})}(x),\ldots,f_0^{(q_n^{r+1})}(x)\big)_*(\mu)\\
\to \nu_0 \cdot(R_{-\gamma\kappa^r})_*A_*(P_0)+\nu_1\cdot
(R_{(1-\gamma)\kappa^r})_*A_*(P_1)+\nu_2\cdot P_2.
\end{multline*}
Therefore,
\begin{equation}\label{eq:wzorp}
P=\nu_0 \cdot(R_{-\gamma\kappa^r})_*A_*(P_0)+\nu_1\cdot
(R_{(1-\gamma)\kappa^r})_*A_*(P_1)+\nu_2 \cdot P_2.
\end{equation}
As $T^f$ is isomorphic to its inverse, by
Corollary~\ref{krytNOWEd2prime}, $\theta_*(P)=P$. In view of
\eqref{eq:miksty}, it follows that the measure $P$ is equal to
\begin{equation}\label{eq:wzorp1}
\nu_0
\cdot(R_{\gamma\kappa^r})_*A_*(B_{-\gamma\kappa^r})_*(P_0)+\nu_1\cdot
(R_{-(1-\gamma)\kappa^r})_*A_*(B_{(1-\gamma)\kappa^r})_*(P_1)+
\nu_2\cdot\theta_*(P_2).
\end{equation}
Since $\nu_0+\nu_1>1/2$ and \eqref{eq:wzorp} and \eqref{eq:wzorp1}
hold, we have
\begin{gather}\label{eq:wiekpol}
\begin{split}
P\big(R_{-\gamma\kappa^r}\circ A(\R^r)\cup
R_{(1-\gamma)\kappa^r}\circ A(\R^r)\big)>1/2,\\
P\big(R_{\gamma\kappa^r}\circ A(\R^r)\cup
R_{-(1-\gamma)\kappa^r}\circ A(\R^r)\big)>1/2.
\end{split}
\end{gather}
As $\gamma\neq 0,1/2$, the sets
$\{-\gamma\kappa^r,(1-\gamma)\kappa^r\}$ and
$\{\gamma\kappa^r,-(1-\gamma)\kappa^r\}$ are disjoint. Hence the
sets $R_{-\gamma\kappa^r}\circ A(\R^r)\cup
R_{(1-\gamma)\kappa^r}\circ A(\R^r)$ and $R_{\gamma\kappa^r}\circ
A(\R^r)\cup R_{-(1-\gamma)\kappa^r}\circ A(\R^r)$ are disjoint,
contrary to \eqref{eq:wiekpol}. This completes the proof of
non-isomorphism of $T^f$ and its inverse.
\end{proof}

\section{Analytic flows on $\T^2$}\label{sec:analytic}
\subsection{Non-reversibility}\label{nonrevana}
Let us recall that that analytic special flows
 over irrational
rotations are precisely analytic reparametrizations of
two-dimensional rotations and that (in case of ergodicity) they
have simple spectra \cite{Co-Fo-Si}, so if they are isomorphic to
their inverses, they are automatically reversible.

The aim of this section is to provide analytic examples that are
not reversible. For this aim we briefly recall the
AACCP\footnote{The acronym comes from ``almost analytic cocycle
construction procedure''.} constructions \cite{Kw-Le-Ru} (see also
\cite{Le-Wy} for some modifications).

An AACCP is given by a collection of the following parameters:  a
sequence $(M_k)_{k\geq1}\subset\N$ together with an infinite real
matrix $((d_{k1},\ldots,d_{kM_k}))_{k\geq1}$ satisfying for each
$k\geq1$
$$
\sum_{i=1}^{M_k}d_{ki}=0.$$ Set
$D_k=\max\{|d_{ki}|:\:i=1,\ldots,M_k\}$. Then we select a sequence
$(\vep_k)_{k\geq1}$ of positive real numbers so that
$$
\sum_{k=1}^\infty\sqrt{\vep_k}M_k<+\infty,
\;\sum_{k=1}^\infty\vep_k<1,\;
\vep_k<\frac1{D_k^2},\;k=1,2,\ldots.$$ Finally, the completing
parameter of the AACCP is a real number $A>1$.

The above AACCP is said to be \emph{ realized over an irrational
number} $\alpha\in[0,1)$ having the continued fraction expansion
\[
\alpha=[0;a_1,a_2,\ldots]
\]
if there exists a strictly increasing sequence
$(n_k)_{k\geq1}\subset\N$ such that for each $k\geq1$ we have
\[
a_{2n_k+1}>2,\qquad
A^{N_k}\frac{D_kM_k}{a_{2n_k+1}q_{2n_k}}<\frac1{2^k},
\]
where $N_k$ is the degree of a real non-negative trigonometric
polynomial $P_k$ satisfying
\[
\int_0^1P_k(t)\,dt=1\quad\text{ and }\quad
P_k(t)<\vep_k\text{ for }t\in(\eta_k/2,1),
\]
where we require the numbers $\eta_k>0$ to satisfy
\[
4M_k\eta_k<\frac{\vep_k}{q_{2n_k}}\quad\text{ and
}\quad\frac1{a_{2n_k+1}q_{2n_k}}<\frac12\eta_k.
\]
Recall now that
\[
\zeta_{2n_k}=\Big\{\big[0,\{q_{2n_k}\alpha\}\big),T\big[0,\{q_{2n_k}\alpha\}\big),
\ldots, T^{q_{2n_k+1}-1}\big[0,\{q_{2n_k}\alpha\}\big)\Big\}
\]
and
\[
\overline{\zeta}_{2n_k}=
\Big\{\big[\{q_{2n_k+1}\alpha\},1\big),T\big[\{q_{2n_k+1}\alpha\},1\big),\ldots,
T^{q_{2n_k}-1}\big[\{q_{2n_k+1}\alpha\},1\big)\Big\}
\]
are two disjoint
Rokhlin towers fulfilling the whole interval $[0,1)$. It follows
that if we set
\[
I_k:=[0,\{a_{2n_k+1}q_{2n_k}\alpha\}),\quad
J_t^k:=T^{(t-1)q_{2n_k}} [0,\{q_{2n_k}\alpha\}),\ \text{ for }\
t=1,\ldots,a_{2n_k+1},
\]
then
$$I_k=\bigcup_{t=1}^{a_{2n_k+1}}J_t^k$$
and $I_k$ is the base of a Rokhlin tower of height $q_{2n_k}$
occupying at least $1-\frac2{a_{2n_k+1}}$ of the space.

Using the above parameters for $\alpha$ over which the AACCP can
be realized, one defines a real valued cocycle
$$
\varphi=\sum_{k=1}\varphi_k$$ as follows. In $I_k$ we choose
consecutively intervals $w_{k,1},\ldots,w_{k,M_k}$ of the same
length $\lambda_k\in(\eta_k,2\eta_k)$, each of which consists of
(the same) odd number $e_k\geq3$ of (consecutive) intervals
$J^k_t$. In general, the intervals $w_{k,i}$ and $w_{k,i+1}$ can
be separated by a certain number of intervals of the form $J^k_s$.
Denote by $J^{k}_{s_{k,i}}$ the middle interval $J^k_t$ in
$w_{k,i}$. Then we set
$$
\varphi_{k}(x)=\left\{\begin{array}{ll}d_{ki} &\text{if $x\in
J^k_{s_{k,i}}$}\\
0&\text{otherwise.}\end{array}\right.$$ Note that $I_{k+1}\subset
J^k_1$, so the supports of $\varphi_k$, $k\geq1$ are pairwise
disjoint.
%and therefore $\varphi$ has correctly been defined.

The following two results have been proved in \cite{Kw-Le-Ru}.
\begin{pr}\label{klr1} The set of $\alpha\in[0,1)$ over which an
AACCP can be realized is residual.\end{pr}
\begin{pr}\label{klr2} The cocycle $\varphi$ defined above is
cohomologous to an analytic cocycle.
\end{pr}
Moreover, we will also make use of the following observation from
\cite{Kw-Le-Ru}.
\begin{lm} For an arbitrary AACCP and $\alpha$ over which it is
realized, the cocycle $\varphi$ is constant on each interval
$T^iI_k$, $i=1,\ldots,q_{2n_k}-1$, $k\geq1$. Moreover, for each
$k\geq1$
\begin{equation}\label{klr4}
\sum_{i=1}^{q_{2n_k}-1}\varphi|_{T^iI_k}=0.\end{equation}
\end{lm}
We now proceed to our special construction. We assume that
\[
M_k=4M'_k\to\infty.
\]
Moreover, we assume that
\begin{equation}\label{mag1}
a_{2n_k+1}=e_kM_k\end{equation} with $e_k\geq3$ odd, $k\geq1$ and
\begin{equation}\label{mag2}
\frac1{\sqrt{M_k}}\sum_{i=1}^{k-1}M_i\leq C_1\end{equation} for a
constant $C_1>0$. The intervals $w_{k,i}$ are then defined as
consecutive unions of $e_k$ (consecutive) subintervals of the form
$J^k_t$. Fix $t_0, u_0\in\R$. For each $k\geq1$ we then set
\begin{equation}\label{eq:dki}
(d_{ki})=\Big(\underbrace{(t_0,u_0,-u_0,-t_0),\ldots,
(t_0,u_0,-u_0,-t_0)}_{M'_k\; \text{times}}\Big).
\end{equation}
Proceeding as in \cite{Le-Wy} and using~(\ref{klr4}), by
construction, we have
\begin{equation}\label{klr5} \varphi^{(e_kq_{2n_k})}-
\varphi_k^{(e_kq_{2n_k})}\to0\quad\text{in measure},\end{equation}
\begin{equation}\label{klr6}
\Big(\varphi_k^{(2e_kq_{2n_k})},\varphi_k^{(e_kq_{2n_k})}\Big)_\ast
\to\frac14\left(\delta_{(t_0+u_0,t_0)}+
\delta_{(0,u_0)}+\delta_{(-u_0-t_0,-u_0)}+\delta_{(0,-t_0)}\right),
\end{equation}
so
\begin{equation}\label{klr7}
\Big(\varphi^{(2e_kq_{2n_k})}, \varphi^{(e_kq_{2n_k})}\Big)_\ast
\to\frac14\left(\delta_{(t_0+u_0,t_0)}+
\delta_{(0,u_0)}+\delta_{(-u_0-t_0,-u_0)}+\delta_{(0,-t_0)}\right).
\end{equation}
Moreover,
\begin{equation}\label{klr8}
\{e_kq_{2n_k}\alpha\}\to0\;\;\text{when}\;\;k\to\infty.
\end{equation}
We now proceed similarly as in \cite{Le-Wy} and notice that if we
set  $C:=|t_0|+|u_0|$ then for each $k\geq1$ we have
\begin{equation}\label{klr9}
\left|\varphi^{(e_kq_{2n_k})}_k\right|\leq C.
\end{equation}
Since the support of $\sum_{i\geq k+1}\varphi_i$ is included in
$I_{k+1}$ and for each $x\in[0,1)$,
$$\#\left(\{x,Tx,\ldots,T^{e_kq_{2n_k}}x\}\cap I_{k+1}\right)
\leq1,$$ we also have
\begin{equation}\label{klr10}
\Big|\sum_{i\geq k+1}\varphi_i^{(e_kq_{2n_k})}\Big|\leq C.
\end{equation}
Finally, for $i=1,\ldots,k-1$ we have $\varphi_i=0$ on $I_k$, so
in view of~(\ref{klr4}), $\varphi_i^{(e_kq_{2n_k})}=0$ except for
the set
\[
[0,1)\setminus\bigcup_{j=0}^{q_{2n_k}-1}T^jI_k\quad\text{ and }
\quad \bigcup_{j=0}^{e_kq_{2n_k}-1}T^jJ^k_{a_{2n_k+1}-e_k}.
\]
Moreover, by \eqref{klr4}, for every $x\in\T$ and $m\geq 0$, $i\geq1$ we
have
\[|\varphi_i^{(m)}(x)|\leq \sum_{j=1}^{M_i}|d_{ij}|\leq CM_i.\]
In view of~(\ref{mag1}) and~(\ref{mag2}), it follows that
\begin{align*}
\Big\|\Big(\sum_{i=1}^{k-1}&\varphi_i\Big)^{(e_kq_{2n_k})}
\Big\|_2\leq C\sum_{i=1}^{k-1}M_i\cdot\mu\Big(\Big\{
x\in[0,1):\:\Big(\sum_{i=1}^{k-1}\varphi_i\Big)^{(e_kq_{2n_k})}
(x)\neq0\Big\}\Big)^{1/2}\\&\leq
C\sum_{i=1}^{k-1}M_i\left(\frac2{a_{2n_k+1}}+
\frac{e_k}{a_{2n_k+1}}\right)^{1/2}\leq2C\frac1{\sqrt{M_k}}
\sum_{i=1}^{k-1}M_i\leq const.
\end{align*} This together with~(\ref{klr9})
and~(\ref{klr10}) implies
\begin{equation}\label{klr11}
\left\|\varphi^{(e_kq_{2n_k})} \right\|_2\leq const.
\end{equation}
The cocycle $\varphi$ is clearly bounded (and of zero mean) and by
Proposition~\ref{klr2}, it is cohomologous to an analytic function
$f_0:\T\to\R$ (of zero mean). Then for each sufficiently large
constant $d>0$ we have
$$
\varphi+d>0,\;\;f:=f_0+d>0$$ and moreover the special flows
$T^{\varphi+d}$ and $T^f$ are isomorphic. In view
of~(\ref{klr8}),~(\ref{klr11}), Proposition~\ref{main1}
and~(\ref{klr7}) we obtain that for some constant $c>0$ ($c=\int_X f\,d\mu$)
\begin{equation}\label{klr12}
\mu^{\varphi+d}_{T^{\varphi+d}_{2ce_kq_{2n_k}},T^{\varphi+d}_{ce_kq_{2n_k}}}\!\to\!\frac14
\Big(\mu^{\varphi+d}_{T^{\varphi+d}_{t_0+u_0},T^{\varphi+d}_{t_0}}\!+\mu^{\varphi+d}_{Id,T^{\varphi+d}_{u_0}}\!+\!\mu^{\varphi+d}_{
T^{\varphi+d}_{-u_0-t_0},T^{\varphi+d}_{-u_0}}\!+\!\mu^{\varphi+d}_{Id,T^{\varphi+d}_{-t_0}}\Big).
\end{equation}
Since $T^{\varphi+d}$ and $T^f$ are isomorphic and~(\ref{klr12})
holds,
\begin{equation}\label{klr13}
\mu^{f}_{T^{f}_{2ce_kq_{2n_k}},T^{f} _{ce_kq_{2n_k}}}
\to\frac14\left( \mu^{f}_{T^{f}_{t_0+u_0},T^{f}_{t_0}} +
\mu^{f}_{Id,T^{f}_{u_0}}+ \mu^{f}_{T^{f}_{-u_0-t_0},T^{f}_{-u_0}}
+\mu^{f}_{Id,T^{f}_{-t_0}} \right).
\end{equation}
Now, the limit measure
\begin{align*}
P:&=\frac14\left(\delta_{(t_0+u_0,t_0)}+
\delta_{(0,u_0)}+\delta_{(-u_0-t_0,-u_0)}+\delta_{(0,-t_0)}\right)\\
&=\lim_{k\to\infty}
\left(f_0^{(2e_kq_{2n_k})},f_0^{(e_kq_{2n_k})}\right)_\ast(\mu)
\end{align*}
is not ``symmetric'' in the sense that $(0,u_0)$ is its atom,
while $(0,-u_0)$ is not provided that $u_0\neq0$ and $t_0\neq\pm
u_0$ and therefore under these additional assumptions, by
Corollary~\ref{kryt}, $T^f$ is not reversible. In this way we have
proved the following result.

\begin{wn}
There is an analytic weakly mixing flow on  $\T^2$  (preserving a
smooth measure) that is not reversible.
\end{wn}

\begin{uw}
If instead of $f$ (constructed above) we  consider
$f_{\vep}:=1+\vep f$ for small enough $\vep>0$ then  the
corresponding special flow $T^{f_{\vep}}$ can be interpreted as
arbitrarily small analytic change of time in the linear flow by
$(\alpha,1)$ on $\T^2$.

Now, note that $\left( f_{\vep}\right)_0=\vep f_0$. Hence
$$
\left(\left( f_{\vep}\right)_0^{(2e_kq_{2n_k})},\left(
f_{\vep}\right)_0^{(e_kq_{2n_k})}\right)_\ast(\mu)\to\left(M_{\vep}\right)_\ast
P=:P_{\vep},$$ where $M_{\vep}(x,y)=(\vep x,\vep y)$. It  follows
that $P_{\vep}$ has the same asymmetries as $P$ and therefore
$T^{f_{\vep}}$ is not reversible. Therefore, for some $\alpha$
irrational there are arbitrarily small analytic changes of time in
the linear flow by $(\alpha,1)$ on $\T^2$ which yield weakly
mixing and non-reversible flows.
\end{uw}

\begin{uw}\label{cztery} If in~(\ref{eq:dki}),
for each $k\geq1$, we consider the following pattern
\begin{align}\label{eq:dkicztery}
\begin{aligned}
&(d_{ki})=\\
&\Big(\underbrace{(a,b,c,b,b),\ldots, (a,b,c,b,b)}_{M'_k/2\;
\text{times}},\underbrace{(-a,-b,-c,-b,-b),\ldots,
(-a,-b,-c,-b,-b}_{M'_k/2\; \text{times}}\Big)
\end{aligned}
\end{align} then
\[
\left(f_0^{(2e_kq_{2n_k})}, f_0^{(e_kq_{2n_k})}\right)_\ast(\mu)\to P
\]
and
\[
\left( f_0^{(3e_kq_{2n_k})},f_0^{(2e_kq_{2n_k})},
f_0^{(e_kq_{2n_k})}\right)_\ast(\mu)\to Q.
\]
We have
\begin{align*}
P&=\frac1{10}\Big(\delta_{(a+b,a)}+\delta_{(b+c,b)}+\delta_{(b+c,c)}+\delta_{(2b,b)}
+\delta_{(a+b,b)}\\
&\qquad+\delta_{(-(a+b),-b)}+\delta_{(-(b+c),-b)}+\delta_{(-(b+c),-c)}
+\delta_{(-2b,b)} +\delta_{(-(a+b),-b)}\Big)
\end{align*}
and
\begin{align*}
Q&=\frac1{10}\Big(\delta_{(a+b+c,a+b,a)}+\delta_{(2b+c,b+c,b)}+\delta_{(2b+c,b+c,c)}+
\delta_{(a+2b,2b,b)} +\delta_{(a+2b,a+b,b)}\\
&\qquad+\delta_{(-(a+b+c),-(a+b),-b)}+\delta_{(-(2b+c),-(b+c),-b)}+
\delta_{(-(2b+c),-(b+c),-c)}\\
&\qquad +\delta_{(-(a+2b,-2b,b)}
+\delta_{(-(a+2b),-(a+b),-b)}\Big).
\end{align*}
It follows that $P$ is invariant under the map
$(x,y)\mapsto(x,x-y)$ and therefore we cannot apply
Corollary~\ref{kryt} but $Q$ is {\bf not} invariant under the map
$(x,y,z)\mapsto(x,x-z,x-y)$, so by
Corollary~\ref{krytNOWEd2prime}, the resulting flow is not
reversible.

Note however that Corollary~\ref{kryt} is sufficient for
non-reversibility of $T^f$ if instead of $(2e_kq_{2n_k},
e_kq_{2n_k})$ we consider $(4e_kq_{2n_k}, 2e_kq_{2n_k})$.
\end{uw}

\begin{problem}
When $\ct=(T_t)_{t\in\R}$ is weakly mixing then the method of
showing non-isomorphism of $\ct$ and its inverse passes to
non-trivial factors (see Remark~\ref{bbbbbb}). Since all flows
non-isomorphic to their inverses considered in the paper are are
weakly mixing, in fact, their non-trivial factors are also
non-isomorphic to their inverses. A natural question arises if
whenever the weak closure joining method applies, $\ct$ is
disjoint with its inverse.
\end{problem}

\subsection{Absence of rational self-similarities}
In this section we will show that the construction presented in
Section~\ref{nonrevana} can be easily modified so that we obtain
an analytic weakly mixing flow on $\T^2$ such that no rational
number is its scale of self-similarity (in particular, it is not
reversible). It remains an open question whether irrational
numbers can be scales of self-similarity for analytic flows
(preserving  smooth measure)\footnote{In \cite{Ku} it is shown
that on  on each compact orientable surface of genus at least~$2$
there is a smooth (non-singular) non-self-similar flow. It is
unknown whether these constructions are non-reversible. It is also
unknown whether a smooth non-self-similar flow can be constructed
on $\T^2$.}  on $\T^2$.

We will now recall a result which will ensure that
a rational number is not a scale of self-similarity of an
ergodic flow.
\begin{pr}[\cite{Le-Wy}]
Let $\cT=(T_t)_{t\in\R}$  and $\cS=(S_t)_{t\in\R}$ be flows on
$(X,\cB,\mu)$ and $(Y,\cC,\nu)$ respectively. Assume additionally
that $\cT$ is weakly mixing and $\cS$ is ergodic. Moreover,
suppose that for a sequence $(t_k)\subset\R$ with $t_k\to\infty$
\[
\mu_{T_{t_k}}\to \int_\R \mu_{T_{-t}}\ dP(t)\quad\text{ and
}\quad\nu_{S_{t_k}}\to \int_\R \nu_{S_{-t}}\ dQ(t).
\]
If $P\neq Q$ then the  flows $\cT$ and $\cS$ are disjoint in the
sense of Furstenberg.
\end{pr}

Notice that when $\cT$ and $\cS$ are disjoint, then for  each
$r\in\R^\ast$ the flows $\cT\circ r$ and $\cS\circ r$ are also
disjoint (indeed, $J(\cT,\cS)=J(\cT\circ r,\cS\circ r)$).
Therefore, the following result holds.

\begin{wn}\label{co}
Let $a,b\in \N$. Assume that $\mathcal{ T}=(T_t)_{t\in\R}$ is a
weakly mixing flow on $\xbm$. Assume also that for some
$t_k\to\infty$
\[
\mu_{T_{at_k}} \to \int_\R \mu_{T_{-at}}\ dP(t)\quad\text{ and
}\quad \mu_{T_{bt_k}} \to \int_\R \mu_{T_{-bt}}\ dQ(t).\] for some
probability measures $P$ and $Q$ on $\R$. If $P\neq Q$ then $\cT
\perp \cT\circ (b/a)$.
\end{wn}

We will now show what sequence of numbers $(d_{ki})$ to use in the
construction of a weakly mixing non-reversible analytic flow
(preserving a smooth measure) on $\T^2$ instead of the one
in~\eqref{eq:dki} to fulfill, for each natural $a<b$, the
assumptions of the above corollary. To this end we partition
$\N=\N_{-1,1}\bigcup_{a<b; a,b\in\N}\N_{a,b}$ so that $\N_{-1,1}$
and  each set $\N_{a,b}$ is infinite.  For $k\in\N_{-1,1}$ we
repeat the construction described in Section~\ref{nonrevana}, so
that the resulting flow will not be reversible.

Take $(a,b)\in\N^2$. By reversing the roles of $a$ and $b$,  we
may assume that $a<b$. We will consider only $k\in\N_{a,b}$.
Assume that $M_k=bM'_k$ and set
\[
(d_{ki})
:=\Big(\underbrace{\Big(t_0,\underbrace{-\frac{t_0}{b-1},\dots,-\frac{t_0}{b-1}}_{b-1
\text{ times}}\Big),\ldots,
\Big(t_0,-\frac{t_0}{b-1},\dots,-\frac{t_0}{b-1}\Big)}_{M'_k\;
\text{times}}\Big).
\]
It follows that for the analytic flow  $\cT=(T_t)_{t\in\R}$ on
$\T^2$ constructed in such a way as in Section~\ref{nonrevana} we
have (with $c=\int_{X}f\,d\mu)$
\begin{align*}
\lim_{k\to\infty, k\in\N_{a,b}} \mu_{T_{ace_kq_{2n_k}}}&= \int_\R \mu_{T_{-t}} \ d\Big(\frac{a}{b}
\delta_{\left(1-\frac{a}{b-1}\, \right)t_0} + \frac{b-a}{b}\,
\delta_{-\frac{a}{b-1}t_0} \Big)(t) \\& =\int_\R \mu_{T_{-at}}\
d\Big(\frac{a}{b}\, \delta_{\left(\frac{1}{a}-\frac{1}{b-1}
\right)t_0}  + \frac{b-a}{b}\, \delta_{-\frac{1}{b-1}t_0}\Big)(t)
\end{align*}
and
\begin{equation*}
\lim_{k\to\infty,k\in\N_{a,b}} \mu_{T_{bce_kq_{2n_k}}}=\mu_{Id} = \int_\R \mu_{T_{-t}}\ d\delta_0(t).
\end{equation*}
In view of Corollary~\ref{co}, this yields $\cT \perp \cT\circ
(a/b)$ for arbitrary natural $a<b$. Since $-1\notin I(\cT)$ and
$I(\cT)$  is a multiplicative subgroup of $\R^\ast$, we have
$\Q\cap I(\cT)=\{1\}$. Hence, we have proved the following result.
\begin{wn}
There is an analytic  weakly mixing flow (preserving a smooth measure) on $\T^2$ that is not
reversible and such that no rational number is its scale of
self-similarity.
\end{wn}

\section{Non-reversible  Chacon's type automorphisms}\label{sec:chacon}
The following result was essentially proved in
\cite{Ry1}~\footnote{Formally, in \cite{Ry1} different sequences
are considered  and two limit joinings (not of the above form) are
considered, but the essence of the argument is the same.
Proposition~\ref{kryt1}  is a natural automorphism counterpart  of
Proposition~\ref{kryt}.}:
\begin{pr}\label{kryt1}
Assume that $T$ is an ergodic automorphism on $\xbm$. Assume also that
$$
\mu_{T^{2q_n},T^{q_n}}\to\int_{\Z^2}\mu_{T^{-a},T^{-b}}\,dP(a,b)
$$
for some probability measure $P$ on $\Z^2$. If the measure $P$ is
not invariant under $\theta(a,b)=(a,a-b)$ then $T$ is not
isomorphic to its inverse. In particular, $T$ is not reversible.
\end{pr}
In this section we consider some rank one automorphisms in
construction  of which along a subsequence we repeat a Chacon's
type construction \cite{Ju-Ra-Sw}; we will obtain   non-reversible
rank one automorphisms.

Recall briefly  a rank one construction (see e.g.
\cite{MR1719722}). For a sequence of positive integers
$(r_n)_{n\in\N}$ with all $r_n\geq2$ and
$(s^{(n)}_1,\ldots,s^{(n)}_{r_n})_{n\in\N}$ with all $s^{(n)}_i$
non-negative integers we define a rank-one transformation by
giving an increasing sequence of Rokhlin towers $(C_n)_{n\in \N}$
such that each $C_n$ consists of $q_n$ pairwise disjoint intervals
of the same length (each such interval is called a level of
$C_n$). More precisely,
$C_n=\{C_{n,1},C_{n,2},\ldots,C_{n,q_n}\}$, the dynamics $T$ is
defined on $\bigcup_{i=1}^{q_n-1}C_{n,i}$ so that $T$ sends
linearly $C_{n,j}$ to $C_{n,j+1}$ for $j=1,\ldots,q_n-1$.  The
tower $C_{n+1}$ is obtained  first by cutting $C_n$ into $r_n$
subcolumns, say $C_n(i)$, $1\leq i\leq r_n$, of equal width,
placing $s^{(n)}_i$ spacers  over each subcolumn $C_n(i)$ and
finally stacking each subcolumn $C_n(i)$ on the top of $C_n(i+1)$
for $1\leq i<r_n$ in order to complete the definition of
$C_{n+1}$. The tower $C_{n+1}$ has the height
$q_{n+1}=r_nq_n+\sum_{i=1}^{r_n}s^{(n)}_i$. The ordering of levels
in $C_{n+1}$ is lexicographical from  the left to the right. The
dynamics $T$ on $\bigcup_{i=1}^{q_{n+1}-1}C_{n+1,i}$ is completed
by sending linearly $C_{n+1,q_n}$ to the first spacer over the
first subcolumn $C_n(1)$, sending this spacer to the one above it,
etc., and when reaching the top spacer we send it to
$C_{n+1,q_n+1}$. We keep going the same procedure for the
remaining columns and stop at the top spacer over $C_n(r_n)$. In
this way we obtain a measure-preserving transformation $T$ defined
on a standard Borel space $\xbm$ although, in general, $\mu$ is
only $\sigma$-finite.  Provided that the number of spacers is not
too large \cite{MR1719722}, $\mu$ can be assumed (and this is our
tacit standing assumption) to be a probability measure.

We will now describe the details of our  particular rank-one
construction. Fix an even positive integer $r\geq 4$ and an
increasing sequence $(n_k)_{k\in \N}$. Suppose that $r_{n_k}=r$
and $r_{n_k+1}\to \infty$\footnote{The assumption on the spacers
over the subcolumns $C_{n_k+1}(i)$ for $\left[
r_{n_k+1}/2\right]+1\leq i\leq r_{n_k+1}$ at step $n_k+1$ of the
construction yields some form of ``rigidity''. This condition can
be modified. What is important, is that we prevent the image of a
single level of tower $C_{n_k}$ under $T^{rq_{n_k}}$ from being
``too scattered''.} and in the construction we place one spacer
over $C_{n_k}(i)$ for $r/2+1\leq i\leq r$ and over $C_{n_k+1}(i)$
for $\left[ r_{n_k+1}/2\right]+1\leq i\leq r_{n_k+1}$.

\begin{tw}\label{nonrevchacon}
Under the above assumptions the constructed rank-one automorphism $T$ is not reversible.
\end{tw}

\begin{proof}
We claim that
\begin{equation*}
\mu_{T^{2q_{n_k}},T^{q_{n_k}}} \to \int_{\Z^2} \mu_{T^{-a},T^{-b}}\ dP(a,b),
\end{equation*}
where
\begin{equation}\label{atomy}
P(1,0)=\frac{1}{r},\ P(\theta(1,0))=P(1,1)=\frac{1}{2r}
\end{equation}
Once we have shown this, the claim will follow
by Proposition~\ref{kryt1}.
%\footnote{The remaining part of the proof is illustrated in Figure~\ref{fig} for $r=4$, $s=6$.}
\begin{figure}[htb]
\centering
\includegraphics[width=0.7\linewidth]{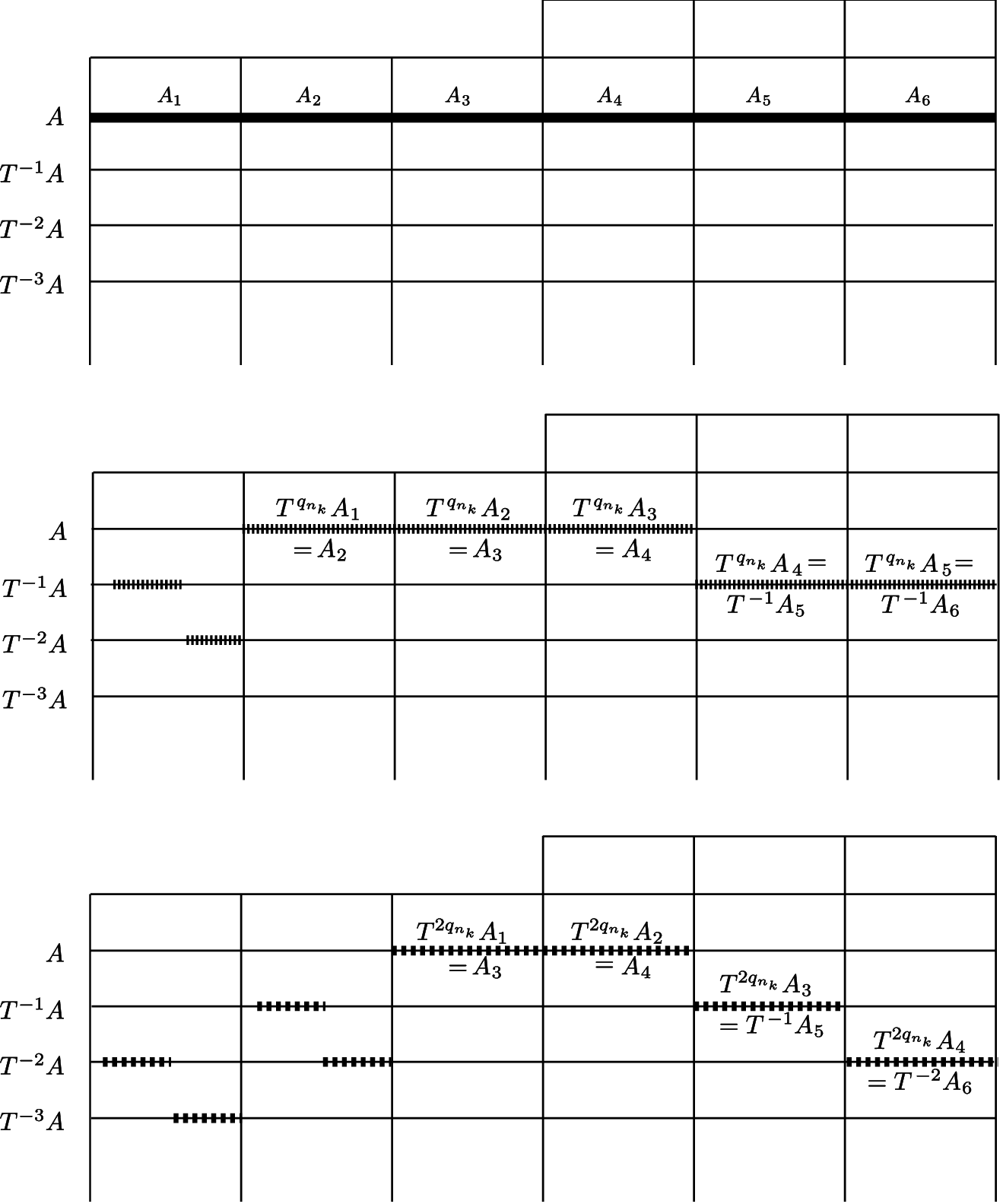}
\caption{Illustration of the proof in the case $r=4, s=6$.}
\label{fig}
\end{figure}
\begin{figure}[ht]
\centering
\includegraphics[width=0.7\linewidth]{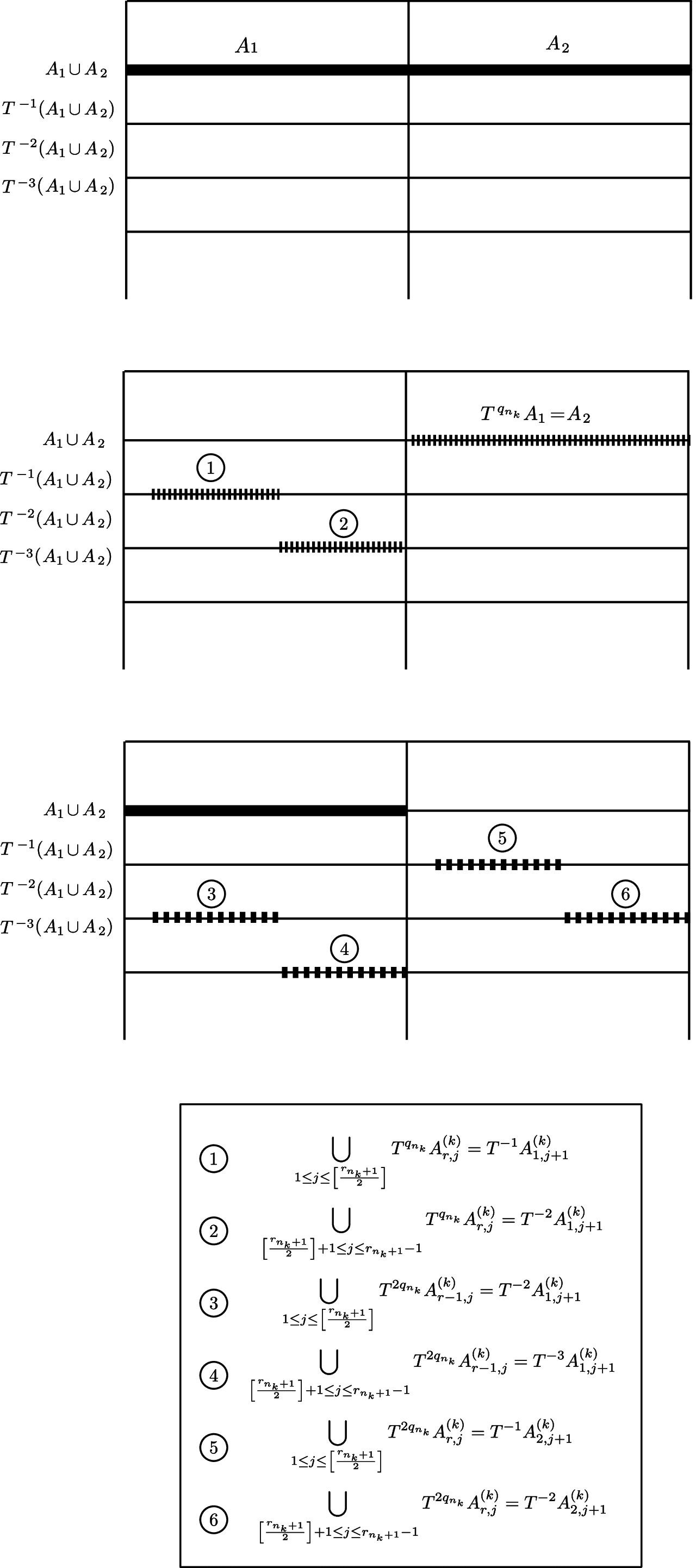}
\caption{A part of Figure~\ref{fig} magnified.}
\label{fig1}
\end{figure}
%\begin{figure}[ht]
%\centering
%\includegraphics[width=0.8\linewidth]{wieze2}
%\caption{A part of Figure~\ref{fig} magnified even more.}
%\label{fig2}
%\end{figure}
Since every measurable set can be  approximated by unions of
levels of sufficiently high towers, it suffices to show
that
\begin{equation*}
\mu_{T^{2q_{n_k}r},T^{q_{n_k}r}}(A[1]\times A[2]\times A[3]) \to \int_{\Z^2} \mu_{T^{-a},T^{-b}}(A[1]\times A[2]\times A[3])\ dP(a,b),
\end{equation*}
%where
%\begin{equation}\label{atomy}
%P(r+1,r/2)=\frac{1}{2s},\ P(\theta(r+1,r/2))=P(r+1,r/2+1)=\frac{1}{s}
%\end{equation}
for $A[1],A[2],A[3]$ being single levels of the tower  $C_{k_0}$
for arbitrarily large $k_0\in\N$ and the measure $P$ satisfies
\eqref{atomy}.
%%%%%%% TUTAJ!!!
Since the towers are arbitrarily high, the levels $A{[1]}, A[2],
A[3]$ can be assumed  not to be any of the first $3$ bottom
levels. Fix $k_0\in \N$ and let $A{[1]}, A[2], A[3]$ be single
levels of the tower $C_{n_{k_0}}$. Without loss of generality we
may assume that $C_{n_{k_0}}$ is at least of height $4$. For each
$k\geq k_0$ the sets $A{[1]}, A[2], A[3]$ become finite disjoint
unions of levels of tower $C_{n_k}$:
\begin{equation}\label{55-}
A[1]=\bigcup_{l=1}^{l_k}A^{(k)}[1,l],\ A[2]=\bigcup_{l=1}^{l_k}A^{(k)}[2,l],\ A[3]=\bigcup_{l=1}^{l_k}A^{(k)}[3,l]
\end{equation}
for some $l_k\geq 1$. Moreover,
\begin{equation}\label{2a}
    \begin{alignedat}{1}
        \text{for any $1\leq l<l'\leq l_k$ there are at least $3$ levels}\\
        \text{of tower $C_{n_k}$ between $A^{(n_k)}[t,l]$ and $A^{(n_k)}[t,l']$ for $t=1,2,3$.}
    \end{alignedat}
\end{equation}

Let $\vep>0$. We claim that for $k\geq k_0$  sufficiently large
and for $A^{(k)}$ being a level of tower $C_{n_k}$ which is not
one of the first $3$ levels we have {\allowdisplaybreaks
\begin{align}\label{1bl}
&\Big|\mu_{T^{2rq_{n_k}},T^{rq_{n_k}}}\big(A^{(k)}\times A^{(k)}\times A^{(k)}\big) - \frac{r-2}{2r}\mu \big(A^{(k)}\big)\Big|
< \vep \mu\big(A^{(k)}\big),\notag\\
&\Big| \mu_{T^{2rq_{n_k}},T^{rq_{n_k}}}\big(A^{(k)}\times A^{(k)}\times TA^{(k)}\big) - \frac{1}{2r}\mu \big(A^{(k)}\big)\Big|
< \vep \mu\big(A^{(k)}\big),\notag\\
&\Big| \mu_{T^{2rq_{n_k}},T^{rq_{n_k}}}\big(A^{(k)}\times A^{(k)}\times T^{2}A^{(k)}\big) - \frac{1}{2r}\mu \big(A^{(k)}\big)\Big|
< \vep \mu\big(A^{(k)}\big),\\
&\Big| \mu_{T^{2rq_{n_k}},T^{rq_{n_k}}}\big(A^{(k)}\times TA^{(k)}\times TA^{(k)}\big) - \frac{1}{r}\mu \big(A^{(k)}\big)\Big|
< \vep \mu\big(A^{(k)}\big),\notag\\
&\Big| \mu_{T^{2rq_{n_k}},T^{rq_{n_k}}}\big(A^{(k)}\times TA^{(k)}\times T^{2}A^{(k)}\big) - \frac{r-3}{2r}\mu\big(A^{(k)}\big)\Big|
< \vep \mu\big(A^{(k)}\big),\notag\\
&\Big| \mu_{T^{2rq_{n_k}},T^{rq_{n_k}}}\big(A^{(k)}\times T^{2}A^{(k)}\times T^{3}A^{(k)}\big) - \frac{1}{2r}\mu \big(A^{(k)}\big)\Big|
< \vep \mu\big(A^{(k)}\big).\notag
\end{align}}
%that for $A^{(k)}$ being a level of tower $C_{n_k}$ which is not one of the first $2r+2$ levels we have
%\begin{equation}\label{1a}
%\begin{alignedat}{2}
%\mu_{T^{4rq_{n_k}},T^{2rq_{n_k}}}\left(T^{-r}A^{(k)}\times T^{-2r}A^{(k)}\times A^{(k)}\right)&=\frac{s-1}{2s}\mu\left(A^{(k)}\right),\\
%\mu_{T^{4rq_{n_k}},T^{2rq_{n_k}}}\left(T^{-r}A^{(k)}\times T^{-2r-1}A^{(k)}\times A^{(k)}\right)&=\frac{1}{4s}\mu\left(A^{(k)}\right),\\
%\mu_{T^{4rq_{n_k}},T^{2rq_{n_k}}}\left(T^{-r}A^{(k)}\times T^{-2r-2}A^{(k)}\times A^{(k)}\right)&\to\frac{1}{4s}\mu\left(A^{(k)}\right),\\
%\mu_{T^{4rq_{n_k}},T^{2rq_{n_k}}}\left(T^{-r-1}A^{(k)}\times T^{-2r-1}A^{(k)}\times A^{(k)}\right)&=\frac{s-2}{2s}\mu\left(A^{(k)}\right),\\
%\mu_{T^{4rq_{n_k}},T^{2rq_{n_k}}}\left(T^{-r-1}A^{(k)}\times T^{-2r-2}A^{(k)}\times A^{(k)}\right)&=\frac{3}{4s}\mu\left(A^{(k)}\right),\\
%\mu_{T^{4rq_{n_k}},T^{2rq_{n_k}}}\left(T^{-r-2}A^{(k)}\times T^{-2r-3}A^{(k)}\times A^{(k)}\right)&\to\frac{1}{4s}\mu\left(A^{(k)}\right).
%\end{alignedat}
%\end{equation}
%Moreover, the convergence in two of the above formulas is uniform, i.e. for all $\vep>0$ there exists $k_0$ such that for $k>k_0$
%\begin{equation}\label{1bl}
%\begin{alignedat}{2}
%\left| \mu_{T^{4rq_{n_k}},T^{2rq_{n_k}}}\left(T^{-r}A^{(k)}\times T^{-2r-2}A^{(k)}\times A^{(k)}\right) - \frac{1}{4s}\mu\left(A^{(k)}\right)\right|& \\
%&< \vep \mu\left(A^{(k)}\right),\\
%\left|\mu_{T^{4rq_{n_k}},T^{2rq_{n_k}}}\left(T^{-r-2}A^{(k)}\times T^{-2r-3}A^{(k)}\times A^{(k)}\right)-\frac{1}{4s}\mu\left(A^{(k)}\right)\right|& \\
%&<\vep \mu\left(A^{(k)}\right).
%\end{alignedat}
%\end{equation}
The proof of all of these inequalities goes along the  same lines,
we will prove only the fifth of them, i.e.
\begin{equation}\label{ddd}
\Big| \mu_{T^{2rq_{n_k}},T^{rq_{n_k}}}\big(A^{(k)}\times TA^{(k)}
\times T^{2}A^{(k)}\big) - \frac{r-3}{2r}\mu\big(A^{(k)}\big)\Big|
< \vep \mu\big(A^{(k)}\big)
\end{equation}
(the proof of~\eqref{ddd} contains all elements of the  proofs of
the other inequalities in~\eqref{1bl}). To make the notation
simpler we will write $C$ for the tower $C_{n_k}$ and we will also
change the notation for the subcolumns. Now, tower $C_{n_k}$ is
cut into $r$ subcolumns of equal width, denoted from left to right
by $C^i$, $1\leq i\leq r$, with one spacer placed spacer over
subcolumns $C^i$ for $r/2+1\leq i\leq r$. Then each of $C^i$ is
cut into $s$ subcolumns of equal width, denoted from left to right
by $C^{i,j}$, $1\leq j\leq r_{n_k+1}$, with one spacer placed over
subcolumns $C^{r,j}$ for $\left[\frac{r_{n_k+1}}{2} \right]+1 \leq
j\leq r_{n_k+1}$. Let
\begin{equation*}
A^{(k)}_i:=A^{(k)}\cap C^i,\ A^{(k)}_{i,j}:=A^{(k)}\cap C^{i,j}
\end{equation*}
for $1\leq i\leq r$ and $1\leq j\leq r_{n_k+1}$.
\\
Notice that we have
\begin{eqnarray}\label{1c}
    \begin{aligned}
        T^{q_{n_k}}A^{(k)}_{i}&=A^{(k)}_{i+1} &\text{ for }1\leq i\leq r/2,\\
        T^{q_{n_k}}A^{(k)}_{i}&=T^{-1}A^{(k)}_{i+1} &\text{ for }r/2+1\leq i\leq r-1
    \end{aligned}
\end{eqnarray}
and
\begin{equation}\label{1d}
    \begin{alignedat}{1}
        T^{q_{n_k}}A^{(k)}_{r,j}&=T^{-1}A^{(k)}_{1,j+1}\text{ for }1\leq j\leq \left[\frac{r_{n_k+1}}{2} \right],\\
        T^{q_{n_k}}A^{(k)}_{r,j}&=T^{-2}A^{(k)}_{1,j+1}\text{ for }\left[\frac{r_{n_k+1}}{2} \right]+1\leq j\leq r_{n_k+1}-1.
    \end{alignedat}
\end{equation}
Moreover
\begin{equation}\label{3a}
\mu\Big(\!A^{(k)}\!\setminus\!\Big(\!\bigcup_{1\leq i\leq r-1}\!A_{i}^{(k)}\cup \!\bigcup_{1\leq j\leq r_{n_k+1}-1}\!A^{(k)}_{r,j}\!\Big)\!\Big)
\!=\!\mu\big(A^{(k)}_{r,r_{n_k+1}}\big)\!=\!\frac{1}{rr_{n_k+2}}\mu\big(A^{(k)}\big).
\end{equation}
We also have
\begin{equation}\label{1c2}
    \begin{alignedat}{2}
        T^{2q_{n_k}}A^{(k)}_{i}&=A^{(k)}_{1,i+2} & \text{  for }1\leq i\leq r/2-1,\\
        T^{2q_{n_k}}A^{(k)}_{1,r/2}&=T^{-1}A^{(k)}_{1,r/2+2}, &\\
        T^{2q_{n_k}}A^{(k)}_{1,j}&=T^{-2}A^{(k)}_{1,j+2} & \text{  for }r/2+1\leq j\leq r-2\\
    \end{alignedat}
\end{equation}
and
\begin{equation}\label{1d2}
    \begin{alignedat}{1}
        T^{2q_{n_k}}A^{(k)}_{r-1,j}&=T^{-2}A^{(k)}_{1,j+1}\text{ for }1\leq j\leq \left[\frac{r_{n_k+1}}{2} \right],\\
        T^{2q_{n_k}}A^{(k)}_{r-1,j}&=T^{-3}A^{(k)}_{1,j+1}\text{ for }\left[\frac{r_{n_k+1}}{2} \right]+1\leq j\leq r_{n_k+1}-1
    \end{alignedat}
\end{equation}
and
\begin{equation}\label{1dd}
    \begin{alignedat}{1}
        T^{2q_{n_k}}A^{(k)}_{r,j}&=T^{-1}A^{(k)}_{2,j+1}\text{ for }1\leq j\leq \left[\frac{r_{n_k+1}}{2} \right],\\
        T^{2q_{n_k}}A^{(k)}_{r,j}&=T^{-2}A^{(k)}_{2,j+1}\text{ for }\left[\frac{r_{n_k+1}}{2} \right]+1\leq j\leq r_{n_k+1}-1.
    \end{alignedat}
\end{equation}
Moreover
\begin{multline}\label{3b}
\mu\left(A^{(k)}\setminus\left(\bigcup_{1\leq i\leq r-2}A_{i}^{(k)}\cup \bigcup_{1\leq j\leq r_{n_k+1}-1}A^{(k)}_{r-1,j}\cup \bigcup_{1\leq j\leq r_{n_k+1}-1}A^{(k)}_{r,j}\right)\right)\\
=\mu\left(A^{(k)}_{r-1,r_{n_k+1}}\right)+\mu\left(A^{(k)}_{r,r_{n_k+1}}\right)=\frac{2}{rr_{n_k+1}}\mu\left(A^{(k)}\right).
\end{multline}
%\\
%Notice that up to a set of arbitrarily small measure for large $k$ we have
%$$
%T^{2rq_{n_k}}A_i^{(k)}, T^{rq_{n_k}}A_i^{(k)}\subset C_i.
%$$
%More precisely, we have
Using all of the~\cref{1c,1d,3a} we obtain
\begin{equation}\label{eeprim}
\mu\left(T^{q_{n_k}}A^{(k)}\setminus \bigcup_{p_2\in\{0,1,2\}}T^{-p_2}A^{(k)} \right)<\frac{1}{rr_{n_k+2}}\mu\left(A^{(k)}\right)
\end{equation}
and using~\cref{1c2,1d2,1dd,3b}
\begin{equation}\label{eeeprim}
\mu\left(T^{2q_{n_k}}A^{(k)}\setminus \bigcup_{p_3\in\{0,1,2,3\}}T^{-p_3}A^{(k)} \right)<\frac{2}{rr_{n_k+2}}\mu\left(A^{(k)}\right).
\end{equation}
We will show now that~\eqref{ddd} is true. %holds for $A_1^{(k)}$.
%\begin{multline}\label{dddd}
%\left| \mu_{T^{2rq_{n_k}},T^{rq_{n_k}}}\left(A_1^{(k)}\times T^{r+1}A_1^{(k)}\times T^{2r+2}A_1^{(k)}\right) - \frac{s-3}{2s}\mu\left(A_1^{(k)}\right)\right| \\
%< \vep \mu\left(A_1^{(k)}\right)
%\end{multline}
We have
{\allowdisplaybreaks
\begin{align*}
A^{(k)}&\cap TT^{q_{n_k}}A^{(k)}\cap T^{2}T^{2rq_{n_k}}A^{(k)}\\
&\stackrel{(*)}{\simeq}\Big(A^{(k)}\cap TT^{q_{n_k}}\Big(\bigcup_{1\leq i\leq r-1}A_{i}^{(k)}\cup \bigcup_{1\leq j\leq r_{n_k+1}-1}A_{r,j}^{(k)} \Big) \Big)\cap T^{2}T^{2q_{n_k}}A^{(k)}\\
&\stackrel{\eqref{1c},\eqref{1d},\eqref{3a}}{=}\Big(\bigcup_{r/2+1\leq i\leq r-1}A_{i+1}^{(k)}\cup\bigcup_{1\leq j\leq \left[\frac{r_{n_k+1}}{2}\right]}A_{1,j+1}^{(k)}  \Big)\cap  T^{2}T^{2q_{n_k}}A^{(k)}\\
&\stackrel{(**)}{\simeq}\Big(\bigcup_{r/2+1\leq i\leq r-1}A_{i+1}^{(k)}\cup\bigcup_{1\leq j\leq \left[\frac{r_{n_k+1}}{2}\right]}A_{1,j+1}^{(k)}  \Big)\\
&\quad\cap T^{2}T^{2q_{n_k}}\Big( \bigcup_{1\leq i\leq r-2}A_{i}^{(k)}\cup \bigcup_{1\leq j\leq r_{n_k+1}-1}A_{r-1,j}^{(k)}\cup\bigcup_{1\leq j\leq r_{n_k+1}-1}A_{r,j}^{(k)}\Big)\\
&\stackrel{\eqref{1c2},\eqref{1d2},\eqref{1dd}}{=}\Big(\bigcup_{r/2+1\leq i\leq r-1}A_{i+1}^{(k)}\cup\bigcup_{1\leq j\leq \left[\frac{r_{n_k+1}}{2}\right]}A_{1,j+1}^{(k)}  \Big)\\
&\quad\cap \Big(\bigcup_{r/2+1\leq i\leq r-2}A_{i+2}^{(k)}\cup \bigcup_{1\leq j\leq \left[\frac{r_{n_k+1}}{2}\right]}A_{1,j+1}^{(k)}\cup \bigcup_{\left[\frac{r_{n_k+1}}{2} \right]+1\leq j\leq r_{n_k+1}-1}A_{2,j+1}^{(k)}  \Big)\\
&=\bigcup_{r/2+3\leq i\leq r}A_{i}^{(k)}\cup \bigcup_{2\leq j\leq \left[\frac{r_{n_k+1}}{2} \right]+1}A_{1,j},
\end{align*}}
where $(*)$ and $(**)$  hold  up to a set of measure
$\frac{1}{rr_{n_k+2}}\mu(A^{(k)})$ and
$\frac{2}{rr_{n_k+2}}\mu(A^{(k)})$, respectively.
%By~\cref{3a,3b}, up to a set of measure $\frac{3}{sr_{n_k+2}}\mu(A_1)$, we have
%\begin{multline*}
%A_1^{(k)}\cap T^{r+1}T^{rq_{n_k}}A_1^{(k)}\cap T^{2r+2}T^{2rq_{n_k}}A_1^{(k)}\\
%\simeq \left(\bigcup_{s/2+1\leq j\leq s-1}A_{1,j+1}^{(k)} \cup\bigcup_{1\leq t\leq \left[ \frac{r_{n_k+2}}{2}\right]}A_{1,1,t+1}^{(k)}\right) \cap T^{2r+2}T^{2rq_{n_k}}A_1^{(k)}\\
%\simeq \left(\bigcup_{s/2+1\leq j\leq s-1}A_{1,j+1}^{(k)}\cup\bigcup_{1\leq t\leq \left[ \frac{r_{n_k+2}}{2}\right]}A_{1,1,t+1}^{(k)} \right)\\
%  \cap \left(\bigcup_{s/2+1\leq j\leq s-2}A_{1,j+2}^{(k)} \cup \bigcup_{1\leq t\leq \left[\frac{r_{n_k+2}}{2} \right]}A_{1,1,t+1}^{(k)}\cup \bigcup_{\left[\frac{r_{n_k+2}}{2} \right]+1\leq t\leq r_{n_k+2}-2}A_{1,2,t+1}^{(k)}\right) \\
%= \bigcup_{s/2+3\leq j\leq s}A_{1,j}^{(k)}\cup \bigcup_{2\leq t\leq \left[ \frac{r_{n_k+2}}{2} \right]+1}A_{1,1,t}^{(k)}.
%\end{multline*}
Therefore up to an error of  absolute value at most
$\frac{3}{rr_{n_k+2}}\mu(A^{(k)})$ {\allowdisplaybreaks
\begin{align*}
\mu&_{T^{2q_{n_k}},T^{q_{n_k}}} \left(A^{(k)}\times TA^{(k)}\times T^{2}A^{(k)} \right)=\mu\left(T^{-2q_{n_k}}A^{(k)}\cap T^{-q_{n_k}+1}A^{(k)}\cap T^{2}A^{(k)} \right)\\
&=\mu\left(A^{(k)}\cap TT^{q_{n_k}}A^{(k)}\cap T^{2}T^{2q_{n_k}}A^{(k)} \right)\\
&\simeq\mu\left(\bigcup_{r/2+3\leq i\leq r}A_{i}^{(k)}\cup \bigcup_{2\leq j\leq \left[ \frac{r_{n_k+1}}{2} \right]+1}A_{1,j}^{(k)} \right)\\
&=\frac{r-r/2-3+1}{r}\cdot \mu(A^{(k)})+\frac{\left[\frac{r_{n_k+1}}{2} \right]+1-2+1}{rr_{n_k+2}}\cdot\mu(A^{(k)})\\
&=\frac{r-4}{2r}\cdot \mu(A^{(k)})+\frac{\left[\frac{r_{n_k+1}}{2} \right]}{rr_{n_k+2}}\mu(A^{(k)})
=\begin{cases}
\frac{r-3}{2r}\mu(A^{(k)}),& \!2|r_{n_k+1} \\
\frac{r-3}{2r}\mu(A^{(k)})-\frac{1}{2rr_{n_k+1}}\mu(A^{(k)}),& \!2\not |r_{n_k+1},
\end{cases}% \to \frac{s-3}{2s}\cdot \mu(A_1^{(k)}),
\end{align*}}
i.e.~\eqref{ddd} indeed holds. In a similar way, all of the inequalities~\eqref{1bl} hold.
%Notice that it follows from~\cref{ee,eee,2a} that
%\begin{equation}\label{65!}
%\begin{alignedat}{1}
%&\mu\left(T^{rq_{n_k}}A_i^{(k)}\setminus\left(\cup_{j=0}^{2r+2}T^{-j}A_i^{(k)}\right)\right)\leq \frac{1}{r_{n_k+2}}\mu\left(A_i^{(k)}\right),\\
%&\mu\left(T^{2rq_{n_k}}A_i^{(k)}\setminus\left(\cup_{j=0}^{2r+2}T^{-j}A_i^{(k)}\right)\right)\leq \frac{2}{r_{n_k+2}}\mu\left(A_i^{(k)}\right).
%\end{alignedat}
%\end{equation}
We obtain
{\allowdisplaybreaks
\begin{align*}
&\Big|\mu_{T^{2q_{n_k}},T^{q_{n_k}}}\left(A[1]\times TA[1]\times T^{2}A[1] \right)\\
&\quad-\frac{r-3}{2r}\mu_{T^{-2},T^{-1}}\left(A[1]\times T A[1]\times T^{2}A[1]\right) \Big|\\
%%%%%%%
&\stackrel{\eqref{55-}}{=}\Big|\mu_{T^{2q_{n_k}},T^{q_{n_k}}}\Big(\Big(\bigcup_{l=1}^{l_k}A^{(k)}[1,l]
\Big)\!\times\!\Big(\bigcup_{l=1}^{l_k}TA^{(k)}[1,l]\Big)\!\times\!\Big(
\bigcup_{l=1}^{l_k}T^{2}A^{(k)}[1,l]\Big)\Big)\\
&\quad-\!\frac{r-3}{2r}\mu_{T^{-2},T^{-1}}\!\Big(\!\Big(\!\bigcup_{l=1}^{l_k}A^{(k)}
[1,l]\!\Big)\!\times\!\Big(\!\bigcup_{l=1}^{l_k}
TA^{(k)}[1,l]\!\Big)\!\times\!\Big(\!\bigcup_{l=1}^{l_k}T^{2}A^{(k)}[1,l]\!\Big)\!\Big)\!\Big|\\
%%%%%%%%%%%%%%%%%%%%%%%%
&{\leq}\sum_{l=1}^{l_k}\left|\mu_{T^{2q_{n_k}},T^{q_{n_k}}}\left(A^{(k)}[1,l]\times TA^{(k)}[1,l]\times T^{2}A^{(k)}[1,l] \right) \right.\\
&\quad\left. -\frac{r-3}{2r}\mu_{T^{-2},T^{-1}}\left( A^{(k)}[1,l]\times TA^{(k)}[1,l]\times T^{2}A^{(k)}[1,l]\right) \right|\\
&\quad+\sum_{l=1}^{l_k}\sum_{\substack{1\leq l',l''\leq l_k \\ \#\{l,l',l''\}>1}}\mu_{T^{2q_{n_k}},T^{q_{n_k}}}\left(A^{(k)}[1,l]\times TA^{(k)}[1,l']\times T^{2}A^{(k)}[1,l''] \right)\\
&\quad+\frac{r-3}{2r}\sum_{l=1}^{l_k}\!\sum_{\substack{1\leq l',l''\leq l_k \\ \#\{l,l',l''\}>1}}\!\mu_{T^{-2},T^{-1}}\!\big(A^{(k)}[1,l]\times TA^{(k)}[1,l']\times T^{2}A^{(k)}[1,l''] \big)\\
%%%%%%%
&=\sum_{l=1}^{l_k}\left|\mu_{T^{2q_{n_k}},T^{q_{n_k}}}\left(A^{(k)}[1,l]\times TA^{(k)}[1,l]\times T^{2}A^{(k)}[1,l] \right) \right.\\
&\quad\left. -\frac{r-3}{2r}\mu\left( T^{2}A^{(k)}[1,l]\cap T^{2}A^{(k)}[1,l]\cap T^{2}A^{(k)}[1,l]\right) \right|\\
&\quad+\sum_{l=1}^{l_k}\sum_{\substack{1\leq l',l''\leq l_k \\ \#\{l,l',l''\}>1}}\mu_{T^{2q_{n_k}},T^{q_{n_k}}}\left(A^{(k)}[1,l]\times TA^{(k)}[1,l']\times T^{2}A^{(k)}[1,l''] \right)\\
&\quad+\frac{r-3}{2r}\sum_{l=1}^{l_k}\sum_{\substack{1\leq l',l''\leq l_k \\ \#\{l,l',l''\}>1}}\mu\left(T^{2}A^{(k)}[1,l]\cap T^{2}A^{(k)}[1,l']\cap T^{2}A^{(k)}[1,l''] \right)\\
%%%%%%%   5   %%%%%%%
&{=}\sum_{l=1}^{l_k}\Big|\mu_{T^{2q_{n_k}},T^{q_{n_k}}}\big(A^{(k)}[1,l]\times TA^{(k)}[1,l]\times T^{2}A^{(k)}[1,l]\big) -\frac{r-3}{2r}\mu\big( A^{(k)}[1,l]\big) \Big|\\
&\quad+\sum_{l=1}^{l_k}\sum_{\substack{1\leq l',l''\leq l_k \\ \#\{l,l',l''\}>1}}\mu_{T^{2q_{n_k}},T^{q_{n_k}}}\left(A^{(k)}[1,l]\times TA^{(k)}[1,l']\times T^{2}A^{(k)}[1,l''] \right)\\
%%%%%%%   6   %%%%%%%
&\stackrel{\eqref{1bl}}{<} \sum_{l=1}^{l_k}\vep \mu\left(A^{(k)}[1,l] \right)\\
&\quad+\sum_{l=1}^{l_k}\sum_{\substack{1\leq l''\leq l_k \\ l''\neq l}}\mu_{T^{2q_{n_k}},T^{q_{n_k}}}\left(A^{(k)}[1,l]\times TA^{(k)}[1,l]\times T^{2}A^{(k)}[1,l''] \right)\\
&\quad+\sum_{l=1}^{l_k}\sum_{\substack{1\leq l'\leq l_k\\ l'\neq l}}\mu_{T^{2q_{n_k}},T^{q_{n_k}}}\left(A^{(k)}[1,l]\times TA^{(k)}[1,l']\times T^{2}A^{(k)}[1,l] \right)\\
&\quad+\sum_{l=1}^{l_k}\sum_{\substack{1\leq l',l''\leq l_k\\ l',l''\neq l}}\mu_{T^{2q_{n_k}},T^{q_{n_k}}}\left(A^{(k)}[1,l]\times TA^{(k)}[1,l']\times T^{2}A^{(k)}[1,l''] \right)\\
%%%%%%%   7   %%%%%
&=\vep \mu(A[1]) \\
&\quad+\sum_{l=1}^{l_k}\sum_{\substack{1\leq l''\leq l_k \\ l''\neq l}} \mu\left(T^{-2q_{n_k}}A^{(k)}[1,l]\cap T^{-q_{n_k}}TA^{(k)}[1,l] \cap T^{2}A^{(k)}[1,l'']\right) \\
&\quad+\sum_{l=1}^{l_k}\sum_{\substack{1\leq l'\leq l_k \\ l''\neq l}}  \mu\left(T^{-2q_{n_k}}A^{(k)}[1,l]\cap T^{-q_{n_k}}TA^{(k)}[1,l'] \cap T^{2}A^{(k)}[1,l]\right) \\
&\quad+\sum_{l=1}^{l_k}\sum_{\substack{1\leq l',l''\leq l_k \\ l',l''\neq l}} \mu\left(T^{-2q_{n_k}}A^{(k)}[1,l]\cap T^{-q_{n_k}}TA^{(k)}[1,l'] \cap T^{2}A^{(k)}[1,l'']\right) \\
%%%%%%%%
&\leq \vep \mu(A[1]) %\\
+\sum_{l=1}^{l_k}\sum_{\substack{1\leq l''\leq l_k \\ l''\neq l}} \mu\left(T^{-2q_{n_k}}A^{(k)}[1,l] \cap T^{2}A^{(k)}[1,l'']\right)\\
&\quad+2\sum_{l=1}^{l_k}\sum_{\substack{1\leq l'\leq l_k \\ l'\neq l}} \mu\left(T^{-2q_{n_k}}A^{(k)}[1,l] \cap T^{-q_{n_k}}TA^{(k)}[1,l']\right)\\
%%%%%%%%   9  %%%%%%%
&= \vep \mu(A[1]) %\\
+\sum_{l=1}^{l_k}\sum_{\substack{1\leq l''\leq l_k \\ l''\neq l}} \mu\left(A^{(k)}[1,l]\cap T^{2q_{n_k}}T^{2}A^{(k)}[1,l''] \right)  \\
&\quad+2\sum_{l=1}^{l_k}\sum_{\substack{1\leq l'\leq l_k \\ l'\neq l}}\mu\left(A^{(k)}[1,l]\cap T^{q_{n_k}}TA^{(k)}[1,l'] \right)  \\
%%%%%%%%
&= \vep \mu(A[1]) %\\
+\sum_{l''=1}^{l_k}\sum_{\substack{1\leq l\leq l_k \\ l\neq l''}} \mu\left(A^{(k)}[1,l]\cap T^{2q_{n_k}}T^{2}A^{(k)}[1,l''] \right)  \\
&\quad+2\sum_{l'=1}^{l_k}\sum_{\substack{1\leq l\leq l_k \\ l\neq l'}}\mu\left(A^{(k)}[1,l]\cap T^{q_{n_k}}TA^{(k)}[1,l'] \right)  \\
%%%%%%%% tutututu
&= \vep \mu(A[1]) %\\
+\sum_{l''=1}^{l_k} \mu\Big(\Big(\bigcup_{\substack{1\leq l\leq l_k\\ l\neq l''}}A^{(k)}[1,l]\Big) \cap T^{2q_{n_k}}T^{2}A^{(k)}[1,l''] \Big)  \\
&\quad+2\sum_{l'=1}^{l_k} \mu\Big(\Big(\bigcup_{\substack{1\leq l\leq l_k\\ l\neq l'}}A^{(k)}[1,l]\Big) \cap T^{q_{n_k}}TA^{(k)}[1,l'] \Big)  \\
%%%%%%%%
&\stackrel{\eqref{eeprim},\eqref{eeeprim},\eqref{2a}}{\leq}\vep\mu(A[1])% \\
+\sum_{l''=1}^{l_k}\frac{2}{rr_{n_k+2}}\mu(A^{(k)}[1,l''])% \\
+2\sum_{l'=1}^{l_k}\frac{1}{rr_{n_k+2}}\mu(A^{(k)}[1,l']) \\
%%%%%%%%
&=\left(\vep+\frac{4}{rr_{n_k+2}} \right)\mu(A[1]).
%=\vep \mu(A_1^{(k)})\\
%%
%+\mu_{T^{2rq_{n_k}},T^{rq_{n_k}}}\left(\cup_{l=1}^{l_k}\left( A^{(k)}[1,l]\times T^{r/2+1}A^{(k)}[1,l]\right)\times \bigcup_{l''\neq l}T^{r+2}A^{(k)}[1,l''] \right)\\
%%
%+\mu_{T^{2rq_{n_k}},T^{rq_{n_k}}}\left(\cup_{l=1}^{l_k}\left(A^{(k)}[1,l]\times X\times T^{r+2}A^{(k)}[1,l] \right)\cap \left(X\times T^{r/2+1}A^{(k)}[1,l']\times X\right) \right)\\
%%%%%%%
%\stackrel{\text{def. of measure } \mu_{T^{2rq_{n_k}},T^{rq_{n_k}}}}{\leq} \vep \mu(A_1^{(k)})+\mu\left(\left\{x\in \bigcup_{l''\neq l}A^{(k)}[1,l'']\colon T^{2rq_{n_k}}x\in A^{(k)}[1,l] \right\} \right)\\
%%
%+\mu\left(\left\{x\in \bigcup_{l'\neq l}A^{(k)}[1,l']\colon T^{rq_{n_k}}x\in A^{(k)[1,l]} \right\} \right)\\
%%%%%%%%
%\stackrel{\eqref{eeprim},\eqref{eeeprim}}{\leq} \vep \mu(A_1^{(k)})+\frac{2}{sr_{n_k+2}}\sum_{l''\neq l}\mu(A^{(k)}[1,l''])+\frac{1}{sr_{n_k+2}}\mu(A^{(k)}[1,l'])\\
%%
%< \left(\vep+\frac{3}{sr_{n_k+2}} \right)\mu(A[1]).
\end{align*}}%
Hence
\begin{align}\label{A:5}
\begin{aligned}
&\mu_{T^{2q_{n_k}},T^{q_{n_k}}}\left(A[1]\times TA[1]\times T^{2}A[1]\right)\\
&\quad\to \frac{r-3}{2r}\mu_{T^{-2},T^{-1}}\Big(A[1]\times TA[1]\times T^{2}A[1]\Big)
=\frac{r-3}{2r}\mu(A[1]).
\end{aligned}
\end{align}
In a similar way
\begin{align}\label{A:1}
\begin{aligned}
        &\mu_{T^{2q_{n_k}},T^{q_{n_k}}}\left(A[1]\times A[1]\times A[1] \right)\\
        &\quad\to\frac{r-2}{2r}\mu_{I,I}\left(A[1]\times A[1]\times A[1]\right)=\frac{r-2}{2r}\mu\left(A[1]\right),
\end{aligned}
\end{align}
\begin{align}\label{A:2}
\begin{aligned}
        &\mu_{T^{2q_{n_k}},T^{q_{n_k}}}\left(A[1]\times A[1]\times TA[1] \right)\\
        &\quad\to\frac{1}{2r}\mu_{T^{-1},T^{-1}}\left(A[1]\times A[1]\times TA[1]\right)=\frac{1}{2r}\mu\left(A[1]\right),
\end{aligned}
\end{align}
\begin{align}\label{A:3}
\begin{aligned}
     &\mu_{T^{2q_{n_k}},T^{q_{n_k}}}\left(A[1]\times A[1]\times T^{2}A[1] \right)\\
&\quad\to\frac{1}{2r}\mu_{T^{-2},T^{-2}}\left(A[1]\times A[1]\times T^{2}A[1]\right)=\frac{1}{2r}\mu\left(A[1]\right),
\end{aligned}
\end{align}
\begin{align}\label{A:4}
 \begin{aligned}
      & \mu_{T^{2q_{n_k}},T^{q_{n_k}}}\left(A[1]\times TA[1]\times TA[1] \right)\\
      &\quad \to\frac{1}{r}\mu_{T^{-1},I}\left(A[1]\times TA[1]\times TA[1]\right)=\frac{1}{r}\mu\left(A[1]\right),
\end{aligned}
\end{align}
\begin{align}\label{A:6}
\begin{aligned}
       & \mu_{T^{2q_{n_k}},T^{q_{n_k}}}\left(A[1]\times T^{2}A[1]\times T^{3}A[1] \right)\\
       &\quad\to\frac{1}{2r}\mu_{T^{-3},T^{-1}}\left(A[1]\times T^{2}A[1]\times T^{3}A[1]\right)=\frac{1}{2r}\mu\left(A[1]\right).
\end{aligned}
\end{align}
%\begin{multline*}
%        \mu_{T^{2rq_{n_k}},T^{rq_{n_k}}}\left(A[1]\times T^{r/2}A[1]\times T^{r}A[1] \right)\to\\
%   =\frac{s-2}{2s}\mu_{T^{-r-1},T^{-r/2}}(A[1]\times T^{r/2}A[1]\times T^{r}A[1])
%\end{multline*}
%\begin{multline*}
%        \mu_{T^{2rq_{n_k}},T^{rq_{n_k}}}\left(A[1]\times T^{r/2}A[1]\times T^{r+1}A[1] \right)\to\\
%   =\frac{1}{2s}\mu_{T^{-r-1},T^{-r/2}}(A[1]\times T^{r/2}A[1]\times T^{r+1}A[1])
%\end{multline*}
%\begin{multline*}
%        \mu_{T^{2rq_{n_k}},T^{rq_{n_k}}}\left(A[1]\times T^{r/2}A[1]\times T^{r+2}A[1] \right)\to\\
%   =\frac{1}{2s}\mu_{T^{-r-2},T^{-r/2}}(A[1]\times T^{r/2}A[1]\times T^{r+2}A[1])
%\end{multline*}
%\begin{multline*}
%        \mu_{T^{2rq_{n_k}},T^{rq_{n_k}}}\left(A[1]\times T^{r/2+1}A[1]\times T^{r+1}A[1] \right)\to\\
%   =\frac{1}{s}\mu_{T^{-r-1},T^{-r/2-1}}(A[1]\times T^{r/2+1}A[1]\times T^{r+1}A[1])
%\end{multline*}
%\begin{multline*}
%        \mu_{T^{2rq_{n_k}},T^{rq_{n_k}}}\left(A[1]\times T^{r/2+2}A[1]\times T^{r+3}A[1] \right)\to\\
%   =\frac{1}{2s}\mu_{T^{-r-3},T^{-r/2-2}}(A[1]\times T^{r/2+2}A[1]\times T^{r+3}A[1])
%\end{multline*}
Let
%\begin{equation*}
$I=\left\{(0,0),(0,1),(0,2),(1,1),(1,2),(2,3)\right\}.$
%\end{equation*}
Since
$$
\frac{r-3}{2r}+\frac{r-1}{2r}+\frac{1}{2r}+\frac{1}{2r}+\frac{1}{r}+\frac{1}{2r}=1,
$$ we have
$$
\mu_{T^{2q_{n_k}},T^{q_{n_k}}}\left(\bigcup_{(p_2,p_3)\in I}A[1]\times T^{p_2}A[1]\times T^{p_3}A[1] \right) \to \mu\left(A[1]\right).
$$
Notice that for $(p_2',p_3')\not\in I$ such that $T^{p_2'}A[1]$ and $T^{p_3'}A[1]$ are levels of tower $C_{n_{k_0}}$
$$
A[1]\times T^{p_2'}A[1]\times T^{p_3'}A[1] \subset A[1] \times \left(\bigcup_{(p_2,p_3)\in I} T^{p_2}A[1]\times T^{p_3}A[1]\right)^c.
$$
Therefore
\begin{align}\label{A:7}
\begin{aligned}
\mu&_{T^{2q_{n_k}},T^{q_{n_k}}}\left(A[1]\times T^{p_2'}A[1]\times T^{p_3'}A[1]\right)
\\
&\leq \mu_{T^{2q_{n_k}},T^{q_{n_k}}}\Big(A[1]\times \Big(\bigcup_{(p_2,p_3)\in I} T^{p_2}A[1]\times T^{p_3}A[1]\Big)^c \Big)\\
&=\mu\left(A[1]\right)-\mu_{T^{2q_{n_k}},T^{q_{n_k}}}\Big(A[1]\times \Big(\bigcup_{(p_2,p_3)\in I}T^{p_2}A[1]\times T^{p_3}A[1] \Big) \Big)\\
&\to \mu\left(A[1] \right)-\mu\left(A[1] \right)=0.
\end{aligned}
\end{align}
Let now
\begin{multline*}
P=\frac{r-2}{2r}\delta_{(0,0)}+\frac{1}{2r}\delta_{(1,1)}+\frac{1}{2r}\delta_{(2,2)}+\frac{1}{r}\delta_{(1,0)}+\frac{r-3}{2r}\delta_{(2,1)}+\frac{1}{2r}\delta_{(3,1)}.
\end{multline*}
Notice that for $(p_2',p_3')\not\in I$ such that $T^{p_2'}A[1],T^{p_3'}A[1]$ are levels of tower $C_{n_{k_0}}$ we have
\begin{equation}\label{A:8}
\int_{\Z^2}\mu_{T^{-a},T^{-b}}\left(A[1]\times T^{p_2'}A[1]\times T^{p_3'}A[1]\right)\ dP(a,b)=0.
\end{equation}
Using~\cref{A:5,A:1,A:2,A:3,A:4,A:6,A:7,A:8} we obtain
\begin{multline*}
    \mu_{T^{2q_{n_k}},T^{q_{n_k}}}(A[1]\times A[2]\times A[3])\to
\int_{\Z^2} \mu_{T^{-a},T^{-b}}(A[1]\times A[2]\times A[3])\ dP(a,b).
\end{multline*}
This implies~\eqref{atomy} and the claim follows.
\end{proof}

\begin{uw}
In the same way, one can show  that some rank one flows are not
reversible. The construction of such flows is similar to the rank
one automorphisms considered in the above theorem. They are also
determined by a sequence of integers  $(r_n)_{n\in\N}$ which
denote the number of subcolumns at each step of the construction.
Now, the role of spacers is played by rectangles placed above the
subcolumns. The additional assumption in the ``flow version'' of
our theorem is that along the subsequence $(n_k)$ the rectangles
are of fixed height and at steps $n_k,n_{k}+1$ the are placed over
the same subcolumns as in Theorem~\ref{nonrevchacon}.
\end{uw}

\begin{uw}
A similar method can be used to show  non-reversibility of the
classical Chacon's automorphism, i.e. the rank one automorphism
which can be constructed as described in the beginning of this
section, dividing the column at each step of the construction into
three subcolumns and placing a spacer above the middle one. More
precisely, to show that this automorphism is not reversible, one
can use the ``automorphism counterpart'' of
Corollary~\ref{krytNOWEd2prime}.
\end{uw}

%\begin{uw}
%Notice that  essential for this method to work is that in the
%subsequence $(k_n)$ which satisfies the assumptions of the theorem
%we have $k_n+1 \neq 1$ for infinitely many $n$.
%\end{uw}

\section{Topological self-similarities of special flows}\label{sec:topol}
In this section we will deal with topological self-similarities of
continuous flows $\ct=(T_t)_{t\in\R}$ on a compact metric spaces.
For each such flow denote by $I_{top}(\ct)$ the subgroup of all
$s\in\R^*$ such that the flows $\ct$ and $\ct\circ s$ are
topologically conjugate. If $I_{top}(\ct)\nsubseteq\{-1,1\}$ then
the flows $\ct$ is called topologically self-similar. More
precisely we will deal with continuous time changes of minimal
linear flows on the two torus. Each such flow is topologically
conjugate the special flow $T^f$ build over an irrational rotation
$Tx=x+\alpha$ on the circle and under a continuous roof function
$f:\T\to\R_+$.

We will show that if $T^f$ is topologically self-similar then $\alpha$ must be
a quadratic irrational and $f$ is topologically cohomological to a constant
function. It follows that if a continuous time change of a minimal linear
flow on the two torus is topologically self-similar then it is topologically
conjugate to a minimal linear flows on the two torus.

Let $(X,d)$ be a compact connected topological manifold. Denote by
$\widetilde{X}$ the universal covering space of $X$ and let
$\pi':\widetilde{X}\to X$ be the covering map. Denote by $\Theta$
the deck transformation group of the covering
$\pi':\widetilde{X}\to X$, i.e. $\Theta$ is the group of
homeomorphisms $\theta:\widetilde{X}\to\widetilde{X}$ such that
$\pi'\circ\theta=\pi'$. Then $\Theta$ is countable  (isomorphic to
the fundamental group of $X$) and it acts in the properly
discontinuous way, that is, for each $x\in X$ there exists an open
$V\ni x$  such that $\theta (V)\cap V=\emptyset$ whenever $Id\neq
\theta\in\Theta$. In what follows we need the following simple
observation.

\begin{lm}\label{km}
Assume that  $Z$ is a topological space and let $G$ be a countable
group (considered with the discrete topology). Assume that $G$
acts on $Z$ as homeomorphisms in the properly discontinuous way.
Assume moreover that $\Phi:\R\to G$ and that
\[
 \R\ni t\mapsto \Phi(t)z\in Z
\]
is continuous  for each $z\in Z$. Then $\Phi$ is continuous, hence
constant.
\end{lm}

\begin{proof}
Fix $t_0\in\R$ and $z\in Z$. Select an open $V\ni z$,  so that
$gV\cap V=\emptyset$ whenever $1\neq g\in G$. By the continuity
assumption, there is an open interval $W\ni t_0$ such that $
\Phi(t)(z)\in \Phi(t_0)(V)$ for $t\in W$. Thus
$\Phi(t)=\Phi(t_0)$. It follows directly that the map $W\ni
t\mapsto \Phi(t)$ is constant, whence $\Phi$ is continuous.
\end{proof}

Let $T:X\to
X$ be a homeomorphism and  $f:X\to\R\setminus\{0\}$ a continuous
function, which is globally either positive or negative. Let us consider
the skew product $T_{-f}:X\times\R\to
X\times\R$, $T_{-f}(x,r)=(Tx,r-f(x))$ and the orbit equivalence relation
$\equiv$ on $X\times\R$ defined by $(x_2,r_2)\equiv(x_1,r_1)$ if there
exists $n\in\Z$ such that $(x_2,r_2)=T^n_{-f}(x_1,r_1)$. Denote by
$X^f$ the quotient space $(X\times\R)/\equiv$. Then $X^f$ is a
compact topological manifold and the canonical projection
$\pi_1=\pi^f_1:X\times \R\to X^f$ is a covering map.

Let us consider the (continuous) flow $(\sigma_t)_{t\in\R}$ on
$X\times\R$ given by $\sigma_t(x,r)=(x,r+t)$. Since $T^n_{-f}$
commutes with $\sigma_t$ for every $n\in\Z$ and $t\in\R$, each
$\sigma_t$ transforms the equivalence classes for $\equiv$ into the
equivalence classes. Therefore $(\sigma_t)_{t\in\R}$ defines a
continuous flow on $X^f$, this flow is denoted
by $T^f$. If the function $f$ is positive the flow $T^f$ is called the
special flow built over the homeomorphism $T$ and under the roof function $f$.
Of course, $T^f_t\circ\pi_1=\pi_1\circ \sigma_t$ for
every $t\in\R$.

 Then, the map
$\pi_2:\widetilde{X}\times\R\to X\times\R$ given by
$\pi_2(\widetilde{x},r)=(\pi'(\widetilde{x}),r)$ is a covering map
and $\widetilde{X}\times\R$ is the universal covering of $X^f$
with the covering map $\pi^f=\pi_1^f\circ\pi_2$.

 For every $\theta\in\Theta$ denote by
$\underline{\theta}:\widetilde{X}\times\R\to\widetilde{X}\times\R$
the trivial extension
$\underline{\theta}(\widetilde{x},r)=({\theta}(\widetilde{x}),r)$.
Note that $\underline{\theta}$ belongs to the deck group of the
universal covering $\pi^f$.

Let $\widetilde{T}:\widetilde{X}\to\widetilde{X}$ be a lift of
$T:X\to X$. Recall that
$\widetilde{T}:\widetilde{X}\to\widetilde{X}$ is a homeomorphism.
Let us consider the group automorphism $\gamma:\Theta\to\Theta$
given by
\[\gamma(\theta)=\widetilde{T}\circ\theta\circ\widetilde{T}^{-1}\]
and the semidirect product $\Theta\rtimes_{\gamma}\Z$ with
multiplication
\[(\theta,m)\cdot(\theta',m')=(\theta\circ\gamma^m(\theta'),m+m').\]
Denote by
$\widetilde{T}_{-\widetilde{f}}:\widetilde{X}\times\R\to\widetilde{X}\times\R$
the skew product
$\widetilde{T}_{-\widetilde{f}}(\widetilde{x},r)=
(\widetilde{T}(\widetilde{x}),r-\widetilde{f}(\widetilde{x}))$,
where $\widetilde{f}=f\circ\pi'$.

\begin{pr}[Proposition~1.1 in \cite{Ke-Ma-Se}]
The deck transformation group $\Theta^f$ of the universal covering
$\pi^f:\widetilde{X}\times\R\to X^f$ is equal to
\[\{\underline{\theta}\circ \widetilde{T}^m_{-\widetilde{f}}:\theta\in\Theta,\;m\in\Z\}\]
and $(\theta,m)\mapsto \underline{\theta}\circ \widetilde{T}^m_{-\widetilde{f}}$
establishes the group isomorphism of $\Theta\rtimes_{\gamma}\Z$
and $\Theta^f$.
\end{pr}
We will identify the groups $\Theta\rtimes_{\gamma}\Z$ and
$\Theta^f$.

For any $s\in\R^*\setminus\{1\}$ let us consider the flow
$T^f\circ (s^{-1})=(T^f_{s^{-1}t})_{t\in\R}$ on $X^f$. This flow is topologically
isomorphic to the flow $(T^{sf}_{t})_{t\in\R}$ on
$X^{sf}$. Indeed, the homeomorphism $U:X\times\R\to X\times\R$
given by $U(x,r)=(x,sr)$ satisfies \[U\circ T_{-f}=T_{-sf}\circ
U\quad\text{ and }\quad U\circ\sigma_{s^{-1}t}=\sigma_t\circ U.\]
Therefore, $U$ induces a homeomorphism $U':X^f\to X^{sf}$ with
$U'\circ T^f _{s^{-1}t}=T^{sf}_t\circ U'$.

Suppose that $s\in I_{top}(T^f)\setminus\{1\}$. As the flows
$T^f$ and $T^f\circ (s^{-1})$ on $X^f$ are
topologically isomorphic,  the flows
$(T^f_{t})_{t\in\R}$ on $X^f$ and $(T^{sf}_{t})_{t\in\R}$ on
$X^{sf}$ are also topologically isomorphic. Thus  there exists a
homeomorphism $S:X^f\to X^{sf}$ such that $S\circ T^f
_{t}=T^{sf}_t\circ S$. Let $\widetilde{S}:\widetilde{X}\times\R\to
\widetilde{X}\times\R$ be a lift of $S$. Then $\widetilde{S}$ is a
homeomorphism such that
${S}\circ\pi^f=\pi^{sf}\circ\widetilde{S}$. Since
$S^{-1}\circ T^{sf}_{-t}\circ S\circ T^f _{t}= id_{X^f}$, its lift
$\widetilde{S}^{-1}\circ \sigma_{-t}\circ \widetilde{S}\circ
\sigma_{t}$ is an element of the deck transformation group
$\Theta^f$, so there exists a map $\R\ni t\mapsto (\theta(t),m(t))\in
\Theta\rtimes_{\gamma}\Z$ such that
\[\widetilde{S}^{-1}\circ \sigma_{-t}\circ \widetilde{S}\circ
\sigma_{t}=\underline{\theta}(t)\circ \widetilde{T}^{m(t)}_{-f}.\]
Now, for each $(\widetilde{x},r)\in\widetilde{X}\times\R$ the map
$$
\R\ni t\mapsto \widetilde{S}^{-1}\circ \sigma_{-t}\circ \widetilde{S}\circ
\sigma_{t}(\widetilde{x},r)\in \widetilde{X}\times\R$$
is continuous. By Lemma~\ref{km} applied to $\Theta^f$ we obtain that the
map $t\mapsto (\theta(t),m(t))$ is constant.
Moreover, $(\theta(0),m(0))=(id_X,0)$, so
\begin{equation}\label{eq:comsigma}
\widetilde{S}\circ \sigma_{t}=\sigma_{t}\circ\widetilde{S}.
\end{equation}
For every $\widetilde{\theta}\in\Theta^{f}$ the homeomorphism
$\widetilde{S}\circ\widetilde{\theta}\circ\widetilde{S}^{-1}$ is a
deck transformation of $\pi^{sf}$, so there exists
$A:\Theta^{f}\to\Theta^{sf}$ such that
\begin{equation}\label{eq:comtheta}
\widetilde{S}\circ\widetilde{\theta}\circ\widetilde{S}^{-1}=A(\widetilde{\theta}).
\end{equation}
Moreover, $A:\Theta^{f}\to\Theta^{sf}$ is a group isomorphism
which can be identified with the automorphism $A:
\Theta\rtimes_{\gamma}\Z\to \Theta\rtimes_{\gamma}\Z$.

Let $\widetilde{S}=(S_1,S_2)$, where $S_1:\widetilde{X}\times\R\to
\widetilde{X}$ and $S_2:\widetilde{X}\times\R\to\R$. Let
$A=(A_1,A_2)$, where $A_1:\Theta\rtimes_{\gamma}\Z\to \Theta$ and
$A_2:\Theta\rtimes_{\gamma}\Z\to\Z$. In view of
\eqref{eq:comsigma},
\[\widetilde{S}(\widetilde{x},t)=\widetilde{S}\circ\sigma_t(\widetilde{x},0)
=
\sigma_t\circ\widetilde{S}(\widetilde{x},0)=(S_1(\widetilde{x},0),S_2(\widetilde{x},0)+t).\]
Therefore
\[\widetilde{S}(\widetilde{x},t)=(V(\widetilde{x}),t+g(\widetilde{x})),\]
where $V:\widetilde{X}\to \widetilde{X}$ is a homeomorphism and
$g:\widetilde{X}\to\R$ is a continuous function. Note that if
$(\theta,m)$ denotes $\underline{\theta}\circ
\widetilde{T}^m_{-\widetilde{f}}$ then
\begin{align*}
\widetilde{S}\circ(\theta,m)(\widetilde{x},r)&=\widetilde{S}
(\theta\circ
\widetilde{T}^m(\widetilde{x}),r-\widetilde{f}^{(m)}(\widetilde{x}))\\&=
(V\circ\theta\circ
\widetilde{T}^m(\widetilde{x}),r-\widetilde{f}^{(m)}(\widetilde{x})+g(\theta\circ
\widetilde{T}^m(\widetilde{x}))),
\end{align*}
while if we  set $(\theta,m)=\underline{\theta}\circ \widetilde{T}^m_{-s\widetilde{f}}$ then
\begin{align*}
(\theta,m)\circ\widetilde{S}(\widetilde{x},r)&=
(\theta,m)(V(\widetilde{x}),r+g(\widetilde{x}))\\
&= (\theta\circ \widetilde{T}^m\circ
V(\widetilde{x}),r+g(\widetilde{x})-s\widetilde{f}^{(m)}(V(\widetilde{x})))
\end{align*}
Therefore,
in view of \eqref{eq:comtheta}, we have
\begin{gather*}
V\circ\theta\circ \widetilde{T}^m(\widetilde{x})=
A_1(\theta,m)\circ \widetilde{T}^{A_2(\theta,m)}\circ
V(\widetilde{x})\\
g(\theta\circ
\widetilde{T}^m(\widetilde{x}))-\widetilde{f}^{(m)}(\widetilde{x})
=g(\widetilde{x})-s\widetilde{f}^{(A_2(\theta,m))}(V(\widetilde{x})).
\end{gather*}
Let us consider the action of the group $\Theta\rtimes_{\gamma}\Z$
on $\widetilde{X}$ defined by
$(\theta,m)(\widetilde{x})=\theta\circ
\widetilde{T}^m(\widetilde{x})$. Then as a conclusion we have the
following.

\begin{tw}\label{th:podsim}
The number $s\in I_{top}(T^f)\setminus\{1\}$ if and only if there
exist a homeomorphism $V:\widetilde{X}\to\widetilde{X}$, a group
automorphism $A: \Theta\rtimes_{\gamma}\Z\to
\Theta\rtimes_{\gamma}\Z$ and a continuous function
$g:\widetilde{X}\to\R$ such that for every
$(\theta,m)\in\Theta\rtimes_{\gamma}\Z$
\begin{gather*}
V\circ(\theta,m)(\widetilde{x})= A(\theta,m)\circ
V(\widetilde{x})\\
s\widetilde{f}^{(A_2(\theta,m))}(V(\widetilde{x}))-\widetilde{f}^{(m)}(\widetilde{x})
=g(\widetilde{x})-g((\theta,m)(\widetilde{x})).
\end{gather*}
\end{tw}

\begin{uw} If $T$ is uniquely ergodic, then so is $T^f$ and therefore in this case
$I_{top}(T^f)\subset I_{T^f}$.\end{uw}

\subsection{Special flows over irrational rotations}
Suppose that $T$ is the rotation by an irrational number
$\alpha\in \R$ on the additive circle $X=\T=\R/\Z$. Then $\widetilde{X}=\R$
and the deck transformation group is the group of translations of
$\R$ by integer numbers, so $\Theta=\Z$ with $n(x)=x+n$. As each
such translation commutes with the lift $\widetilde{T}:\R\to\R$,
$\widetilde{T}x=x+\alpha$, $\gamma=id_{\Z}$. Thus
$\Theta^f=\Z\times\Z$ and the action of this group on
$\widetilde{X}=\R$ is given by
\begin{equation}\label{lll1}(n,m)x=x+n+m\alpha.\end{equation}

We will now prove the following result describing  topological
self-similarities of $T^f$ whose second part is to be compared
with Remark~\ref{uwagarev2}.

\begin{pr}\label{prop:basicrotation}
Let $\alpha\in\R\setminus\Q$ and let $f:\T\to\R$ be a continuous
positive function. Then $s\in I_{top}(T^f)\setminus\{1\}$ if and
only if there exist a matrix $[a_{ij}]=A\in GL_2(\Z)$,
$\delta\in\R$ and a continuous function $g:\R\to\R$ such that
\begin{gather}
\label{zalalpha}
a_{12}+a_{22}\alpha =(a_{11}+a_{21}\alpha)\alpha,\\
\label{zalalphass}
a_{11}+a_{21}\alpha=\sigma s=\sigma(a_{22}-a_{21}\alpha)^{-1}\\
\label{linetwokon} s\widetilde{f}^{(a_{21}n+a_{22}m)}(\sigma
sx+\delta)-\widetilde{f}^{(m)}({x}) =g({x})-g(x+n+m\alpha),
\end{gather}
for all $m,n\in\Z$, where $\sigma=\det A$.

Moreover, $-1\in I_{top}(T^f)$ if and only if there exist a
continuous map $g:\T\to\R$ and $\delta\in\T$ such that
\begin{equation}\label{eq:odw}
f(\delta-x)-f(x)=g(x)-g(x+\alpha).
\end{equation}
\end{pr}

\begin{proof}
By Theorem~\ref{th:podsim} and~(\ref{lll1}), $s\in I_{top}(T^f)\setminus\{1\}$ if
and only if there exist a homeomorphism $V:\R\to\R$, a group
automorphism \[
A: \Z^2\to \Z^2,\ A(n,m)=(a_{11}n+a_{12}m,a_{21}n+a_{22}m)\ \text{ with }\
A:=[a_{ij}]\in
GL_2(\Z)\] and a continuous function $g:\R\to\R$ such that for
every $(n,m)\in\Z^2$ and $x\in\R$
\begin{gather}
V(x+n+m\alpha)= V(x)+(a_{11}n+a_{12}m)+(a_{21}n+a_{22}m)\alpha\label{lineone}\\
s\widetilde{f}^{(a_{21}n+a_{22}m)}(V({x}))-\widetilde{f}^{(m)}({x})
=g({x})-g(x+n+m\alpha).\label{linetwo}
\end{gather}
Let us consider $v:\R\to\R$, $v(x)=V(x)-(a_{11}+a_{21}\alpha)\,x$.
In view of~\eqref{lineone}, we obtain
\begin{align*}
v(x+n+m\alpha)&=V(x+n+m\alpha)-(a_{11}+a_{21}\alpha)(x+n+m\alpha)\\
&=v(x)+(a_{12}+a_{22}\alpha-(a_{11}+a_{21}\alpha)\alpha)\,m.
\end{align*}
It follows that $v$ is $\Z$-periodic, so $v:\T\to\R$. Moreover (by taking $n=0$ and $m=1$ above),
$v(x+\alpha)=v(x)+a_{12}+a_{22}\alpha-(a_{11}+a_{21}\alpha)\alpha$,
so
\[\int_{\T}v(x)\,dx=\int_{\T}v(x+\alpha)\,dx=\int_{\T}v(x)\,dx+a_{12}+a_{22}\alpha-(a_{11}+a_{21}\alpha)\alpha.\]
Thus $a_{12}+a_{22}\alpha=(a_{11}+a_{21}\alpha)\alpha$, so \eqref{zalalpha} holds.
Moreover, $v$ is a
constant function and $V(x)=\gamma x+\delta$ with
$\gamma:=a_{11}+a_{21}\alpha$ and  some real $\delta$.

\emph{Case 1.} Suppose that $a_{12}$ or $a_{21}$ is equal to zero.
As $\alpha$ is irrational and by~(\ref{zalalpha}),
$a_{12}+(a_{22}-a_{11})\alpha-a_{21}\alpha^2=0$, it follows that
$a_{12}=a_{21}=0$ and $a_{11}=a_{22}=\pm 1$. Hence $V(x)=\pm
x+\delta$ and, by \eqref{linetwo},
\[s\widetilde{f}^{(\pm m)}(\pm x+\delta)-\widetilde{f}^{(m)}({x})
=g({x})-g(x+n+m\alpha).\]
Setting $m=0$ we have $g(x+n)=g(x)$, so
$g$ is $\Z$-periodic. Therefore, $g$ can be treated as a map on $\T$ and taking $m=1$
we have
\[sf^{(\pm 1)}(\pm x+\delta)-f(x)=g(x)-g(x+\alpha)\quad\text{ for all }x\in\T.\]
Recalling that $f^{(-1)}(y)=-f(y-\alpha)$, it follows that
\begin{align*}
(\pm s-1)\int_{\T}f(x)\,dx&=s\int_{\T}f^{(\pm 1)}(\pm
x+\delta)\,dx-\int_{\T}f(x)\,dx\\
&=\int_{\T}g(x)\,dx-\int_{\T}g(x+\alpha)\,dx=0.
\end{align*}
Since $s\neq 1$ and $f$ is positive, it follows that $s=-1$ and
$a_{11}=a_{22}=-1$. Therefore, \eqref{zalalphass}, \eqref{linetwokon}
and \eqref{eq:odw} hold.

\emph{Case 2.} Suppose that both $a_{12}$ and $a_{21}$ are
non-zero. Since $a_{12}+(a_{22}-a_{11})\alpha-a_{21}\alpha^2=0$,
the irrational number $\alpha$ is a quadratic irrational. In view
of \eqref{linetwo} (by substituting $m$ by $0$,  and by
substituting $n$ by $-a_{22}m$ and $m$ by $a_{21}m$,
respectively), we obtain
\begin{equation}\label{lll2}
s\widetilde{f}^{(a_{21}n)}(\gamma x+\delta)
=g({x})-g(x+n),\end{equation}
\begin{equation}\label{lll3}
-\widetilde{f}^{(a_{21}m)}({x})
=g({x})-g(x+(a_{21}\alpha-a_{22})m).
\end{equation}

It follows that for every $x\in\R$
\begin{equation}\label{zbilroz}
\lim_{|y|\to\infty}\frac{g(x)-g(x+y)}{y}=sa_{21}\int_{\T} f\,dx.
\end{equation}
Indeed, if $|y|$ is large enough
\begin{align*}
\frac{g(x)-g(x+y)}{y}&=\frac{g(x)-g(x+\{y\})}{y}+\frac{g(x+\{y\})-g(x+\{y\}+[y])}{[y]}\frac{[y]}{y}\\
&\stackrel{\eqref{lll2}}{=}\frac{g(x)-g(x+\{y\})}{y}+sa_{21}\frac{f^{(a_{21}[y])}(V(x+\{y\}))}{a_{21}[y]}\frac{[y]}{y}.
\end{align*}
Since $|g(x)-g(x+\{y\})|\leq 2\|g\|_{C[x,x+1]}$, $f^{(n)}/n$ tends
to $\int_{\T} f\,dx$ uniformly and $[y]/y\to 1$ as $|y|\to\infty$,
we get \eqref{zbilroz}. Furthermore (by taking $y=(a_{21}\alpha-a_{22})m$ in~(\ref{zbilroz})),
\begin{align*}
\frac{-\int_{\T} f\,dx}{a_{21}\alpha-a_{22}}
\leftarrow&\frac{-\widetilde{f}^{(a_{21}m)}({x})/a_{21}m}{(a_{21}\alpha-a_{22})}\\&
\stackrel{\eqref{lll3}}{=}
\frac{g(x)-g(x+(a_{21}\alpha-a_{22})m)}{a_{21}(a_{21}\alpha-a_{22})m}\to
s\int_{\T} f\,dx,
\end{align*}
hence $s=(a_{22}-a_{21}\alpha)^{-1}$.

In view of \eqref{zalalpha}, $(1,\alpha)A=\gamma(1,\alpha)$, so
$(1,\alpha)A^{-1}=\gamma^{-1}(1,\alpha)$. Moveover,
\[A^{-1}=\sigma
           \begin{pmatrix}
             a_{22} & -a_{12} \\
             -a_{21} & a_{11}
           \end{pmatrix}
         ,\text{ where }\sigma:=\det A=\pm1.
\]
It follows that $\gamma^{-1}=\sigma(a_{22}-a_{21}\alpha)$, hence
$\gamma=\sigma s$, this yields \eqref{zalalphass}.
Therefore, $V(x)=\sigma sx+\delta$, so
\eqref{linetwo} gives \eqref{linetwokon}.

It follows that if $-1\in I_{top}(T^f)$ then $a_{12}$ or $a_{21}$
is equal to zero. Otherwise, using Case 2 we have $a_{11}+\alpha
a_{21}=\gamma=-\sigma$, so $a_{21}=0$, a contradiction. Moreover, by Case 1, this
yields the second part of the proposition.
\end{proof}

\begin{wn}
If $\alpha\in\R\setminus \Q$ is not a quadratic irrational then
$I_{top}(T^f)\subset\{1,-1\}$.
\end{wn}

\begin{wn}\label{cor:roznica}
There exists a continuous time change $(\varphi_t)_{t\in\R}$ of a minimal
linear flow on $\T^2$ such that $I((\varphi_t)_{t\in\R})=\R^*$ and
$I_{top}((\varphi_t)_{t\in\R})=\{1\}$.
\end{wn}

\begin{proof}[Proof of Corollary~\ref{cor:roznica}]
On the modular space $\Gamma\backslash PSL_2(\R)$,
$\Gamma=PSL_2(\Z)$, by Corollary~\ref{cor:horrev} and Corollary~\ref{cor:horrev1},
$I((h_t)_{t\in\R})=\R^*$ and $C((h_t)_{t\in\R})=\{h_t:\:t\in\R\}$.
As it was shown in \cite{RatLB} that this flow is loosely Bernoulli, so $(h_t)_{t\in\R}$
is isomorphic to a special flow over any irrational rotation $Tx=x+\alpha$ on the circle,
see \cite{KatMon}, \cite{Or-Ru-We}. Moreover, the roof function $f:\T\to\R_+$ can
be chosen continuous, see \cite{Ko9}.

If $\alpha$ is not a quadratic irrational, then
$I_{top}(T^f)\subset\{-1,1\}$. On the other hand, since $T^f$ is
measure-theoretically isomorphic with $(h_t)_{t\in\R}$, we have
$I(T^f)=\R^*$ and $C(T^f)=\{T^f_t:\:t\in\R\}$.  Moreover, $T^f$ is
a special representation of a continuous time change of a minimal
linear flow on $\T^2$.

Now we will see that $f$ can be chosen so that $-1\notin I_{top}(T^f)$ which will
finish the proof.
Suppose that there exists a special representation $T^f$ of the horocycle
flow $(h_t)_{t\in\R}$ such that $-1\in I_{top}(T^f)$.
Then we can construct another continuous function $f':\T\to\R$ such that
$T^{f'}$ is isomorphic to $T^f$ and $-1\notin I_{top}(T^{f'})$.

Since $-1\notin I_{top}(T^f)$, by Proposition~\ref{prop:basicrotation}, there exist
$\delta\in\T$ and $g:\T\to\R$ continuous such that
\begin{equation}\label{eq:coheqnaf}
f(\delta-x)-f(x)=g(x)-g(x+\alpha).
\end{equation}
Let $j:\T\to\R$ be a measurable map such that $x\mapsto j(x)-j(x+\alpha)$ is continuous
and
\begin{equation}\label{zal:naj}
\text{$x\mapsto j(x)+j(\delta-x)$ is not a.e.\ equal to any continuous function; }
\end{equation}
the existence of such a map will be discussed at the end of the proof.

We claim that if $f'=f+j-j\circ T$ then $-1\notin I_{top}(T^{f'})$.
Otherwise, by Proposition~\ref{prop:basicrotation}, there exist
$\delta'\in\T$ and $g':\T\to\R$ continuous such that
\[f'(\delta'-x)-f'(x)=g'(x)-g'(x+\alpha).\]
In view of \eqref{eq:coheqnaf}, it follows that
\begin{multline*}
f(\delta-x)-f(\delta'-x)=(f(\delta-x)-f(x))-(f'(\delta'-x)-f'(x))\\
\quad-(j(x)-j(x+\alpha))+(j(\delta'-x)-j(\delta'-x+\alpha))\\
=((g-g'-j)(x)-(g-g'-j)(x+\alpha))+(j(\delta'-x)-j(\delta'-x+\alpha)).
\end{multline*}
Replacing $x$ by $\delta'-x$ we have
\[f(x+\delta-\delta')-f(x)=h(x)-h(x+\alpha),\]
with
\[h(x)=(j+g'-g)(\delta'-x+\alpha)+j(x).\]
Since $C(T^f)=\{T^f_t:\:t\in\R\}$, in view of Lemma~\ref{lm:trivcent},
there exist $k\in\Z$ and $t_0\in\R$ such that $\delta-\delta'=k\alpha$
and $h=t_0-f^{(k)}$ a.e. Therefore
\[j(\delta'-x+\alpha)+j(x)=(g-g')(\delta'-x+\alpha)+t_0-f^{(k)}(x)\quad \text{ a.e.}\]
Moreover,
\begin{align*}j(\delta-x)-j(\delta'-x+\alpha)&=j(\delta-x)-j(\delta-x-(k-1)\alpha)\\
&=(f')^{(-k+1)}(\delta-x)-f^{(-k+1)}(\delta-x),
\end{align*}
Adding both equations we obtain that the map $x\mapsto
j(\delta-x)+j(x)$ is a.e.\ equal to a continuous map, contrary to
\eqref{zal:naj}.

Finally we point out a measurable map $j:\T\to\R$ such that $j-j\circ T$ is continuous
and satisfies \eqref{zal:naj}. Let $(q_n)_{n\geq 1}$ be a subsequence of
denominators of $\alpha$ such that $q_{n+1}\geq 2q_n$ for $n\geq 1$.
Let us consider an $L^2$ map $j:\T\to\R$ with the Fourier series
\[j(x)=\sum_{n\geq 1}\frac{1}{n}\cos2\pi q_n(x-\delta/2).\]
Since
\[j(x)-j(x+\alpha)=\sum_{n\geq 1}\frac{2}{n}\sin2\pi q_n(x+\alpha/2-\delta/2)\sin\pi q_n\alpha\]
with $|\sin\pi q_n\alpha|\leq\|q_n\alpha\|<1/q_{n+1}<1/2^n$ for $n\geq 1$,
we can choose $j$ such that $j-j\circ T$ is continuous.
Moreover, $j(\delta-x)=j(x)$ for a.e.\ $x\in\T$. Therefore, we need to show
that $j$ (or equivalently $j_\delta(x)=j(x+\delta/2)$) is not a.e.\ equal to any continuous function.
Some elementary arguments show that the Fourier series of $j_\delta$ is not Ces\`aro summable at $0$.
Then, by classical Fejer's theorem, $j_\delta$ is not a.e.\ equal to any continuous function,
which completes the proof.
\end{proof}

The following lemma is easily obtained by induction.

\begin{lm}\label{lem:pomoc}
Let $0<|\rho|<1$ and let $(x_n)_{n\geq 0}$ be a real sequence
such that \[|\rho x_{n+1}- x_n|\leq M\quad\text{ for }\quad n\geq 0.\] Then
\[|\rho^nx_n- x_0|\leq \frac{1-|\rho|^{n}}{1-|\rho|}M\leq \frac{M}{1-|\rho|}\quad\text{ for }\quad n\geq 0.\]
\end{lm}

\begin{tw}
If there exists $s\in I_{top}(T^f)\setminus\{-1,1\}$ then $f$ is
cohomologous to a constant function via a continuous transfer
function.
\end{tw}

\begin{proof}
Without loss of generality we can assume that $|s|<1$. By
Proposition~\ref{prop:basicrotation}, there exist a matrix
$[a_{ij}]=A\in GL_2(\Z)$, $\delta\in\R$ and a continuous function
$g:\R\to\R$ satisfying  \eqref{zalalpha}-\eqref{linetwokon}. Let
us consider
\[F(x):=\widetilde{f}(x)-\int_{\T}{f}(t)\,dt\quad\text{ and }\quad
G(x):=g(x)+xsa_{21}\int_{\T}{f}(t)\,dt.\]
In view of
\eqref{linetwokon} and \eqref{zalalphass},
\begin{align}\label{eq:fin1}
sF^{(a_{21}n+a_{22}m)}(\sigma s x+\delta)-F^{(m)}({x})
=G({x})-G(x+n+m\alpha).
\end{align}
with (remembering that $F$ is $1$-periodic) $\int_{\T}F(t)\,dt=0$.

Choose any $x_0\in[0,1]$ and $m_0\geq \frac{2|sa_{21}|}{1-|s|}$. Let us define
inductively three sequences: $(x_k)_{k\geq 0}$ taking values in $[0,1)$ and
two other integer-valued sequences $(m_k)_{k\geq 0}$, $(n_k)_{k\geq 0}$ so
that:
\[n_{k}:=-[x_k+m_k\alpha],\quad m_{k+1}:=a_{21}n_k+a_{22}m_k,\quad x_{k+1}:=\{\sigma s x_k+\delta\}\]
for all $k\geq 0$. In view of \eqref{eq:fin1}, it follows that
\[sF^{(m_{k+1})}(x_{k+1})-F^{(m_k)}({x_k})
=G(x_k)-G(x_k+n_k+m_k\alpha)\] and
$x_k+n_k+m_k\alpha=\{x_k+m_k\alpha\}\in[0,1)$. Therefore,
\[|sF^{(m_{k+1})}(x_{k+1})-F^{(m_k)}({x_k})|\leq C:=2\max_{x\in[0,1]}|G(x)|.\]
By Lemma~\ref{lem:pomoc},
\begin{equation}\label{eq:diffs}
|s^kF^{(m_{k})}(x_{k})-F^{(m_0)}({x_0})|\leq \frac{C}{1-|s|}\quad\text{ for }\quad k\geq 0.
\end{equation}
Moreover, as $s(a_{22}-a_{21}\alpha)=1$ (see \eqref{zalalphass}), we have
\begin{align*}sm_{k+1}-m_k&=sa_{21}n_k+(sa_{22}-1)m_k=-sa_{21}[x_k+m_k\alpha]+(sa_{22}-1)m_k\\
&= -sa_{21}(x_k+m_k\alpha)+(sa_{22}-1)m_k+sa_{21}\{x_k+m_k\alpha\}
\\
&=
\big(s(a_{22}-a_{21}\alpha)-1\big)m_k+sa_{21}\big(\{x_k+m_k\alpha\}-x_k\big)\\
&= sa_{21}\big(\{x_k+m_k\alpha\}-x_k\big).
\end{align*}
Hence $|sm_{k+1}-m_k|\leq |sa_{21}|$ for $k\geq 0$. In view of
Lemma~\ref{lem:pomoc}, it follows that
\[|s^km_k-m_0|\leq \frac{|sa_{21}|}{1-|s|},\]
so the sequence $(s^km_k)_{k\geq 0}$ is bounded and (by the lower bound of $m_0$) bounded away from zero. Thus
$|m_k|\to\infty$ as $k\to\infty$.

By the unique ergodicity of the rotation $T$, the sequence
$F^{(n)}/n$ tends uniformly to $\int_{\T}F(t)\,dt=0$ as
$|n|\to\infty$. It follows that,
\[s^kF^{(m_{k})}(x_{k})=s^km_k\frac{F^{(m_{k})}(x_{k})}{m_k}\to 0.\]
Therefore, passing to $k\to\infty$ in \eqref{eq:diffs}, we
have $|F^{(m_0)}({x_0})|\leq C/(1-|s|)$. Consequently,
$\|F^{(m)}\|_{\sup}\leq C/(1-|s|)$ for every $m\geq 1$. In view of
the classical Gottschalk-Hedlund theorem (Theorem 14.11 in \cite{Go-He}), $F=f-\int_{\T}f(t)\,dt$
is a coboundary with a continuous transfer function.
\end{proof}

\begin{uw} Consider the quadratic number $\alpha\in(0,1)$
satisfying $\frac1\alpha=\alpha+1$. Let $\cT=(T_t)_{t\in\R}$ be
the linear flow on $\T^2$ given by $(\alpha,1)$, that is
$T_t(x,y)=(x+t\alpha,y+t)$. Then $1/\alpha\in I_{top}(\cT)$.
Indeed, the rescaled flow $\cS=(S_t)_{t\in\R}$ is given by the
formula $S_t(x,y)=(x+t,y+\frac1\alpha t)$ and it is easy to see
that the homeomorphism $A:\T^2\to\T^2$ given by the matrix
$A=\left[\begin{array}{ll} 0 & 1\\1 & 1\end{array}\right]$
satisfies $A\circ T_t=S_t\circ A$ for each $t\in\R$.\end{uw}

%\vspace{5mm}
%
%\noindent Krzysztof Fr\k{a}czek, Joanna Ku\l aga,
%Mariusz Lema\'nczyk\\
%Faculty of Mathematics and Computer Science\\
%Nicolaus Copernicus University\\
%12/18 Chopin street, 87--100 Toru\'n,
%Poland\\
%fraczek@mat.uni.torun.pl, joanna.kulaga@gmail.com, mlem@mat.uni.torun.pl


\begin{thebibliography}{99}
\bibitem{Aa-Le-Ma-Na}J.\ Aaronson,  M.\ Lema\'nczyk, C.\ Mauduit, H.\ Nakada,
\emph{Koksma inequality and group extensions of Kronecker
transformations}, in Algorithms, Fractals and Dynamics edited by
Y. Takahashi, Plenum Press 1995, 27-50.

\bibitem{An}H.\ Anzai, \emph{Ergodic skew product transformations on the torus},
Osaka Math. J. \textbf{3} (1951), 83-99.

\bibitem{Bo-Sa-Zi} J.\ Bourgain, P.\ Sarnak, T.\ Ziegler, {\em Disjointness of Mobius from horocycle flows},
\url{arXiv:1110.0992}.

\bibitem{Co-Fo-Si} I.P. Cornfeld, S.V. Fomin, Y.G. Sinai,
\emph{Ergodic Theory}, Springer-Verlag, New York, 1982.

\bibitem{Da-Le}A.\ Danilenko, M. Lema\'nczyk, \emph{Spectral
mutiplicities for ergodic flows}, to appear in Discrete Continuous
Dynam.\ Systems, \url{arXiv:1008.4845}.

\bibitem{Da-Ry}A.\ Danilenko, V.V.\ Ryzhikov, \emph{On self-similarities of ergodic
flows}, Proc. London Math. Soc. \textbf{104} (2012), 431-454.


\bibitem{Ei-Wa} M.\ Einsiedler, T.\ Ward, {\em Ergodic Theory:
With a View Towards Number Theory}, Graduate Texts in Mathematics, Springer 2010.

%\bibitem{Fr-dens}K. Fr�czek, \emph{Density of mild mixing property for vertical flows of %abelian differentials},
%Proc. Amer. Math. Soc. \textbf{137} (2009), 4229-4142.

\bibitem{Fr-Le}K. Fr\k{a}czek, M. Lema\'nczyk,
\emph{ A class of special flows over irrational rotations which is
disjoint from mixing flows}, Ergodic Theory Dynam. Systems
\textbf{24} (2004),  1083-1095.

%\bibitem{Fr-Le1} K. Fr\k{a}czek, M. Lema\'nczyk, \emph{
%On disjointness properties of some smooth flows}, Fund. Math.
%\textbf{185} (2005), 117-142.


\bibitem{Fr-Le.selfs}K.\ Fr\k{a}czek, M.\ Lema\'nczyk,
{\em On the self-similarity problem for flows}, Proc. London Math.
Soc. \textbf{99} (2009), 658-696.

\bibitem{Fu}H.\ Furstenberg, \emph{Disjointness in ergodic theory,
minimal sets and diophantine approximation}, Math.\ Syst.\ Th.\
\textbf{1} (1967), 1-49.

\bibitem{Fu3}H.\ Furstenberg, \emph{Recurrence in Ergodic
Theory and Combinatorial Number Theory}, Princeton
University Press, Princeton, New Jersey, 1981.



\bibitem{Gl}E. Glasner, \emph{Ergodic Theory via Joinings},
Mathematical Surveys and Monographs \textbf{101}, AMS, Providence,
RI, 2003.



\bibitem{Go-Ju-Le-Ru}
G.R.\ Goodson, A.\ del Junco, M.\ Lema\'nczyk, D.\ Rudolph,
\emph{Ergodic transformations conjugate to their inverse by
involutions}, Ergodic Theory Dynam. Systems \textbf{16} (1996),
97-124.

%\bibitem{Go-Le} G.R.\ Goodson,  M.\ Lema\'nczyk, \emph{
%Transformations conjugate to their inverses have even essential
%values}, Proc.\ Amer.\ Math.\ Soc.\ \textbf{124} (1996),
%2703-2710.

\bibitem{Go-He} W.\ Gottschalk, G.\ Hedlund, \emph{Topological dynamics},
American Mathematical Society Colloquium Publications, Vol. 36.
AMS, Providence, R. I., 1955.

\bibitem{He} M. Herman, {\em Sur la conjugaison
diff\'erentiable des diff\'eomorphismes du cercle \`a des
rotations}, Publ. Math. IHES {\bf 49} (1979), 5-234.

\bibitem{Ju}A.\ del Junco, \emph{Disjointness of measure-preserving
transformations, minimal self-joinings and category}, Ergodic theory and
dynamical systems, I (College Park, Md., 1979-80), pp. 81-89,
Progr. Math., 10, Birkhauser, Boston, Mass., 1981.

\bibitem{Ju-Ra-Sw} A.\ del Junco, M.\ Rahe, L.\ Swanson, \emph{Chacon's
automorphism has minimal self joinings}, Journal d'Analyse
Mathematique \textbf{37} (1980), 276-284.

\bibitem{KatMon}A.B.\ Katok, \emph{Monotone equivalence in ergodic theory},
(Russian) Izv. Akad. Nauk SSSR Ser. Mat. \textbf{41} (1977), 104-157.


\bibitem{Ka-Th} A.\ Katok, J.-P.\ Thouvenot, \emph{Spectral
Properties and Combinatorial Constructions in Ergodic Theory},
Handbook of dynamical systems. Vol. 1B, 649--743, Elsevier B. V.,
Amsterdam, 2006.

\bibitem{Ke-Ma-Se} H.\ Keynes, N.\ Markley, M.\ Sears,
\emph{The structure of automorphisms of real suspension flows},
Ergodic Theory Dynam.\ Systems \textbf{11} (1991),  349-364.

\bibitem{Kh} A.Ya.\ Khinchin,  {\em Continued fractions},
The University of Chicago Press, Chicago-London, 1964.

\bibitem{Ko9}A.V.\ Kochergin,
\emph{The homology of functions over dynamical systems}, (Russian)
Dokl.\ Akad.\ Nauk SSSR \textbf{231} (1976), 795-798.

\bibitem{Kw-Le-Ru} J.\ Kwiatkowski, M.\ Lema\'nczyk, D.\ Rudolph,
\emph{A class of real cocycles having an analytic coboundary modification},
Israel J.\ Math.\ \textbf{87} (1994), 337-360.

\bibitem{Ku-Ni} L.\ Kuipers, H.\ Niederreiter, \emph{Uniform
distribution of sequences}, Pure and Applied Mathematics,
Wiley-Interscience, New York-London-Sydney, 1974.

\bibitem{Ku} J.\ Ku\l aga, {\em On the
self-similarity problem for smooth flows on orientable surfaces},
to appear in Ergodic Theory  and Dynamical Systems,
\url{arXiv:1011.6166}.

\bibitem{Le} M. Lema\'nczyk, \emph{Spectral Theory of
Dynamical Systems}, Encyclopedia of Complexity and System Science,
Springer-Verlag (2009), 8554-8575.

\bibitem{Le-Pa-Th}  M. Lema\'nczyk, F. Parreau, J.-P.\ Thouvenot,
{\em Gaussian automorphisms whose ergodic self-joinings are
Gaussian}, Fund.\ Mah.\ {\bf 164} (2000), 253-293.


%\bibitem{Le-Pa-Vo} M.\ Lema\'nczyk, F.\ Parreau, D.\ Voln\'y,
%\emph{Ergodic properties of real cocycles and pseudo--homogeneous
%Banach spaces}, Trans. Amer. Math. Soc. \textbf{348} (1996),
%4919-4938.

\bibitem{Le-Wy}M. Lema\'nczyk, M. Wysoki\'nska, \emph{On analytic
flows on the torus which are disjoint from systems of
probabilistic origin}, Fund.\ Math. \textbf{195} (2007), 97-124.

\bibitem{Ma}S.A.\ Malkin, \emph{An example of two metrically nonisomorphic
ergodic automorphisms with the same simple spectrum}, (Russian)
Izv. Vyssh. Uchebn. Zaved. Matematika  \textbf{6} (1968), 69-74.

\bibitem{MR1719722}
M.G.\ Nadkarni, \emph{Spectral theory of dynamical systems},
Birkhäuser Advanced Texts, Birkhäuser Verlag, Basel, 1998.

\bibitem{Or} D.\ Ornstein, \emph{Ergodic theory, randomness, and ``chaos''},
Science {\bf 243} (1989), 182-187.

\bibitem{MR0272985} D.\ Ornstein,
\emph{Imbedding Bernoulli shifts in flows}, 1970 Contributions to
Ergodic Theory and Probability (Proc.\ Conf., Ohio State Univ.,
Columbus, Ohio, 1970) pp. 178-218 Springer, Berlin.


\bibitem{Or-Ru-We}D.\ Ornstein, D.\ Rudolph, B.\ Weiss,
{\em Equivalence of measure preserving transformations}, Mem.\
Amer.\ Math.\ Soc.\ {\bf 37} (1982), no. 262.
%\bibitem{Or-Sm} D. Ornstein, M. Smorodinsky, \emph{Continuous speed changes for
%flows}, Israel J. Math. \textbf{31} (1978), 161-168.


\bibitem{Os}V.I.\ Oseledets, {\em An example of two nonisomorphic
systems with the same simple singular spectrum}, Functional Anal.\
Appl.\ {\bf 5} (1971), 75-79.

\bibitem{RatLB}M.\ Ratner, \emph{Horocycle flows are loosely Bernoulli},
Israel J. Math. \textbf{31} (1978), 122-132.

\bibitem{Ra-rig}M.\ Ratner, {\em Rigidity of horocycle flows}, Ann. of Math. (2) {\bf 115} (1982), 597-614.


\bibitem{Ra} M. Ratner, \emph{Horocycle flows, joinings and rigidity of products},
Ann. of Math. (2) \textbf{118} (1983), 277-313.


\bibitem{Ry}V.V.\ Ryzhikov, \emph{Joinings, intertwinings operators,
factors, and mixing properties of dynamical systems}, Russian
Acad.\ Sci.\ Izv.\ Math.\ \textbf{42} (1994), 91-114.

\bibitem{Ry1}V.V.\ Ryzhikov, {\em Partial multiple mixing
on subsequences may distinguish between automorphisms $T$ and
$T^{-1}$}, Mathematical Notes {\bf 74} (2003), 889-895.

%\bibitem{ViB} M. Viana, \emph{Dynamics of Interval Exchange Transformations and %Teichm\"uller Flows},
%Lecture notes available from
%\url{http://w3.impa.br/~viana/out/ietf.pdf}


%\bibitem{YoLN}J.-C. Yoccoz \emph{Interval Exchange Maps and Translation Surfaces}, Lecture %notes available from
%\url{http://www.college-de-france.fr/media/equ_dif/UPL15305_PisaLecturesJCY2007.pdf}



\end{thebibliography}
\end{document}